\documentclass[reqno,10pt]{amsart}
\usepackage{amsmath}
\usepackage{amssymb}
\usepackage{amsfonts}
\usepackage{graphicx}
\usepackage[abs]{overpic}
\usepackage{caption}
\usepackage{subcaption}
\usepackage{wrapfig}
\usepackage{float}
\usepackage{graphicx}

\usepackage{tikz}
\usetikzlibrary{decorations.pathreplacing}

\definecolor{blue_links}{RGB}{13,0,180} 

\usepackage{hyperref}
\hypersetup{
    colorlinks=true, 
    linktoc=all,     
    linkcolor=blue_links,  
    citecolor=blue_links,
    urlcolor=blue_links,
}

\usepackage{mathtools}

\newtheorem{theorem}{Theorem}[section]
\newtheorem{lemma}[theorem]{Lemma}
\newtheorem{proposition}[theorem]{Proposition}

\newtheorem{definition}[theorem]{Definition}

\newtheorem{rem}[theorem]{Remark}
\newtheorem{example}[theorem]{Example}

\newcommand{\N}{\mathbb{N}}

\newcommand{\R}{\mathbb{R}}

\newcommand{\ZZZ}{\color{black}} 
\newcommand{\UUU}{\color{black}} 
\newcommand{\EEE}{\color{black}}

\newcommand\ep{\varepsilon}

\DeclareMathOperator*{\aplim}{ap-\lim}
\def\Id{\mathbf{Id}}
\def\id{\mathbf{id}}

\def\eps{\varepsilon}

\def\dist{\operatorname{dist}}

\def\XXint#1#2#3{{\setbox0=\hbox{$#1{#2#3}{\int}$}
     \vcenter{\hbox{$#2#3$}}\kern-.5\wd0}}

\usepackage{color}

\newcommand{\YYY}{\color{black}}

\numberwithin{equation}{section}

\begin{document} 

\title[Non-interpenetration conditions from nonlinear to linearized Griffith fracture]
{Non-interpenetration conditions in the passage from nonlinear to linearized Griffith fracture}
\author[S. Almi]{Stefano Almi}
\address[Stefano Almi]{Department of Mathematics and Applications ``R.~Caccioppoli'', University of Naples Federico II, Via Cintia, Monte S. Angelo, 80126 Napoli, Italy.}
\email{stefano.almi@unina.it}
\author[E. Davoli] {Elisa Davoli} 
\address[Elisa Davoli]{Institute of Analysis and Scientific Computing, TU Wien, 
Wiedner Hauptstra\ss e 8-10, 1040 Vienna, Austria}
\email{elisa.davoli@tuwien.ac.at}
\author[M. Friedrich]{Manuel Friedrich} 
\address[Manuel Friedrich]{Department of Mathematics, Friedrich-Alexander Universit\"at Erlangen-N\"urnberg. Cauerstr.~11,
D-91058 Erlangen, Germany, \& Mathematics M\"{u}nster,  
University of M\"{u}nster, Einsteinstr.~62, D-48149 M\"{u}nster, Germany}
\email{manuel.friedrich@fau.de}

\subjclass[2010]{49J45, 70G75,  74B20, 74G65, 74R10, 74A30}
\keywords{Non-interpenetration, Griffith fracture, linearization, Ciarlet-Ne\v{c}as, contact condition}

\begin{abstract}
 We characterize the passage from nonlinear to linearized Griffith-fracture theories under non-interpenetration constraints. In particular, sequences of deformations satisfying a Ciarlet-Ne\v{c}as condition in $SBV^2$   and for which a convergence of the energies is ensured, are shown to admit asymptotic representations in $GSBD^2$ satisfying a suitable contact condition. With an explicit counterexample, we prove that this result fails if convergence of the energies does not hold.  We further prove that each limiting displacement satisfying the contact condition can be approximated by an energy-convergent sequence of deformations fulfilling a Ciarlet-Ne\v{c}as condition. The proof relies on a piecewise Korn-Poincar\'e inequality in $GSBD^2$, on a careful blow-up analysis around jump points, as well as on a refined $GSBD^2$-density result guaranteeing enhanced contact conditions for the approximants.

\end{abstract}
\maketitle

\section{Introduction}

 A crucial question in materials science is to provide an accurate description of phenomena exhibiting an intrinsic nonlinear nature, as well as to establish the range of validity of their linearized approximations. 
 A key challenge in this direction consists in the mathematical modeling of impenetrability. In this paper we provide an analysis of impenetrability constraints for brittle hyperelastic materials and address the passage from nonlinear to linearized descriptions for Griffith-fracture theories.

To illustrate the main difficulties involved in the mathematics of impenetrability, consider the simplest modeling scenario in which finite strain deformations play a significant role, namely that of nonlinear elasticity. A standard constitutive assumption for large strain frameworks is the requirement that a body should not be allowed to interpenetrate itself during elastic deformations, and that extreme compressions should lead to a blow-up of the elastic energy, therefore being energetically unfavorable. Although the theory of nonlinear elasticity is by now quite classical (see, e.g.,~\cite{ball} for an introduction to the topic), necessary conditions on the stored-energy density guaranteeing existence of minimizers of nonlinearly elastic energy functionals encoding the behavior described above are not yet known, cf.~\cite{ball2} and~\cite{benesova.kruzik}.

The existence of injective energy minimizers for hyperelastic materials was pioneered by {\sc J.~Ball} in \cite{ball3} (see also \cite{sverak} for a regularity analysis of minimizing configurations, and \cite{fonseca.gangbo} for a related local invertibility result). In the subsequent work \cite{ball}, it was pointed out that requiring the positivity of the determinant of the gradient of deformations is neither enough to ensure local injectivity everywhere, nor sufficient to prevent a global loss of injectivity. In \cite{Ciarlet-Necas:1987}, {\sc P.G.~Ciarlet} and {\sc J.~Ne\v{c}as} proposed a condition compatible with the existence theory of minimizers, ensuring frictionless contact, non-selfpenetrability, as well as injectivity almost everywhere when combined with a positivity constraint on the determinant of nonlinear strains. For an open bounded domain $\Omega\subset \R^d$, $d\in \mathbb{N}$, the Ciarlet-Ne\v{c}as condition reads as follows:
\begin{equation}
\label{eq:CN-classic}
\int_{\Omega}{\rm det}\,\nabla y(x)\, {\rm d}x\leq  \mathcal{L}^d(y(\Omega)),
\end{equation}
where $y(\Omega)$ is the deformed set and $\mathcal{L}^d(y(\Omega))$ its $d$-dimensional Lebesgue measure.

 Almost  everywhere injectivity of deformations has been analyzed in~\cite{BouHenMol2020} for limits of Sobolev homeomorphisms, in~\cite{MolVod2020} in  the  presence of distorsion penalizations (see, e.g.,~\cite{HenKos2014, Res1989}), and in~ \cite{healey.kromer} for second-grade non-simple materials (cf.~\cite{Toupin:62,Toupin:64}), whereas a first numerical implementation of the Ciarlet-Ne\v{c}as condition as an energy penalization has been exploited in \cite{mielke.roubicek} in the setting of finite strain elastoplasticity. For completeness, we also mention the numerical analysis of a nonlocal alternative to the Ciarlet-Ne\v{c}as condition for non-simple materials in \cite{kromer.valdman}, as well as \cite{kromer} for a generalization of \cite{ball}. 

We focus here on impenetrability constraints in the setting of brittle hyperelastic bodies, and restrict ourselves to the planar case for simplicity.  Following Griffith's theory of crack propagation \cite{Bourdin-Francfort-Marigo:2008, Francfort-Marigo:1998, griffith}, for a set $\Omega\subset \mathbb{R}^2$, the variational modeling of fracture mechanics relies on the competition between a frame-indifferent bulk energy and a surface term. This in turn rewrites as the minimization of a functional of the form:
\begin{align}\label{rig-eq: Griffith-base}
\mathcal{E}(y) =\int_{ \Omega} W(\nabla y(x)) \,  {\rm d} \EEE x +  \kappa  \mathcal{H}^{1}(J_y),\end{align}
where $W	\colon \mathbb{M}^{2\times 2}\to [0,+\infty)$ is a nonlinear elastic energy density, $\kappa>0$ is a material constant, deformations $y\colon\Omega\to \R^2$ are meant to belong to the class $SBV(\Omega)$ of \emph{special functions of bounded variation} \cite{Ambrosio-Fusco-Pallara:2000}, \EEE  $\nabla y$ denotes the absolutely continuous part of their gradient, $J_y$ is their jump set, and the latter energy-term, $\mathcal{H}^{1}(J_y)$, penalizes the crack length. See also \cite{ambrosio.braides, DeGiorgi-Ambrosio:1988}  for an introduction to the topic.

A generalization of \eqref{eq:CN-classic} in this setting has been introduced and characterized in \cite{Giacomini-Ponsiglione:2008}. In the passage from nonlinear elasticity to large-strain Griffith theories, a first modeling difficulty is related to the fact that deformations do not admit, a priori, continuous representatives, so that the notion of volume of the deformed set in the right-hand side  of~\eqref{eq:CN-classic} is no longer well-posed. In \cite{Giacomini-Ponsiglione:2008}, this difficulty has been solved by means of the weaker notion of measure-theoretic image of the deformed set $[y(\Omega)]$. This, in turn, is defined by considering approximate-differentiability points of admissible deformations, cf.\ Definitions \ref{def: mti} and \ref{def: CN} below for the precise formulations. In the same paper, existence of minimizers of \eqref{rig-eq: Griffith-base} inheriting the Ciarlet-Ne\v{c}as condition is ensured, and alternative formulations of impenetrability are also discussed. In particular, in \cite[Section~6.1]{Giacomini-Ponsiglione:2008} a contact condition of the form 
\begin{equation}
\label{eq:CC-intro}[u](x)\cdot \nu_u(x)\geq 0\text{ for $\mathcal{H}^1$-a.e.\ } x\in J_u\end{equation}
is proposed as a linearized counterpart to \eqref{eq:CN-classic} for displacements $u\colon \Omega\to \R^2$, where $J_u$ denotes the jump set of the displacement, $[u]$ its jump opening, and $\nu_u$ its approximate unit normal. We refer also to \cite[Section 5.1]{Francfort-Marigo:1998}. \EEE
A study of quasistatic crack growth under impenetrability constraints has been carried out in the series of works \cite{DalMaso-Lazzaroni:2010, dalmaso.lazzaroni, lazzaroni}. A variational model including both cavitation and fracture has been analyzed in~\cite{henao.moracorral}. Ambrosio-Tortorelli approximations of brittle fracture models under a non-interpenetration constraint are the subject of~\cite{Chambolle-Conti-Francfort:2018}.

The goal of this paper is to provide a rigorous analysis of the connection between the Griffith-counterpart of \eqref{eq:CN-classic} proposed in \cite{Giacomini-Ponsiglione:2008} and the contact condition in \eqref{eq:CC-intro} by performing a  nonlinear-to-linear passage.  
Before discussing our results, we briefly review the literature on linearization for brittle hyperelastic materials. A simultaneous discrete-to-continuum and nonlinear-to-linear study for general crack geometries and for deformations close to the identity is the subject of \cite{FriedrichSchmidt:2014.2}, whereas a linearization analysis for quasistatic evolution models and under additional assumptions on the admissible cracks has been performed in \cite{NegriToader:2013} (see also \cite{Zanini}).
An effective linearized Griffith energy as $\Gamma$-limit of nonlinear and frame indifferent models in the small strain regime and under no assumptions on the crack has been identified in the planar setting in \cite{Friedrich:15-2}, and recently extended in dimension $d\geq 2$ in \cite{higherordergriffith} for the framework of non-simple materials. We refer to Subsection \ref{s:nonlinear} below for a precise description of this latter result. We only mention here that, since no a priori bounds are assumed on the deformations, the function spaces in which the analysis is developed are those of \emph{generalized special functions of bounded variation}, $GSBV$, and \emph{generalized special functions of bounded deformation}, $GSBD$, cf.\ \cite{Ambrosio-Fusco-Pallara:2000, DalMaso:13}.
The topology in which the linearization in \cite{Friedrich:15-2, higherordergriffith} is performed is that of a \emph{tripling of the variable}, in which to every sequence of deformations with equibounded rescaled energies, one associates a sequence of Caccioppoli partitions, corresponding piecewise rigid motions,  and rescaled displacement fields which are defined separately on each  component of the partitions, see Definition \ref{def:conv}. The limiting displacement field obtained by means of this procedure is referred to as the \emph{asymptotic representation} of the sequence of deformations. 

The starting point of our analysis is the linearization result in \cite{higherordergriffith}. The focus of our study is the asymptotic behavior of higher-order Griffith fracture energies under the $GSBV$-version of the impenetrability constraint in \eqref{eq:CN-classic}. Our contribution is  threefold. Our first result is in the negative, for we give an example that, in the linearization process, sequences of deformations with equibounded nonlinear Griffith energies and satisfying a $GSBV$-formulation of \eqref{eq:CN-classic} might lead  to limiting displacements violating \eqref{eq:CC-intro}.  We further show that, in the absence of additional conditions, a linearized version of \eqref{rig-eq: Griffith-base} under \eqref{eq:CC-intro} is not the variational limit of \eqref{rig-eq: Griffith-base} complemented by \eqref{eq:CN-classic}. In particular, our construction suggests that the linearized counterpart of~\eqref{rig-eq: Griffith-base} contains an additional anisotropic surface term being positive when \eqref{eq:CC-intro} is violated, which depends on the orientation and on the amplitude of the jump of the displacement~$u$. This is shown in Examples~\ref{ex} and~\ref{ex:anisotropic}, and motivates the remaining part of our analysis. 
 
 Our second contribution is to \EEE prove that adding further assumptions on the sequence of deformations under consideration and restricting the analysis to ``energy-convergent sequences" leads in fact to a linearized Griffith model constrained by the contact condition  \eqref{eq:CC-intro}. A simplified version of our result reads as follows, we refer to Theorem \ref{thm: main} and Theorem \ref{thm: CNCC2} for the precise statements. 
\begin{theorem}
\label{thm:meta-main1}
Let $(y_\ep)_\ep$ be a sequence of deformations with equibounded Griffith energies, satisfying \eqref{eq:CN-classic}, and such that their nonlinear energies converge to the linearized energy of their asymptotic representation $u$. Then $u$ satisfies \eqref{eq:CC-intro}.
\end{theorem}

Third, we prove that for each limiting displacement $u$ fulfilling \eqref{eq:CC-intro}, an energy-convergent sequence  satisfying the impenetrability condition \eqref{eq:CN-classic} \EEE can be constructed (see Theorem \ref{th: recovery}).
\begin{theorem}
\label{thm:meta-main2}
Let $u$ satisfy \eqref{eq:CC-intro}. Then, there exists a sequence $(y_\ep)_\ep$ as in Theorem~\ref{thm:meta-main1} having~$u$ as asymptotic representation.
\end{theorem}


The proof of Theorem \ref{thm:meta-main1} is performed by contradiction: we postulate the existence of sets of positive measure where \eqref{eq:CC-intro} is violated and show that this cannot be the case by means of  a careful blow-up argument, cf.~Proposition~\ref{prop: Korn}, which in turn essentially relies on a piecewise Korn-Poincar\'e inequality, see Proposition \ref{th: kornpoin-sharp}. Since this latter result is currently only available in dimension $d=2$, this is the reason why our analysis is restricted to the planar setting.
The main ingredient for establishing Theorem \ref{thm:meta-main2} is a density result in $GSBD$ (see Theorem \ref{th: crismale-density2}) keeping track of boundary data, which in turn provides approximants of the given displacement satisfying a strengthened version of the contact condition in \eqref{eq:CC-intro}.

We remark  once again that our analysis shows the following: In general, a linear Griffith energy under the contact condition~\eqref{eq:CC-intro} does not provide a linearized counterpart to the nonlinear model~\eqref{rig-eq: Griffith-base} under the Ciarlet-Ne\v{c}as constraint~\eqref{eq:CN-classic}, and convergence of minimizers of the nonlinear model to the linearized one is not ensured. In fact,  as mentioned above, the compactness in $GSBD$ (see~\cite{higherordergriffith} and Definition~\ref{def:conv} below) fails to guarantee ~\eqref{eq:CC-intro} in the linearization process, unless there is convergence of the energies (cf.~Examples~\ref{ex} and~\ref{ex:anisotropic}), which cannot \EEE be proven a priori for sequences of minimizers.  Further, motivated by Example ~\ref{ex:anisotropic}, we conjecture that the effect of adding the Ciarlet-Ne\v{c}as condition~\eqref{eq:CN-classic} to the functional~\eqref{rig-eq: Griffith-base} is given by  the presence   of an additional anisotropic surface term  possibly depending on the orientation and on the amplitude of the jump of limiting displacements. The precise characterization of this surface term goes beyond the scope of this work, for it relies on the identification of a suitable cell-formula for the local limiting energy density around jump points. This will be the subject of a forthcoming analysis.

This paper is organized as follows. Section \ref{sec:prel} collects some preliminary results, basic properties of the spaces $GSBV$ and $GSBD$, as well as Proposition \ref{th: kornpoin-sharp} and  Theorem \ref{th: crismale-density2}. Section \ref{s:results} contains the precise formulation of \eqref{eq:CN-classic} and \eqref{eq:CC-intro}, the definition of nonlinear and linearized energy functionals, the description of our notion of convergence,  Examples \ref{ex} and \ref{ex:anisotropic}, and the statement of our main results. Section \ref{sec:blow} is devoted to the blow-up argument, whereas Sections \ref{sec: proof} and \ref{s:th: recovery} tackle the proofs of Theorems \ref{thm:meta-main1} and \ref{thm:meta-main2}.

\section{Preliminaries and notation}
\label{sec:prel}
 In this section, we   introduce the basic notation and define the function spaces we will use throughout the paper.

\subsection{Basic notation} 
 We denote by $\Omega$ an open bounded subset of~$\R^2$ with Lipschitz boundary~$\partial\Omega$. The symbols~$\mathcal{L}^2$ and~$\mathcal{H}^{1}$ represent the Lebesgue and the $1$-dimensional Hausdorff measure \EEE in~$\mathbb{R}^2$, respectively. We set  $\mathbb{S}^{1} \coloneqq \lbrace x \in \R^2: \, |x|=1\rbrace$. \EEE 
 The identity  map \EEE on $\R^2$ is indicated by $\id$ and its   gradient, \EEE the identity matrix, by $\Id \in \R^{2 \times 2}$.  The  spaces \EEE of symmetric and skew symmetric matrices are denoted by $\R^{2 \times 2}_{\rm sym}$ and $\R^{2 \times 2}_{\rm skew}$, respectively. We set ${\rm sym}( \mathrm{F})  \coloneqq  \frac{1}{2}(\mathrm{F}^T + \mathrm{F}) \EEE $ for $ \mathrm{F} \EEE \in \R^{2\times 2}$ and define $SO(2)  \coloneqq  \lbrace \mathrm{R}\in \R^{2 \times 2}: \mathrm{R}^T \mathrm{R} = \Id, \, \det \mathrm{R}=1 \rbrace$. For every $\mathrm{F} \in \R^{2\times 2}$ we denote by $\dist(\mathrm{F}, SO(2))$ the distance of~$\mathrm{F}$ from the set $SO(2)$. 
 
 For an $\mathcal{L}^2$-measurable set $E\subset\mathbb{R}^2$,  the symbol $\chi_E$ denotes its indicator function.  For two sets $A,B \subset \R^2 \EEE$, we define  $A \triangle B  \coloneqq \EEE (A\setminus B) \cup (B \setminus A)$. By~$B_\rho(x) \subset \R^2$ we denote the open ball  with center $x \in \R^2$ and radius~$\rho$. The symbol~$Q_{\rho}$ stands for the paraxial square \EEE   centered in \EEE the origin and   with \EEE side length $\rho$. 

 A mapping~$a$ of the form $a(x)  = \mathrm{A}\, x + b$ for $\mathrm{A} \in \R^{2 \times 2}_{\rm skew}$ and $b \in \R^2$ is called   an \emph{infinitesimal rigid motion}. \EEE In the next sections, we will make use of the following  elementary lemma on  affine mappings in order to control the norm of   infinitesimal \EEE rigid motions. We refer to~\cite[Lemma~3.4]{FM} or \cite[Lemma~4.3]{Conti-Iurlano:15} for similar statements (The proof relies on \EEE the equivalence of norms in finite dimensions).
		
\begin{lemma}\label{lemma: rigid motion}
Let  $x_0 \in \R^2 \EEE $, and let $R, \theta >0$.   Let $a\colon \R^2 \to \R^2$ be affine  and let $E \subset B_R(x_0) \subset \R^2$ with  $\mathcal{L}^2(E) \ge \theta R^2$.  Then, there exists a constant $\bar{c}_\theta \ge 1$ only depending on  $\theta$  such that 
\begin{align*}
 \Vert a\Vert_{L^\infty(B_R(x_0))}  \le \bar{c}_\theta \Vert  a \Vert_{L^\infty(E)}.
\end{align*}
\end{lemma}

We   conclude this subsection with \EEE the basic notation for the slicing technique. For $\xi \in \mathbb{S}^1$, we let
\begin{align}\label{eq: slicing-not}
\Pi^\xi \coloneqq \{w\in \R^2\colon w\cdot \xi=0\}\,,
\end{align}
and for any $w\in \R^2$ and $B\subset \R^2$ we let
\begin{align}\label{eq: slicing-not2}
  B^\xi_w \coloneqq \{t\in \R\colon w+t\xi \in B\},\quad \quad \quad  \pi_\xi(B) = \lbrace w \in \Pi^\xi\colon B^\xi_w \neq \emptyset \rbrace\,.
\end{align}

We will use the abbreviation a.e. to indicate that a property holds almost everywhere, namely outside a set of zero measure.

\subsection{Area formula} We recall   below \EEE the area formula for a.e.-approximately differentiable   maps \EEE and refer to \cite[Chapter 3]{GMS} for a complete treatment. For every measurable set $E\subset \Omega$,   every map $y\colon \Omega\to \R^2\EEE$, and every $z\in \R^2$, \EEE let~$m(y, z, E)$ be the number of preimages   via $y$ of $z$ \EEE in the set~$E$, that is, 
$$m(y,z,E) \coloneqq \#\lbrace x\in E\colon y(x) = z\rbrace\,.$$
Let us assume that $y\colon \Omega \to \R^2$ is a.e.-approximately differentiable in $\Omega$ and let~$\Omega_{d} \subseteq \Omega$ be the set of approximate differentiability of~$y$. Then, the area formula (see e.g.\  \cite[Theorem~1, Section~1.5, Chapter~3]{GMS}) states that for every measurable set $E\subset \Omega$ the function   $z \mapsto m(y,z,E\cap \Omega_d)$ \EEE is measurable and
\begin{align}\label{eq:area}
\int_E |\det \nabla y(x)| \, {\rm d}x=\int_{\R^2}  m(y,z,E\cap \Omega_d) \, {\rm d}z\,. 
\end{align}

\subsection{Sets of finite perimeter}
 For a set of finite perimeter $E$, we denote by $\partial^* E$ its essential boundary and by  $(E)^1$ the points where $E$ has density one, see \cite[Definition 3.60]{Ambrosio-Fusco-Pallara:2000}. A set of finite perimeter $E$ is called \emph{indecomposable} if it cannot be written as   $E_\alpha \cup E_\beta$ with $E_\alpha \cap E_\beta = \emptyset$, $\mathcal{L}^{ 2 \EEE}(E_\alpha), \mathcal{L}^{ 2 \EEE}(E_\beta) >0$, and $\mathcal{H}^{ 1 \EEE}(\partial^* E) = \mathcal{H}^{ 1 \EEE}(\partial^* E_\alpha) + \mathcal{H}^{ 1 \EEE}(\partial^* E_\beta)$. \EEE  Note  that this notion generalizes the concept of  connectedness to sets of finite perimeter. By \cite[Theorem 1]{Ambrosio-Morel}  for each set of finite perimeter $E$  there exists a unique finite or countable family of pairwise disjoint indecomposable sets $( E_i)_i$ such that   $E=\bigcup_{i}E_i$, $\mathcal{L}^2(E_i)>0$ for every $i$, and \EEE $\mathcal{H}^{ 1 \EEE}(\partial^* E) = \sum_i \mathcal{H}^{ 1 \EEE}(\partial^* E_i)$. The sets~$(E_i)_i$ are called the \emph{connected components} of $E$. We call $E$ \emph{simple} if both $E$ and $\R^2\setminus E$ are indecomposable.  For an indecomposable set $E$ we define the  \emph{saturation} ${\rm sat}(E)$ of $E$ as the union of $E$ and its `holes', i.e., the connected components of $\R^2 \setminus E$ with finite  measure, see \cite[Definition~2]{Ambrosio-Morel}. In a similar fashion, for general sets of finite perimeter $E$ with connected components  $ ( E_i)_i \EEE$, we define ${\rm sat}(E) = \bigcup_i {\rm sat}(E_i)$.

We also recall the  structure theorem of the boundary of planar sets   $E$ \EEE of finite perimeter in  \cite[Corollary 1]{Ambrosio-Morel}: there exists a unique countable decomposition  of $\partial^* E$ into pairwise disjoint rectifiable Jordan curves. We \EEE say that $\Gamma \subset \R^2$ is  a rectifiable Jordan curve if $\Gamma = \gamma([a,b])$ for some $a < b$ and some  Lipschitz continuous map $\gamma$, one-to-one on $[a,b)$, and such that  $\gamma(a) = \gamma(b)­$.  According to the  Jordan curve  theorem, any Jordan curve splits $\R^2 \setminus \Gamma$ into exactly one bounded and one unbounded component, denoted by ${\rm int}(\Gamma)$ and ${\rm ext}(\Gamma)$, respectively, where ${\rm int}(\Gamma)$ denotes the bounded component. 

For the definition and properties of Caccioppoli partitions we refer to \cite[Section~4.4]{Ambrosio-Fusco-Pallara:2000}.

\subsection{Function spaces}\label{s:function-spaces}

 We use   the \EEE standard notation $GSBV(\Omega; \R^{2})$ for the space of {\em generalized special functions of bounded variation}, see \cite[Section 4]{Ambrosio-Fusco-Pallara:2000} and \cite[Section 2]{DalMaso-Francfort-Toader:2005}. We recall that a function~$y \in GSBV(\Omega; \R^{2})$ admits an approximate gradient~$\nabla y$ a.e.~in~$\Omega$. We denote by~$J_{y}$ the set of approximate jump   points \EEE of~$y \in GSBV(\Omega;\R^{2})$, that is, the set of   points \EEE $x \in \Omega$   for which \EEE there exist $\nu \in \mathbb{S}^{1}$ and $a, b \in \R^{2}$ such that $a \neq b$ and
\begin{equation}\label{e:GSBV-jump}
\aplim_{\substack{z \to x\\ (z - x) \cdot \nu >0}} \, y(z) = a \qquad \text{and} \qquad \aplim_{\substack{z \to x\\ (z - x) \cdot \nu <0}} \, y(z) = b\,,
\end{equation}
where the symbol $\aplim$ denotes the approximate limit. We recall that~$J_{y}$ is an $\mathcal{H}^{1}$-rectifiable set, \EEE and that the triple~$(a, b, \nu)$ is uniquely defined, up to a permutation of $a$ and~$b$ and a change of sign of~$\nu$. In particular,~$\nu$ is the approximate unit normal to~$J_{y}$ and we denote it by~$\nu_{y}$ from now on. The approximate limits~$a$ and~$b$ at $x \in J_{y}$ are indicated by~$y^{+}_{x}$ and~$y^{-}_{x}$.

We further set
\begin{align}\label{eq: space2}
GSBV^2(\Omega;\R^2) = \lbrace y \in GSBV(\Omega;\R^2): \ \nabla y \in L^2(\Omega;\R^{2\times  2}), \ \mathcal{H}^{1}(J_y) < + \infty \rbrace.
\end{align}
 We define the space
\begin{align}\label{eq: space}
GSBV^2_2(\Omega;\R^2) := \big\{ y \in GSBV^2(\Omega; \R^2): \ \nabla y \in GSBV^2(\Omega;\R^{2\times 2})\big\}.
\end{align}
The approximate differential and the jump set of $\nabla y$ will be denoted by $\nabla^2 y$ and $J_{\nabla y}$, respectively.  (To avoid confusion, we point out that in~\cite{DalMaso-Francfort-Toader:2005} the notation $GSBV^2_2(\Omega;\R^2)$ was used for $GSBV^2(\Omega;\R^2) \cap L^2(\Omega;\R^2)$.)  

We notice that spaces similar to~\eqref{eq: space} already appeared, for instance, in~\cite{Carriero1, Carriero2} to treat second order free discontinuity functionals, e.g., a weak formulation of the Blake $\&$ Zissermann model~\cite{BZ} of image segmentation. Since a function in~$GSBV^{2}_{2} (\Omega; \R^{2})$ is allowed to exhibit discontinuities, our analysis is outside of the framework of the space of special functions with bounded Hessian~$SBH(\Omega)$, considered for second order energies for elastic-perfectly plastic plates (see, e.g.,~\cite{Carriero3}). \EEE

In order to treat linear models of fracture, we need the space $GSBD(\Omega)$ of {\em generalized special functions of bounded deformation}, introduced in~\cite{DalMaso:13}. We recall that a function~$u \in GSBD(\Omega)$ admits an approximate symmetric gradient $e(u) \in L^{1}(\Omega; \R^{2 \times 2}_{\rm sym})$ and its jump set $J_{u}$, defined similarly to~\eqref{e:GSBV-jump}, is $\mathcal{H}^{1}$-rectifiable, so that the approximate unit normal~$\nu_{u}$ to~$J_{u}$ is defined $\mathcal{H}^{1}$-a.e.~on~$J_{u}$ together with the approximate limits $u^{+}_{x}$ and~$u^{-}_{x}$, $x \in J_{u}$. As usual, we set $[u] \coloneqq u^{+} - u^{-}$ as the jump of~$u$ through~$J_{u}$. We further let
\begin{displaymath}
GSBD^{2}(\Omega) \coloneqq \{ u \in GSBD(\Omega) \colon \, e(u) \in L^{2}(\Omega; \R^{2 \times 2}_{\rm sym}) , \, \mathcal{H}^{1}(J_{u}) <+\infty\}\,.
\end{displaymath}

We conclude this section by recalling two technical results concerning   $GSBD^2$\EEE-functions. The first one is a  \emph{piecewise Korn-Poincar\'e inequality} \cite{Solombrino}.

\begin{proposition}[Piecewise Korn-Poincar\'e inequality]\label{th: kornpoin-sharp}
Let $Q \subset \R^2$ be an open square and  let  $0 < \theta \le \theta_0$ for some $\theta_0$ sufficiently small. Then, there   exists \EEE some $C_{\theta}=C_{\theta}(\theta) \ge 1$   such that the following holds: for  each $u \in GSBD^2(Q)$ we find a (finite) Caccioppoli partition $Q = R \cup \bigcup^{J}_{j=1} P_j$, and corresponding rigid motions $(a_j)_{j=1}^J$ such that  
\begin{subequations}
\label{eq: kornpoinsharp2XXX}
\begin{align}
& \sum_{j=1}^{J}\mathcal{H}^1\big( (\partial^* P_j \cap Q) \setminus J_u \big) +\mathcal{H}^1\big( (\partial^* R \cap Q )\setminus J_u  \big) \le \theta (\mathcal{H}^1(J_u) + \mathcal{H}^1(\partial Q)),\label{eq: kornpoinsharp2XXX-1}\\
 & \vphantom{\sum_{i=1}^{J}} \mathcal{L}^2(R)   \le \theta (\mathcal{H}^1(J_u)+ \mathcal{H}^1(\partial Q))^2, \ \ \ \ \ \  \mathcal{L}^2(P_j) \ge \mathcal{L}^2(Q)\theta^3  \ \ \ \text{for all $j=1,\ldots,J$},     \label{eq: kornpoinsharp2XXX-2}\\
& \vphantom{\sum_{i=1}^{J}} \Vert u - a_j \Vert_{L^\infty(P_j)}  \le C_{\theta} \Vert  e(u) \Vert_{L^2( Q \EEE )} \ \ \ \quad  \text{for all $j=1,\ldots,J$}.\label{eq: kornpoinsharp2XXX-3}
\end{align}
\end{subequations}
\end{proposition}

\begin{proof}
The statement is a slightly simplified version of \cite[Theorem 4.1]{Solombrino}.  We briefly explain how the result  can be obtained therefrom. We first suppose that $Q$ is the unit square. \EEE   We define $\theta_0 \le 1/c$, where $c$ is the constant from  \cite[Theorem 4.1]{Solombrino} and apply \cite[Theorem 4.1]{Solombrino} for $\theta/c$ in place of $\theta$.   Then, \eqref{eq: kornpoinsharp2XXX-1} follows from \cite[(18)(i)]{Solombrino}, where we denote the  component $P_0$ by $R$. Item \eqref{eq: kornpoinsharp2XXX-2} follows from \cite[(17)(i), (18)(ii)]{Solombrino}, choosing $\theta_0$ sufficiently small such that $C_\Omega \ge  \theta_0$. Finally, \eqref{eq: kornpoinsharp2XXX-3} follows from  \cite[(18)(iii)]{Solombrino}, where also a corresponding Korn-type estimate has been proved.  Eventually, if $Q$ is not the  unit \EEE square, the result follows by a standard rescaling argument, see \cite[Remark 4.2]{Solombrino}. 
\end{proof}

In the next sections (see in particular Theorem~\ref{th: recovery}) we will also deal with boundary conditions. As usual in $BV$ and $BD$-like spaces, we impose a Dirichlet boundary condition by forcing a displacement $u \in GSBD(\Omega)$ to take a prescribed value on the set $\Omega' \setminus \overline{\Omega}$, where $\Omega'$ is an open bounded subset of~$ \R^{2}$ with Lipschitz boundary~$\partial\Omega'$ such that $\Omega \subseteq \Omega'$. Precisely, for a boundary   datum  $h \in W^{2,\infty}(\Omega';\R^2)$ we introduce the space
\begin{align}\label{eq: boundary-space}
GSBD^2_h (\Omega') & \coloneqq  \lbrace u \in GSBD^2(\Omega')\colon \, u = h \text{ on } \Omega' \setminus \overline{\Omega} \rbrace\,.
\end{align}

In what follows, we will make   a \EEE geometrical assumption on the Dirichlet part of the boundary $\partial_{D} \Omega \coloneqq \Omega' \cap \partial\Omega$, which will allow us to exploit a density result in $GSBD^{2}_{h}(\Omega')$ (see~Theorem~\ref{th: crismale-density2}). Precisely, we assume that there exists a decomposition $\partial \Omega = \partial_D \Omega \cup \partial_N\Omega \cup N$ with 
\begin{align}\label{eq: density-condition2}
\partial_D \Omega, \partial_N\Omega \text{ relatively open}, \ \  \  \mathcal{H}^{d-1}(N) = 0, \ \ \  \partial_D\Omega \cap \partial_N \Omega = \emptyset, \ \ \  \partial (\partial_D \Omega) = \partial (\partial_N \Omega),
\end{align}
and there exist $\bar{\delta}>0$ small and $x_0 \in\R^d$ such that for all $\delta \in (0,\bar{\delta})$ there  holds  
\begin{align}\label{eq: density-condition}
O_{\delta,x_0}(\partial_D \Omega  ) \subset \Omega,
\end{align}
where $O_{\delta,x_0}(x) := x_0 + (1-\delta)(x-x_0)$. 

We conclude this section with the statement of a density result in $GSBD^{2}_{h}(\Omega')$. To   shorten \EEE the notation, we introduce the space $\mathcal{W}(\Omega;\R^2)$ of all functions $u \in SBV(\Omega;\R^2)$ such that $J_u$ is a finite union of disjoint segments and $u \in W^{k,\infty}(\Omega \setminus J_u; \R^2)$ for every $k \in \N$. The following   theorem \EEE is essentially  a consequence of results in  \cite{Crismale2} and   \cite{Cortesani-Toader:1999}. The exact statement can be found in \cite[Theorem 3.6]{higherordergriffith}.

\EEE 

\begin{theorem}[Density with boundary data]\label{th: crismale-density2}
Let $\Omega \subset \Omega' \subset  \R^2$  be bounded Lipschitz domains satisfying \eqref{eq: density-condition2}--\eqref{eq: density-condition}.  Let $h \in W^{r,\infty}(\Omega')$ for $r \in \N$  and let $u \in GSBD^2_h(\Omega')$. Then, \EEE there exists a sequence of functions $(u_n)_n$ in $SBV^2(\Omega; \R^2)$, a sequence of neighborhoods $(U_n)_n$ of $\Omega' \setminus \Omega$, and a sequence of neighborhoods $(\Omega_n)_n$ of $\Omega \setminus U_n$  such that $U_{n} \subset \Omega'$, $\Omega_{n} \subset \Omega$,  $u_n =  h \EEE$ on $\Omega' \setminus \overline{\Omega}$, $u_n|_{U_n} \in W^{r,\infty}(U_n;\R^2)$, and $u_n|_{\Omega_n} \in \mathcal{W}(\Omega_n;\R^2)$, and  the following properties hold:
\begin{subequations}
\label{eq: dense-boundary}
\begin{align}
&   u_n \to  u  \text{ in measure on } \Omega', \label{eq: dense-boundary-1}\\
&  \lim_{n \to \infty} \, \Vert e(u_n) - e(u) \Vert_{L^2(\Omega')} = 0, \label{eq: dense-boundary-2}\\
&    \lim_{n \to \infty} \, \mathcal{H}^{1}(J_{u_n}) = \mathcal{H}^{1}(J_u).\EEE  \label{eq: dense-boundary-3}
\end{align}
\end{subequations}
In particular, $u_n \in W^{r,\infty}(\Omega \setminus J_{u_n};\R^2)$   for every $n\in \N$. \EEE
\end{theorem}  

\EEE

\section{Setting and main results}\label{s:results}

\subsection{Ciarlet-Ne\v{c}as and contact conditions\EEE}

 The  aim of this paper is to characterize \EEE the relation between Ciarlet-Ne\v{c}as and contact conditions in the passage from nonlinear to linearized brittle fracture. 
Following~\cite[Section~2]{Giacomini-Ponsiglione:2008}, we   first \EEE give a precise meaning to the Ciarlet-Ne\v{c}as non-interpenetration condition,   see~\cite{Ciarlet-Necas:1987}. \EEE To do this, we recall the definition of measure theoretic image of $GSBV$-functions. 

\begin{definition}[Measure theoretical image]\label{def: mti}
{\normalfont
Let $y \in GSBV(\Omega;\R^2)$ and let $ \Omega_{d}  \subseteq \EEE \Omega$ be the set of points where  $y$ \EEE is approximate differentiable. We define  $y_d$  by
$$ y_d (x)  \coloneqq  \begin{cases}  \tilde{y}(x)       & \text{ for } x \in  \Omega_d, \\  0 & \text{ else,}  \end{cases} $$
where $\tilde{y}(x)$ denotes the Lebesgue value of $y$ at $x \in  \Omega_d$. \EEE Given a measurable set $E  \subseteq \EEE \Omega$, we say that  $y_d(E)$ \EEE is the measure theoretic image  of $E$ under the map  $y$, \EEE and we denote it by $[y(E)]$.\\
}
\end{definition}

\begin{definition}[Ciarlet-Ne\v{c}as \EEE non-interpenetration condition for $GSBV$-maps]
\label{def: CN}
{\normalfont
We say that $y \in GSBV(\Omega;\R^2)$ satisfies  the  {\em  Ciarlet-Ne\v{c}as \EEE  non-interpenetration  condition}  if $\det \nabla y(x)>0$ for a.e.\ $x\in\Omega$ and
\begin{equation}
\label{e:CN}
\tag{{\rm CN}}
 \int_\Omega \det \nabla y \, {\rm d}x  \le \mathcal{L}^{ 2 } ([y(\Omega)])\,,
 \end{equation}
where $[y(\Omega)]$ denotes the image of $\Omega$ under $y$ according to Definition \ref{def: mti}.
}
\end{definition}

\begin{rem}
{\normalfont
We recall that~\eqref{e:CN} is equivalent to a.e.~injectivity under the assumption~$\det \nabla{y} (x) > 0$ for a.e.~$x \in \Omega$, see~\cite[Proposition~2.5]{Giacomini-Ponsiglione:2008}. Here, we say that $y$ is a.e.-injective if for every representative $\bar{y}$ of $y$  there exists an $\mathcal{L}^2$-negligible set $E \subset \Omega$ such that the restriction of $\bar{y}$ on $\Omega \setminus E$ is injective.
}
\end{rem}
\EEE

We now define the linearized contact condition for functions in $GSBD^{2}(\Omega)$.
\EEE

\begin{definition}[Contact condition for $GSBD$-maps]
{\normalfont
We say that $u \in GSBD^2(\Omega)$ satisfies the {\em contact condition} if  

\begin{equation}
\label{eq: CC}
\tag{{\rm CC}}
\nu_u(x) \cdot [u](x) \ge 0  \quad \text{for $\mathcal{H}^1$-a.e.\ $x \in J_u$}.
\end{equation}
\EEE
}
\end{definition}

Our first observation is the following: consider \EEE a sequence of deformations $(y_\ep)_\ep\subset GSBV^2(\Omega;\R^2)$ satisfying \eqref{e:CN} such that \EEE  their associated rescaled displacements $(u_\ep)_\ep$, defined as
\begin{align}\label{eq: rescalidipl}
u_\eps \coloneqq \frac{1}{\eps} (y_\eps - \id)\,, 
\end{align} 
have uniformly bounded linearized energies, i.e., 
\begin{equation}
\label{eq:unif-en}
\sup_{\eps >0}  \, \mathcal{F}(u_\eps)  <+\infty, \quad \quad \text{ where } \mathcal{F}(u) :=  \Vert e(u_\eps) \Vert^2_{L^2(\Omega)} + \mathcal{H}^{1}(J_{u_\eps})\,.
\end{equation} \EEE
\YYY This \EEE is not enough to guarantee measure convergence of $u_\eps$ \EEE to a displacement $u$ satisfying \eqref{eq: CC}.\EEE

\begin{example}\label{ex}
{\normalfont
Let $\Omega = (-1,1)^2$ and let $u =  (-1,0)\chi_{\lbrace x_1 >0\rbrace}$. Clearly, we have $J_u = \lbrace 0 \rbrace \times (-1,1)$, $\nu_u = e_1$, and $[u] = -e_1$. Therefore, we have $[u] \cdot \nu_u  = -1 <0$ on $J_u$. We now construct a sequence $(y_\eps)_\eps \subset GSBV^2(\Omega;\R^2)$ satisfying  \eqref{eq:unif-en},   as well as \eqref{e:CN}, \EEE  and \EEE such that $u_\eps \to u$ in measure on $\Omega$, where   $u_\eps \in GSBV^2(\Omega;\R^2)$ is defined in \eqref{eq: rescalidipl}. To this end, we let  (see also Figure~\ref{f:1})
\begin{equation}
\label{e:def-example-1}
u_\eps :=    (-1,0)\chi_{\lbrace x_1 >0\rbrace }     +  \left(\frac{2}{\eps},0 \right) \chi_{\lbrace-{2\eps}< x_1 <0\rbrace}\,, \qquad y_\eps = \id + \eps u_\eps\,.
\end{equation}
Then, we see that $\nabla y_\eps  = \Id$ on $\Omega$ and $\mathcal{H}^1(J_{y_\eps}) = 4$, i.e.,  \eqref{eq:unif-en} holds true. \EEE It is also easy to check that $u_\eps \to u$ in measure on $\Omega$.  Finally, the functions $y_\eps$ satisfy the   Ciarlet-Ne\v{c}as  \EEE non-interpenetration  condition,  since for $\eps$ sufficiently small the three sets \EEE
$$[y_\eps( \lbrace x_1 < - {2\eps}\rbrace )], \quad  [y_\eps( \lbrace  -{2\eps} < x_1 < 0 \rbrace   )], \quad [y_\eps( \lbrace x_1 >0\rbrace )]
$$ 
are pairwise disjoint.

\medskip
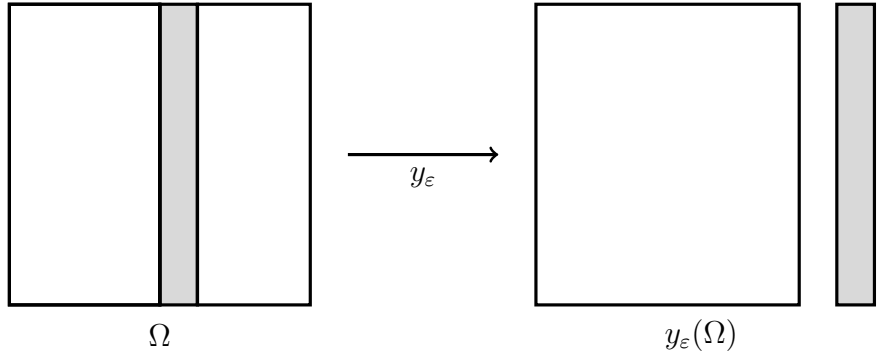
\begin{figure}[h!]
\begin{tikzpicture}
\draw[black, very thick] (0,0) rectangle (2,4);
\draw[very thick](0,0) rectangle (2,4);
\filldraw[fill= gray, very  thick, fill opacity=0.3] (2,0) rectangle (2.5,4);
\draw[very thick] (2.5,0) rectangle (4,4);
\node  at (2, -0.4) {\Large{$\Omega$}};
\draw[very thick] (7,0) rectangle (10.5,4);
\draw[black, very thick] (2, 0) -- (2, 4);
\draw[black, very thick] (2.5, 0) -- (2.5, 4);
\filldraw[fill = gray, very thick, fill opacity=0.3] (11, 0) rectangle (11.5, 4);
\draw[black, ->, very thick] (4.5, 2) -- (6.5, 2);
\node at (9.2, -0.4) {\Large{$y_{\eps} (\Omega)$}};
\node at (5.5, 1.7) {\Large{$y_{\eps}$}};
\end{tikzpicture}
\caption{Graphic representation of the deformation~$y_{\eps}$ in~\eqref{e:def-example-1}.} \label{f:1}
\end{figure}
}
\end{example}

A crucial point in the example is that the length of the jump along the sequence has twice the size of the limiting jump. Our second result shows that, under a suitable energy convergence of the rescaled displacements and  a slightly  stronger control on elastic energies, the pathological situation in Example \ref{ex} can be avoided.

\begin{theorem}[From Ciarlet-Ne\v{c}as to contact condition]
\label{thm: main}
Let $\Omega \subseteq \R^2$ be open and bounded. \EEE  Let $(y_\eps)_\eps \subset GSBV^2(\Omega;\R^2)$ be a sequence satisfying~\eqref{e:CN}. For every $\ep>0$, let  $u_\eps$ be defined as in \eqref{eq: rescalidipl}, \EEE  and assume that there exists $u \in GSBD^2(\Omega)$ such that $u_\eps \to u$ in measure on~$\Omega$. Suppose moreover that  there exists $\gamma > \tfrac{1}{2}$ such that
\begin{subequations}
\label{eq: assu-u}
\begin{align}
&   \sup_{\eps>0} \,  \eps^{1-\gamma}\Vert  \nabla u_\eps \Vert_{L^2(\Omega)} <+\infty,\label{eq: assu-u-1}\\
&   \lim_{\eps \to 0} \, \Vert e(u_\eps) \Vert^2_{L^2(\Omega)} + \mathcal{H}^{1}(J_{u_\eps}) = \Vert e(u) \Vert^2_{L^2(\Omega)} + \mathcal{H}^{1}(J_{u}). \label{eq: assu-u-2}
\end{align}
\end{subequations}
Then, $u$ satisfies  \eqref{eq: CC}.  
\end{theorem}

Let us comment on the hypotheses of Theorem~\ref{thm: main}. By a compactness argument, see Proposition~\ref{prop: comp} and \eqref{eq: first conditions-3} below, assumption~\eqref{eq: assu-u-1} holds for sequences~$(u_{\eps})_\ep$ \EEE such that the corresponding deformation fields \EEE $y_{\eps} = \id + \eps u_{\eps}$ have  bounded nonlinear Griffith energy $\mathcal{E}_{\eps}$,  defined in~\eqref{rig-eq: Griffith en} below. \EEE Condition~\eqref{eq: assu-u-2} is instead stronger. In variational terms, it requires the rescaled displacements  $(u_{\eps})_{\eps}$ associated  to $(y_{\eps})_{\eps}$   \EEE to be an {\em energy-convergent sequence}  for the limiting displacement~$u$, in terms of the energy $\mathcal{F}$ defined in \eqref{eq:unif-en}. \EEE As shown in Example~\ref{ex}, condition~\eqref{eq: assu-u-2} cannot be weakened to a more traditional energy bound of the form \eqref{eq:unif-en}. \EEE Hence, Theorem~\ref{thm: main} states that   \eqref{e:CN} yields the contact condition \eqref{eq: CC}, provided  $(u_{\eps})_\ep$ is an energy-convergent sequence for $\mathcal{F}$. \EEE
\EEE

 We defer the proof of  Theorem~\ref{thm: main}  to  Section~\ref{sec: proof} and continue here with the presentation of our results. 
 The second  part of Section~\ref{s:results} is devoted to the definitions of linear and nonlinear Griffith energies and to the passage from nonlinear to linear models, under the additional Ciarlet-Ne\v{c}as and contact conditions.  

\subsection{From nonlinear to linear Griffith models with non-interpenetration} \label{s:nonlinear}

 We start by introducing the nonlinear Griffith energy for non-simple materials. We let $W \colon \R^{2 \times 2} \to [0,+\infty)$ be a  single well, frame indifferent stored energy   density. To be precise, \EEE we suppose  that there exists $c>0$ such that
\begin{subequations}
\label{assumptions-W}
\begin{align}
&  W \text{ continuous and $C^3$ in a neighborhood of $SO(2)$},\label{assumptions-W-1}\\
 &   W(\mathrm{RF}) = W(\mathrm{F}) \text{ for all } \mathrm{F} \in \R^{2 \times 2}, \mathrm{R} \in SO(2),\label{assumptions-W-2}\\
 &  W(\mathrm{F}) \ge c\dist^2(\mathrm{F},SO(2)) \ \text{ for all $\mathrm{F} \in \R^{2 \times 2}$}, \  W(\mathrm{F}) = 0 \text{ iff } \mathrm{F} \in SO(2). \label{assumptions-W-3}
\end{align}
\end{subequations}

 Let us fix $\kappa>0$,  $\beta \in (\frac{2}{3},1)$, and two open bounded subsets $\Omega \subseteq \Omega'$ of~$\R^{2}$ with Lipschitz boundaries $\partial\Omega$ and $\partial\Omega'$, respectively, such that~\eqref{eq: density-condition2}--\eqref{eq: density-condition} hold. Recalling the definition~\eqref{eq: space} of the space $GSBV^{2}_{2}(\Omega'; \R^{2})$, \EEE for  $\eps >0$ we define the energy $\mathcal{E}_\eps \colon GSBV^2_2( \Omega' ; \R^2 \EEE) \to [0,+\infty]$ by
\begin{align}\label{rig-eq: Griffith en}
\mathcal{E}_\eps(y) =  \begin{cases}
\displaystyle  \eps^{-2}\int_{ \Omega' \EEE} W(\nabla y(x)) \,dx +\eps^{-2\beta} \int_{ \Omega' \EEE} |\nabla^2 y(x)|^2 \, dx  +  \kappa  \mathcal{H}^{1}(J_y) \EEE  & \text{ if $J_{\nabla y}  \subseteq \EEE J_y$,} \\ 
\displaystyle \vphantom{\int_{\Omega'}} + \infty & \text{ else.} \end{cases}
\end{align}
Here and in the following, the inclusion  $J_{\nabla y}  \subseteq \EEE J_y$ has to be understood up to an  $\mathcal{H}^{1}$\EEE-negligible set.  Since $W$ grows quadratically around  $SO(2)$, \EEE the parameter $\eps$ corresponds to the typical scaling of strains for configurations with finite energy.  We further notice that the choice of two open sets $\Omega$ and $\Omega'$ is due to the fact that we are interested in boundary value problems, where a Dirichlet datum is to be imposed on~$\Omega' \setminus \overline{\Omega}$ (see also Section~\ref{s:function-spaces}). \EEE

 Due to the presence of the second term  in \eqref{rig-eq: Griffith en}, we deal with a Griffith-type model for \emph{nonsimple materials}. Elastic energies  depending  on the second gradient of the deformation were introduced by {\sc Toupin} \cite{Toupin:62,Toupin:64} to  enhance compactness and rigidity properties. In our context, we consider a second gradient term (describing the absolutely continuous part of the gradient of $\nabla y$) \EEE for a material undergoing fracture, which has a regularization effect on the entire intact region $\Omega' \setminus  J_y $ of the material. This is modeled by the condition $J_{\nabla y}  \subseteq \EEE J_y$. 

 We finally remark that the condition $J_{\nabla y} \subseteq J_y$ in~\eqref{rig-eq: Griffith en} is not closed under convergence in measure on~$\Omega'$, and to guarantee the existence of minimizers one needs to pass to a suitable relaxation, see \cite[Proposition 2.1 and Theorem 2.2]{higherordergriffith}.\EEE

 The corresponding linearized Griffith model is represented by the functional $\mathcal{E}\colon GSBD^2(\Omega') \to [0,+\infty)$   given by \EEE
\begin{align}\label{rig-eq: Griffith en-lim}
\mathcal{E}(u) \coloneqq  \int_{\Omega'} \frac{1}{2} Q(e(u))\,{\mathrm dx}  + \kappa\mathcal{H}^{1}(J_u),
\end{align}
where  $Q\colon \R^{2 \times 2} \to [0,+\infty)$  is the quadratic form $Q(\mathrm{F}) = D^2W(\Id)\mathrm{F} : \mathrm{F}$ for all $\mathrm{F} \in \R^{2 \times 2}$. In view of  \eqref{assumptions-W},  $Q$ is positive definite on $\R^{2 \times 2}_{\rm sym}$ and vanishes on $\R^{2 \times 2 }_{\rm skew}$.  \EEE

The $\Gamma$-convergence of~$\mathcal{E}_{\eps}$ to~$\mathcal{E}$ has been studied in~\cite{higherordergriffith} in dimension~$d \geq 2$, with neither non-interpenetration nor contact conditions. We also refer to~\cite{Friedrich:15-2} for a linearization result in dimension~$d=2$ without second order regularization in~\eqref{rig-eq: Griffith en}.  \EEE
%
%
%
%

Justified by~\cite[Section 6.1]{Giacomini-Ponsiglione:2008},~\cite[Appendix A]{DalMaso-Lazzaroni:2010}, or \cite[Section 5.1]{Francfort-Marigo:1998}, \EEE a natural conjecture would be that, in this limiting passage, the conditions  \eqref{e:CN} and   \eqref{eq: CC} could simply be included on the nonlinear and linear level, respectively. Example~\ref{ex} showed that this is not the case, as \eqref{eq: CC} is not maintained for limits of sequences satisfying \eqref{e:CN}. In the next example we further show that the variational limit of the functionals~$\mathcal{E}_{\eps}$ cannot  expected to be expressed by means of the classical Griffith energy.

\begin{example}
\label{ex:anisotropic}
{\normalfont
Let $\Omega= (-1, 1)^{2}$, let $\mu = (\mu_{1}, \mu_{2}) \in \R^{2}$ be such that $\mu_{1} <0$, and let $u= (\frac{\mu_{1}}{2}, \mu_{2} )  \chi_{\lbrace x_1 >0\rbrace}$. To fix the ideas, we also assume $\mu_{2}<0$. As in Example~\ref{ex}, $J_{u} = \{0\} \times (-1, 1)$ has length $\mathcal{H}^{1}(J_{u}) = 2$ and normal vector~$\nu_{u} = e_{1}$. Hence, $[u] \cdot e_{1} = \frac{\mu_{1}}{2} <0$ on~$J_{u}$. For $\varepsilon>0$ we set $n_{\eps} := \lfloor \frac{1}{\varepsilon | \mu_{2}|} \rfloor$,  where $\lfloor \cdot \rfloor$ denotes the integer \EEE part, and let 
\begin{align*}
R_{\eps}^{k} & := \Big(- \eps\frac{|\mu_{1}|}{2}, \eps \frac{|\mu_{1}|}{2} \Big) \times \big(-1 + 2k \eps |\mu_{2}| , -1 + (2k+1) \eps | \mu_{2} | \big)\qquad  \text{for $k =0, \ldots, n_{\eps}-1$},\\
R_{\eps} & := \bigcup_{k=0}^{n_{\eps}- 1} R^{k}_{\eps}\,.
\end{align*}
Then, we define $y_{\eps} \in GSBV^2(\Omega;\R^2)$ as (see also Figure~\ref{f:2})
\begin{equation}
\label{e:def-example-2}
y_{\eps} :=  \id + \eps u_\eps \qquad \text{with} \qquad u_{\eps} :=\left( \frac{\mu_{1}}{2} ,  \mu_{2} \right) \chi_{ (\Omega \setminus R_{\eps})\cap\{x_1>0\}} + \left(\frac{2}{\eps}, 0 \right) \chi_{R_{\eps}}\,,
\end{equation}
so that~$(y_{\eps})_\ep$ satisfies \eqref{eq:unif-en}, and \EEE \eqref{e:CN}, as in Example~\ref{ex}, and  $u_{\eps} \to u$ in measure. \EEE

\medskip
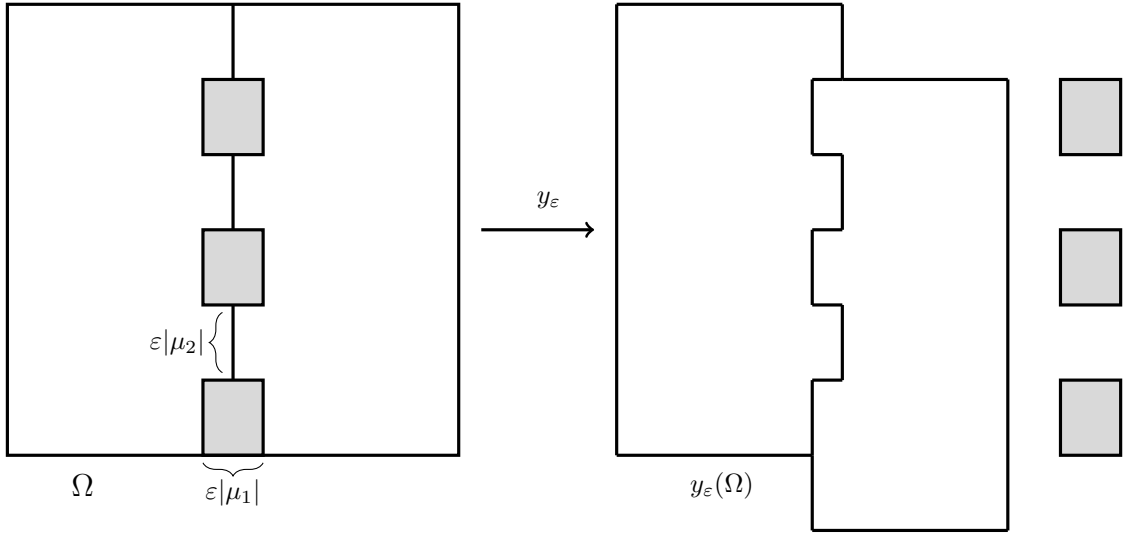
\begin{figure}[h!]
\begin{tikzpicture}
\draw[black, very thick] (-2,-2) rectangle (4,4);
\filldraw[fill= gray, very  thick, fill opacity=0.3] (0.6,-2) rectangle (1.4,-1);
\draw[black, very thick] (1, -1) -- (1, 0);
\filldraw[fill= gray, very  thick, fill opacity=0.3] (0.6,0) rectangle (1.4,1);
\draw[black, very thick] (1, 1) -- (1, 2);
\filldraw[fill= gray, very  thick, fill opacity=0.3] (0.6,2) rectangle (1.4,3);
\draw[black, very thick] (1, 3) -- (1, 4);
\node  at (-1, -2.4) {\Large{$\Omega$}};
\draw [decorate,decoration={brace, amplitude=5pt}, xshift=-4pt,yshift=0pt]
(1,-0.9) -- (1,-0.1) node [black,midway,xshift=-0.6cm] {  $\ep|\mu_{2}|$};
\draw [decorate,decoration={brace, amplitude=5pt}, xshift=0pt,yshift=0pt]
(1.4,-2.1)--(0.6,-2.1)  node [black,midway,yshift=-0.4cm] {$\ep|\mu_{1}|$};
\draw[black, very thick, ->] (4.3, 1) -- ( 5.8, 1);
\node at (5.2, 1.4) {$\Large{y_{\eps}}$};
\draw[black, very thick] (6.1,-2) -- (8.7,-2);
\draw[black, very thick] (6.1, -2) -- (6.1, 4);
\draw[black, very thick] (8.7, -2) -- (8.7, -1);
\draw[black, very thick] (8.7, -1) -- (9.1, -1);
\draw[black, very thick] (9.1, -1) -- (9.1, 0);
\draw[black, very thick] (9.1, 0) -- (8.7, 0);
\draw[black, very thick] (8.7, 0) -- (8.7, 1);
\draw[black, very thick] (8.7, 1) -- (9.1, 1);
\draw[black, very thick] (9.1, 1) -- (9.1, 2);
\draw[black, very thick] (9.1, 2) -- (8.7, 2);
\draw[black, very thick] (8.7, 2) -- (8.7, 3);
\draw[black, very thick] (8.7, 3) -- (9.1, 3);
\draw[black, very thick] (9.1, 3) -- (9.1, 4);
\draw[black, very thick] (9.1, 4) -- (6.1, 4);
\draw[black, very thick] (8.7, -2) -- (8.7, -3);
\draw[black, very thick] (8.7, -3) -- (11.3, -3);
\draw[black, very thick] (11.3, -3) -- (11.3, 3);
\draw[black, very thick] (11.3, 3) -- (9.1, 3);
\node at (7.5, -2.4) {$\Large{y_{\eps}(\Omega)}$};
\filldraw[fill= gray, very  thick, fill opacity=0.3] (12,-2) rectangle (12.8,-1);
\filldraw[fill= gray, very  thick, fill opacity=0.3] (12,0) rectangle (12.8,1);
\filldraw[fill= gray, very  thick, fill opacity=0.3] (12,2) rectangle (12.8,3);
\end{tikzpicture}
\caption{Graphic representation of the deformation~$y_{\eps}$ in~\eqref{e:def-example-2}. The set $R_{\eps}$ is in gray.} \label{f:2}
\end{figure}

In particular, the jump set~$J_{y_{\eps}}$ satisfies the inequalities
\begin{align*}
\mathcal{H}^{1} ( J_{y_{\eps}}) & \leq  \mathcal{H}^{1} (\partial R_{\eps}) + (n_{\eps} +  1 \EEE )  \eps |\mu_{2}| = 2n_{\eps} \eps | \mu_{1}| + 2 n_{\eps} \eps |\mu_{2}| + (n_{\eps} + 1 \EEE) \eps | \mu_{2}|
\\
&
= 3 n_{\eps} \eps | \mu_{2}| + 2 n_{\eps} \eps | \mu_{1}|  +  \eps \EEE | \mu_{2}| \leq 3 \EEE + 2 \frac{| \mu_{1}|}{|\mu_{2}|} +  \eps \EEE | \mu_{2}| \,,\\
\mathcal{H}^{1}(J_{y_{\eps}}) & \geq  \mathcal{H}^{1} ( \partial R_{\eps}) -2 \eps|\mu_{1}| + (n_{\eps} - 1)  \eps |\mu_{2}| =  2n_{\eps} \eps | \mu_{1}| + 2 n_{\eps} \eps |\mu_{2}| -2 \eps|\mu_{1}| + (n_{\eps} -1) \eps | \mu_{2}|
\\
&
= 3 n_{\eps} \eps | \mu_{2}| + 2 n_{\eps} \eps | \mu_{1}| -2 \eps|\mu_{1}| - \eps | \mu_{2}|  \geq \EEE  3 - 4 \eps | \mu_{2}| + 2 \frac{| \mu_{1}|}{|\mu_{2}|} -  4 \eps |\mu_{1}|\,.
\end{align*}
Thus, we deduce that
\begin{displaymath}
\lim_{\eps \to 0} \mathcal{H}^{1}(J_{y_{\eps}}) = 3 + 2 \frac{|\mu_{1}|}{|\mu_{2}|}\,,
\end{displaymath}
 whereas in comparison $\mathcal{H}^1(J_u) = 2$. \EEE
}
\end{example}

 Examples \ref{ex} and \ref{ex:anisotropic} suggest that, besides $\kappa  \mathcal{H}^{1}( J_u )$, the formulation of the variational limit of~$\mathcal{E}_{\eps}$ in~\eqref{rig-eq: Griffith en}  should  account for an additional \EEE anisotropic surface term being positive whenever  \eqref{eq: CC} is violated. This term should depend \EEE on the orientation and on the amplitude of the jump of the displacement~$u$.   The full characterization of this $\Gamma$-limit is beyond the scope of the present contribution, and in the following we restrict our attention to \emph{energy-convergent sequences}. In order to \EEE comply with boundary conditions on $\Omega' \setminus  \overline{\Omega}\EEE$, for $h \in W^{2, \infty}(\Omega'; \R^{2})$ and $\eps >0$  we  set \EEE
\begin{align}\label{eq: boundary-spaces}
\mathcal{S}_{\eps,h} & = \lbrace y \in GSBV_2^2(\Omega';\R^d)\colon \ y = \id + \eps h \text{ on } \Omega' \setminus \overline{\Omega} \rbrace,
\end{align} 
and also recall the definition $GSBD^2_h(\Omega')$ in \eqref{eq: boundary-space}. \EEE

 We start by clarifying the definition of convergence. \EEE  The general idea in linearization results (see, e.g., \cite{Agostiniani, Braides-Solci-Vitali:07, DalMasoNegriPercivale:02, MFMK, FriedrichSchmidt:2014.2, NegriToader:2013, Schmidt:08, Schmidt:2009})  is to obtain compactness for the rescaled displacement fields $(u_\eps)_\eps$ associated to a sequence $(y_\eps)_\eps$ with $\sup_\eps \mathcal{E}_\eps(y_\eps) < +\infty$, see \eqref{eq: rescalidipl}.  For bodies undergoing fracture, however,  \EEE no compactness can be expected: consider,  for instance, \EEE the functions $y_\eps  \coloneqq \EEE \id \chi_{\Omega' \setminus B} + {\rm R}\,{\id}\EEE \chi_{B}$, for a small ball  $B \subset \Omega$ and a rotation  $\mathrm{R} \in SO(2)$, $\mathrm{R} \neq \Id$. \EEE Then $|u_\eps|, |\nabla u_\eps| \to \infty$ on $B$ as $\eps \to 0$.  As observed in~\cite[Theorem~2.3]{higherordergriffith}, this phenomenon can be avoided  if the deformation is \emph{rotated back to the identity} on the set $B$. This justifies the following notion of convergence, see also~\cite[Definition~2.4]{higherordergriffith}.

\begin{definition}[Asymptotic representation]
\label{def:conv}
{\normalfont
Fix  $\gamma \in (\frac{2}{3},\beta)$. \EEE We say that a sequence $(y_\eps)_\eps$ with $y_\eps \in \mathcal{S}_{\eps,h}$ is \emph{asymptotically represented} by a limiting displacement $u \in GSBD^2_h  (\Omega') \EEE$, and write $y_\eps \rightsquigarrow u$, if there exist sequences of Caccioppoli partitions $(P_j^\eps)_j$ of $\Omega'$ and corresponding rotations $(  R ^\eps_j)_j \subset SO( 2 \EEE)$ such that,  setting
\begin{equation}\label{eq: modifica}
y^{\rm rot}_\eps    \coloneqq    \sum_{j=1}^\infty R^\eps_j \,  y_\eps \,  \chi_{P^\eps_j}  \qquad \text{and} \qquad u_{\eps} \coloneqq \frac{1}{\eps} ( y^{\rm rot}_\eps \EEE  - \id),
\end{equation}
the following conditions hold:
\begin{subequations}
\label{eq: noname}
\begin{align}
 &  \Vert {\rm sym} (\nabla y_\eps^{\rm rot}) -\Id\Vert_{L^2(\Omega')} \le C\eps,\label{eq: first conditions-2} \\ 
 &  \Vert  \nabla y_\eps^{\rm rot} - \Id \Vert_{L^2(\Omega')} \le C\eps^\gamma, \label{eq: first conditions-3}\\
  &   |\nabla y_\eps^{\rm rot} - \Id| \le C\big(\eps^\gamma + \dist(\nabla y_\eps^{\rm rot},SO(2)) \big)  \ \ \text{a.e.~on $\Omega'$}\EEE \label{eq: first conditions-3.5}\\
& u_\eps \to u \ \ \ \text{a.e.\ in $\Omega' \setminus E_u$},\label{eq: the main convergence-1} \\
 & e(u_\eps) \rightharpoonup e(u)  \ \ \text{weakly in $L^2(\Omega' \setminus E_u; \R^{2\times 2}_{\rm sym})\EEE$},\label{eq: the main convergence-2}\\
 &  \mathcal{H}^{1}(J_{u}) \le \liminf_{\eps \to 0} \, \mathcal{H}^{1}(J_{u_\eps}) \le \liminf_{\eps \to 0}  \, \mathcal{H}^{1}(J_{y_\eps} \cup J_{\nabla y_\eps}),\label{eq: the main convergence-3}\\
 & e(u) = 0 \ \ \text{ on } \ E_u, \ \ \ \mathcal{H}^{1}\big((\partial^* E_u  \cap \Omega' \EEE )\setminus J_u \big) = \mathcal{H}^{1}(J_u \cap  (E_u)^1 \EEE ) = 0,\label{eq: the main convergence-4}
\end{align}
\end{subequations}
where    $E_u := \lbrace x\in \Omega: \, |u_\eps(x)| \to \infty \rbrace$ is a set of finite perimeter. 
}
\end{definition}

\begin{rem}
{\normalfont
The presence of the set $E_u$ is due to the compactness result in $GSBD^2(\Omega')$, see~\cite{Crismale}.  We point out that the behavior of the sequence  cannot be controlled on this set, but that this is \EEE not an issue  for minimization problems of Griffith energies since a minimizer can be recovered by choosing $u$ affine on $E_u$ with $e(u)=0$, cf.\ \eqref{eq: the main convergence-4}. We also note that $E_u \subset \Omega$, i.e., $E_u \cap (\Omega'\setminus \overline{\Omega}) = \emptyset$. \EEE    

We speak of asymptotic representation instead of convergence, and we use the symbol $ \rightsquigarrow $, in order to emphasize that  Definition \ref{def:conv} cannot be understood as a convergence with respect to a certain topology. Indeed, the  limit~$u$ \EEE for a given (sub-)sequence $(y_\eps)_\eps$ is not uniquely determined, but   rather \EEE depends on the choice of the sequences $(P_j^\eps)_j$ and $(R^\eps_j)_j$. For details in  that direction, in particular concerning a characterization of limiting displacements, we refer to \cite[Subsection~2.2]{higherordergriffith}.
} 
\end{rem}

 We have the following  compactness result for asymptotic representations.

\begin{proposition}[Compactness]\label{prop: comp}
Let $h \in W^{2, \infty}(\Omega'; \R^{2})$, and assume that $W$ satisfies \eqref{assumptions-W}.  Let $\gamma \in (\frac{2}{3},\beta)$.   Let $(y_\eps)_\eps$ be a sequence satisfying $y_\eps \in \mathcal{S}_{\eps,h}$ and $\sup_\eps \mathcal{E}_\eps(y_\eps) < +\infty$. Then there exists a subsequence (not relabeled) and $u \in GSBD^2_h(\Omega')$ such that $y_\eps \rightsquigarrow u$. 
\end{proposition}
The statement has been shown in \cite[Theorem~2.3]{higherordergriffith}. Actually, property \eqref{eq: first conditions-3.5} has not been stated there explicitly, but has been used in the proof, see \cite[(4.11)]{higherordergriffith}. We now present a consequence of Theorem \ref{thm: main} about the passage from the Ciarlet-Ne\v{c}as to the contact condition, whose proof is also postponed to  Section~\ref{sec: proof}.

\begin{theorem}[From Ciarlet-Ne\v{c}as to contact condition in the asymptotic representation]\label{thm: CNCC2}
Let  $h \in W^{2, \infty}(\Omega'; \R^{2})$, and assume that $W$ satisfies \eqref{assumptions-W}. Let $(y_\eps)_\eps$ be a sequence satisfying $y_\eps \in \mathcal{S}_{\eps,h}$ and \eqref{e:CN}. Let \EEE $u \in GSBD^2_h(\Omega')$ be such that $y_\eps \rightsquigarrow u$ and  $\mathcal{E}_\eps(y_\eps) \to \mathcal{E}(u)$ as $\eps \to 0$. Finally, assume that \EEE $y_\eps^{\rm rot}$ in \eqref{eq: modifica} also satisfies \eqref{e:CN}. Then, $u$ satisfies  \eqref{eq: CC} on $J_u \setminus \partial^* E_u$. 
\end{theorem} 
 
The assumption that also  $y_\eps^{\rm rot}$ satisfies \eqref{e:CN} is not really restrictive since it would also be possible to consider modifications of the form $y^{\rm rot}_\eps    \coloneqq    \sum_{j=1}^\infty (R^\eps_j \,  y_\eps  - b_j^\eps) \,  \chi_{P^\eps_j}$ for suitable $(b_j^\eps)_j \subset \R^2$ (cf.\ \cite[Theorem 2.2]{Friedrich:15-2}) such that \eqref{e:CN} holds. The full proof of this statement would be quite technical. Since this is not the main focus of the paper, we would rather not dwell on this point and just explain the general idea behind it: for $\ell\in \N$, assume that the sum of the contributions on the first $\ell-1$ sets is injective. If the local rotation of the contribution on the $\ell$-th set creates some overlapping, a translation $b_\ell^\eps$ is added to restore injectivity. An induction argument on the ordering of the partition then yields  the claim. \EEE 

 We conclude this section by stating the  third main \EEE contribution  of this work, whose proof is postponed to Section~\ref{s:th: recovery}. In particular, we assert that for every $u \in GSBD^{2}_{h}  (\Omega')$ satisfying the contact condition~\eqref{eq: CC} there exists an energy-convergent sequence (in the sense of Definition~\ref{def:conv}) which fulfills the Ciarlet-Ne\v{c}as  condition \EEE \eqref{e:CN}.\EEE

\begin{theorem}[Existence of energy-convergent sequences]\label{th: recovery}
Let $\Omega \subset \Omega' \subset  \R^2\EEE$ be bounded Lipschitz domains  satisfying~\eqref{eq: density-condition2}--\eqref{eq: density-condition}, let $h \in W^{2, \infty}(\Omega'; \R^{2})$, and assume that $W$ satisfies \eqref{assumptions-W}. Then, for every  $u \in GSBD^2_h (\Omega') $ \EEE satisfying~\eqref{eq: CC} \EEE there exists a sequence $(y_\eps)_\eps$ satisfying   \eqref{e:CN} and such that $y_\eps \in \mathcal{S}_{\eps,h}$, $y_\eps \rightsquigarrow u$, and  
  $$\lim_{\eps \to 0 } \mathcal{E}_{\eps}(y_\eps) = \mathcal{E}(u).$$ 
\end{theorem}


\section{Structural result for blow up around jump points}
\label{sec:blow}


This section is devoted to a preliminary result needed in the proofs of  Theorems \ref{thm: main} and \ref{th: recovery}. \EEE  For $\rho >0$, we set $Q^{\pm}_\rho \coloneqq Q_\rho \cap \lbrace  \pm x \cdot e_1 >0 \rbrace$. Here and in the following, $\pm$ is a placeholder for both $+$ and $-$.

\begin{proposition}\label{prop: Korn}
Let $0 < \rho \le 1$, let $v \in GSBD^2(Q_\rho)$,  let   $\omega^+ , \omega^- \in \R^2$, \EEE and let $0<\eta \le \min\lbrace\frac{1}{7}| \omega^+-\omega^-\EEE|,\theta_0,  10^{-4}  \rbrace$, where $\theta_0$ is the constant of Proposition \ref{th: kornpoin-sharp}.  Assume that  
\begin{subequations}
\label{eq: many properties2-lemma}
\begin{align}
&   \vphantom{\int} \mathcal{H}^1\big(J_v \cap Q_\rho\big) \le    \rho (1 + \eta), \label{eq: many properties2-lemma-1}\\
 &  \vphantom{\int} \mathcal{L}^2\Big(\Big\{ x\in Q^{+}_\rho\colon \, |v -  \omega^+\EEE| > \frac{\eta}{\bar{c}_{\eta^3/2}}  \Big\}  \Big) +  \mathcal{L}^2\Big(\Big\{ x\in Q^{-}_\rho\colon \, |v -  \omega^-\EEE| > \frac{\eta }{\bar{c}_{\eta^3/2} } \Big\}  \Big)   \le \rho^2\eta^4,\label{eq: many properties2-lemma-2} \\
 &   \int_{Q_\rho} |e(v)|^2 \, {\rm d}x \le \frac{\rho \eta^2}{C^2_\eta \bar{c}_{\eta^3/2}^2}, \label{eq: many properties2-lemma-3}  
\end{align} 
\end{subequations}
\EEE where $C_\eta \ge 1$ denotes the constant of Proposition \ref{th: kornpoin-sharp} applied for $\theta= \eta$,  and $\bar{c}_{\eta^3/2} \ge 1$ \EEE denotes the constant of Lemma \ref{lemma: rigid motion} applied for  $\theta = \eta^{3}/2$.  Then there exist two  disjoint \EEE sets $D^+, D^- \subseteq Q_\rho$  such that 
\begin{subequations}
\label{eq: two case4-lemma}
 \begin{align}
&  \Vert v -  \omega^+\EEE \Vert_{L^\infty(D^+)} \le 3\eta \quad \quad \text{and} \quad \quad \Vert v -  \omega^- \EEE\Vert_{L^\infty(D^-)} \le 3\eta, \label{eq: two case4-lemma-1} \\
 &  \mathcal{H}^1\Big( \big( (\partial^* D^+ \cup \partial^* D^-) \setminus J_{v}\big) \cap Q_\rho \Big) \le  6\eta\rho.\label{eq: two case4-lemma-2}
\end{align}
\end{subequations}
 Moreover, there exist two curves $\Gamma^{\pm} \subseteq \partial^{*} D^{\pm} \cap Q_\rho$ connecting $(-\frac{\rho}{2}, \frac{\rho}{2}) \times \{ - \frac{\rho}{2}\}$ to $(-\frac{\rho}{2}, \frac{\rho}{2}) \times \{ \frac{\rho}{2}\}$. \EEE 
\end{proposition}


\begin{rem}
{\normalfont
Later in the proofs of Theorem \ref{thm: main} and Theorem \ref{th: recovery} we will show that \eqref{eq: many properties2-lemma} holds in the blow-up around jump points.
}
\end{rem}

\begin{proof}[Proof of Proposition~\ref{prop: Korn}]
We first apply Proposition \ref{th: kornpoin-sharp} to construct the sets $D^\pm$. Afterwards, we prove the properties stated in  \eqref{eq: two case4-lemma}.

\noindent \emph{Step 1:   Application of   the \EEE piecewise Korn inequality.} We start by applying Proposition \ref{th: kornpoin-sharp}  for  $v$  and for $\theta = \eta$ on the set $Q_\rho$. (Note that $\eta \le \theta_0$ by assumption.)    We obtain   a (finite) Caccioppoli partition $Q_\rho = R \cup \bigcup^{J}_{j=1} P_j$, and corresponding rigid motions $(a_j)_{j=1}^{J}$ such  that~\eqref{eq: kornpoinsharp2XXX} holds. By assumptions~\eqref{eq: many properties2-lemma-1} and~\eqref{eq: many properties2-lemma-3} \EEE and the fact that $\mathcal{H}^1(\partial Q_\rho) = 4\rho$ we get
\begin{subequations}
\label{eq: kornpoinsharp2}
\begin{align}
 &   \sum_{j=1}^{J}\mathcal{H}^1\big( (\partial^* P_j \cap Q_\rho) \setminus J_{v} \big) +\mathcal{H}^1\big( (\partial^* R \cap Q_\rho )\setminus J_{v}  \big) \le \eta(\rho(1+\eta) + 4\rho),\label{eq: kornpoinsharp2-1}\\
 & \vphantom{\sum_{i=1}^{J}} \mathcal{L}^2(R)   \le \eta (\rho(1+\eta) + 4\rho)^2, \quad \quad   \mathcal{L}^2(P_j) \ge \rho^2\eta^3  \ \ \ \text{for all $j=1,\ldots,J$},     \label{eq: kornpoinsharp2-2}\\
& \vphantom{\sum_{i=1}^{J}}  \Vert v - a_j \Vert_{L^\infty(P_j)}  \le C_{\eta} \Vert  e(v) \Vert_{L^2(Q_\rho)} \le \frac{\sqrt{\rho}\eta}{\bar{c}_{\eta^3/2}}   \ \ \ \quad  \text{for all $j=1,\ldots,J$}.\label{eq: kornpoinsharp2-3}
\end{align}
\end{subequations}\EEE
We now show that for each $j=1,\ldots,J$ we have
\begin{align}\label{eq: two case}
\Vert v -  \omega^+ \EEE\Vert_{L^\infty(P_j)} \le 3\eta \quad \quad \quad \text{or} \quad \quad \quad \Vert v -  \omega^- \EEE\Vert_{L^\infty(P_j)} \le 3\eta.
\end{align}
 In fact,  since \eqref{eq: many properties2-lemma-2} and \eqref{eq: kornpoinsharp2-2} hold and  $\eta \le \frac{1}{4}$, we find that
 \begin{equation}
 \label{e: v - v+}
 \mathcal{L}^2 \Big(\Big\{  |v -  \omega^+\EEE| \le  \frac{\eta}{\bar{c}_{\eta^3/2}} \EEE \Big\} \cap P_j \Big) \ge \frac{\rho^2 \eta^3}{  4 \EEE} \qquad \text{or} \qquad \mathcal{L}^2 \Big (\Big\{  |v-  \omega^-\EEE| \le \frac{\eta}{\bar{c}_{\eta^3/2}} \EEE \Big\} \cap P_j \Big) \ge \frac{\rho^2 \eta^3}{  4 \EEE}.
 \end{equation}
  Without loss of generality we may assume that~\eqref{e: v - v+} \EEE holds true for~$ \omega^+\EEE$, and we write $  S_j \EEE  \coloneqq \EEE \lbrace  |v -  \omega^+\EEE| \le \eta/\bar{c}_{\eta^3 /2}\rbrace \cap P_j$. By  \eqref{eq: kornpoinsharp2-3},   the assumption that \EEE $\rho \le 1$, and the triangle inequality we get $\Vert  \omega^+\EEE - a_j \Vert_{L^\infty(  S_j \EEE)}  \le \frac{ 2\eta}{\bar{c}_{\eta^3/2}} \EEE$.  By applying Lemma \ref{lemma: rigid motion} for    $\theta =\frac{\eta^3}{2}$ \EEE and $R = \frac{\sqrt{2}\rho}{2}$ we then get
\begin{align*}
\Vert a_j -  \omega^+\EEE \Vert_{L^\infty(P_j)} \le \Vert a_j -  \omega^+\EEE \Vert_{L^\infty(Q_\rho)}     \le  \bar{c}_{\eta^3/2} \EEE  \Vert  a_j -  \omega^+\EEE \Vert_{ L^\infty  (  S_j \EEE)}\le 2\eta,
\end{align*} 
where we used that $\mathcal{L}^2(  S_j \EEE) \ge \frac{\eta^3}{2} R^2$  by \eqref{e: v - v+}. \EEE Another application of  \eqref{eq: kornpoinsharp2-3} \EEE and using  $\bar{c}_{\eta^3/2}  \ge 1$ \EEE  implies \eqref{eq: two case} for $ \omega^+\EEE$. In a similar fashion, we obtain the estimate for $ \omega^-\EEE$.

Since $| \omega^+-\omega^-\EEE| \ge 7\eta$ by assumption on $\eta$, we observe that for each $P_j$ estimate \eqref{eq: two case} either holds for $ \omega^+\EEE$ or for $ \omega^-\EEE$.   We denote by $\mathcal{J}^+$ \EEE the set of indices such that \eqref{eq: two case} holds for $ \omega^+\EEE$, and set $\mathcal{J}^- = \lbrace 1,\ldots,J \rbrace \setminus \mathcal{J}^+$. We define the sets
$$D^+  \coloneqq  \bigcup_{j \in \mathcal{J}^+} \EEE P_j, \quad \quad \quad  D^-  \coloneqq  \bigcup_{j \in \mathcal{J}^-} \EEE P_j. $$

\noindent \emph{Step 2: Proof of \eqref{eq: two case4-lemma}:} We start by observing that \eqref{eq: two case} implies 
$\Vert v -   \omega^+ \EEE \Vert_{L^\infty(D^+)} \le 3\eta$ and $\Vert v -  \omega^-\EEE \Vert_{L^\infty(D^-)} \le 3\eta$, i.e., \eqref{eq: two case4-lemma-1} holds.  By  \eqref{eq: kornpoinsharp2-1} and   since $\eta \le 1$ \EEE we also find  that \eqref{eq: two case4-lemma-2} holds true. In particular, by \eqref{eq: many properties2-lemma-1} \EEE 
\begin{equation}
\label{e:contradictD}
\mathcal{H}^1 \big((\partial^{*}D^{+} \cup \partial^{*}D^{-} ) \cap Q_{\rho} \big) \leq \rho + 7 \eta \rho\,.
\end{equation}
 It remains to show the existence of the curves $\Gamma^{\pm} \subseteq \partial^{*}D^{\pm}$. \EEE First, we note \EEE that $D^\pm \supset \lbrace  |v -   \omega^\pm\EEE| \le \eta  / \bar{c}_{\eta^3/2} \EEE \rbrace   \cap  (Q_\rho \setminus R)$ by construction and therefore we find by  \eqref{eq: many properties2-lemma-2} and \eqref{eq: kornpoinsharp2-2} \EEE that
 \begin{align}\label{eq: two case3}
\mathcal{L}^2\big( D^\pm \cap Q^{\pm}_\rho   \big) \ge \frac{1}{2}\rho^2 -\rho^2\eta^4 - \mathcal{L}^2(R) \ge  \frac{1}{2}\rho^2 -\rho^2\eta^4 - 36\rho^2\eta \ge \frac{1}{2}\rho^2-  C_0   \rho^2\eta,
\end{align}
  where we set $C_0 = 100$ for notational convenience.  Thus, by \eqref{eq: two case3} we get that \EEE
\begin{equation}
\label{eq: slice-e1}
\mathcal{H}^1\Big(\big\{ w \in \Pi^{e_1}\colon \,  \mathcal{L}^1\big((D^\pm \cap Q^{\pm}_\rho)^{e_1}_w  \big) >0 \big\}\Big) \ge \rho - 2C_0\eta\rho.
\end{equation}
This in turn implies 
\begin{equation}
\label{eq: slice-e1-2}
\mathcal{H}^1\Big(\big\{ w \in \Pi^{e_1}\colon \,  \mathcal{H}^0\big((\partial^* D^\pm \cap Q_\rho)^{e_1}_w  \big) \ge 1 \big\}\Big) \ge \rho - 4C_0\eta\rho.
\end{equation}
We further claim that
\begin{equation}
\label{eq: slice-e1-4}
\mathcal{H}^1\Big(\big\{ w \in \Pi^{e_1}\colon \,  \mathcal{H}^0\big((\partial^* D^\pm \cap Q_\rho)^{e_1}_w  \big) \ge 2 \big\}\Big) \leq \frac{\rho}{2} + \frac{7}{2} \eta \rho\,.
\end{equation}
Indeed, if~\eqref{eq: slice-e1-4} were not true, by the area formula (see, e.g., \cite[Theorem 2.71]{Ambrosio-Fusco-Pallara:2000}) and by~\eqref{e:contradictD} \EEE we would have  that
\begin{align*}
\rho + 7\eta \rho & < 2 \, \mathcal{H}^1\Big ( \big \{ w \in \Pi^{e_1} \colon \,  \mathcal{H}^0 \big ( (\partial^* D^\pm \cap Q_\rho)^{e_1}_w  \big) \ge 2 \big \} \Big ) \leq \int_{\Pi^{e_{1}}}  \mathcal{H}^0\big((\partial^* D^\pm \cap Q_\rho)^{e_1}_w  \big) \, {\rm d} \mathcal{H}^{1}(w) 
\\
&
= \int_{\partial^* D^\pm \cap Q_\rho} | \nu_{D^{\pm}} \cdot e_{1} | \, {\rm d} \mathcal{H}^{1} \leq \mathcal{H}^{1} ((\partial^{*} D^{+} \cup \partial^{*}D^{-}) \cap Q_{\rho}) \leq \rho + 7\eta\rho\,,
\end{align*}
where by $\nu_{D^\pm}$ we denote the  outer unit normal of $\partial^* {D^\pm}$. Thus, \eqref{eq: slice-e1-4} holds true. \EEE Let us set $\delta \coloneqq 16C_{0} \eta$ and $K_{\delta} \coloneqq (-\frac{\delta\rho}{2}, \frac{\delta\rho}{2}) \times (-\frac{\rho}{2}, \frac{\rho}{2})$. Then, by \eqref{eq: two case3} we get
\begin{align*}
\mathcal{L}^2\big( D^\pm \cap Q^{\pm}_\rho  \cap K_{\delta} \big) & \ge  \mathcal{L}^2\big( D^\pm \cap Q^{\pm}_\rho) - \mathcal{L}^{2} (Q_{\rho}^{\pm} \setminus K_{\delta}) 
\\
&
\geq \frac{1}{2}\rho^2-  C_0   \rho^2\eta - \frac{1}{2} (1-\delta) \rho^{2} =  \frac{1}{2}\delta \rho^2-  C_0   \rho^2\eta\,.
\end{align*}
Arguing as in~\eqref{eq: two case3}--\eqref{eq: slice-e1-2} we deduce that 
\begin{equation}
\label{eq: slice-e1-3}
\mathcal{H}^1\Big(\big\{ w \in \Pi^{e_1}\colon \,  \mathcal{H}^0\big((\partial^* D^\pm \cap K_{\delta})^{e_1}_w  \big) \ge 1 \big\}\Big) \ge \rho - 4\frac{C_0\eta}{\delta}\rho = \frac{3}{4} \rho\,,
\end{equation}
where in the last equality we have used the definition of~$\delta$. Hence, combining~\eqref{eq: slice-e1-4} and~\eqref{eq: slice-e1-3} we infer that
\begin{equation}
\label{eq: slice-e1-5}
\mathcal{H}^1\Big(\big\{ w \in \Pi^{e_1}\colon \,  \mathcal{H}^0\big((\partial^* D^\pm \cap K_{\delta})^{e_1}_w  \big) =1 \quad \text{and} \quad   \mathcal{H}^0\big((\partial^* D^\pm \cap Q_\rho)^{e_1}_w  \big) =1 \big\}\Big) \geq \frac{\rho}{4} - \frac{7}{2} \eta \rho\,. 
\end{equation}
 We now prove the existence of a curve~$\Gamma^{+} \subseteq \partial^{*}D^{+}$ connecting $(-\frac{\rho}{2}, \frac{\rho}{2}) \times \{ -\frac{\rho}{2}\}$ with $(-\frac{\rho}{2}, \frac{\rho}{2}) \times \{ \frac{\rho}{2}\}$. The argument for $\Gamma^{-} \subseteq \partial^{*}D^{-}$ is the same, with a different notational realization. In view of~\eqref{eq: slice-e1-5}, we can fix $w \in (-\frac{\rho}{2} , \frac{\rho}{2})$ and $t \in (-\frac{\delta\rho}{2}, \frac{\delta\rho}{2})$ such that
\begin{equation}
\label{e:pointy}
y \coloneqq (t, w) \in \partial^{*}D^{+} \cap K_{\delta} \qquad \text{and} \qquad (s, w)  \notin \partial^{*}D^{+} \cap Q_{\rho} \quad \text{for $s \in (-\tfrac{\rho}{2},\tfrac{\rho}{2})$, \EEE $s \neq t$}.
\end{equation}
Without loss of generality, we may assume that there exists the approximate unit normal~$\nu_{D^{+}}(y)$ to~$\partial^{*}D^{+}$ in~$y$ and that $\nu_{D^{+}} (y) \cdot e_{1} \neq 0$. By~\cite[Corollary 1]{Ambrosio-Morel}, $\partial^* D^+$ can be decomposed uniquely into at most countably many  pairwise disjoint rectifiable Jordan curves. Let us denote by~$\Lambda \subseteq \partial^{*}D^{+}$ the Jordan curve containing~$y$. Then,~\eqref{e:pointy} and $\nu_{D^{+}} (y) \cdot e_{1} \neq 0$ imply \EEE that $\Lambda \cap Q_{\rho} \subsetneq \Lambda$. Thus,~$\Lambda$ must connect~$y$ to~$\partial Q_{\rho}$. Let us denote by~$\Gamma^{+} \subseteq \Lambda$ the sub-curve of $\Lambda$ containing~$y$ and intersecting~$\partial Q_{\rho}$ only in its endpoints.

We now show that such endpoints lie in~$(-\frac{\rho}{2}, \frac{\rho}{2}) \times \{ -\frac{\rho}{2}\}$ and in~$(-\frac{\rho}{2}, \frac{\rho}{2}) \times \{ \frac{\rho}{2}\}$, respectively. By contradiction, let us assume that one of the endpoints is of the form $(\frac{\rho}{2}, \overline{w})$ or $(-\frac{\rho}{2}, \overline{w})$ with $\overline{w} \in (-\frac{\rho}{2}, \frac{\rho}{2})$. Setting $|w - \overline{w}| = \zeta \rho$ for some $\zeta \in (0, 1)$, by~\eqref{eq: slice-e1-2} and by definition of~$\delta$ and $\eta$ we   estimate
\begin{align*}
\mathcal{H}^{1} ((\partial^{*}D^{+} \cup \partial^{*}D^{-}) \cap Q_{\rho}) & \geq \sqrt{ \Big( \frac{\rho(1 - \delta)}{2} \Big)^2 + | w - \overline{w}|^{2}} + (\rho - 4C_{0} \eta \rho) - | w - \overline{w}|
\\
&
=\rho \Big( \sqrt{ \frac{(1 - \delta)^{2}}{4} + \zeta^{2}} - \zeta\Big) + (\rho - 4C_{0} \eta \rho) 
\\
&
\geq \rho  \frac{(1 - \delta)^{2}}{ 4 + \sqrt{20}} + \rho - 4C_{0}\eta \rho \geq \rho  \frac{(1 - \delta)^{2}}{ 9} + \rho - 4C_{0}\eta \rho > \rho + 7\eta\rho\,,
\end{align*}
which is in contradiction to~\eqref{e:contradictD}. Hence, both endpoints of~$\Gamma^{+}$ lie  on~$(-\frac{\rho}{2}, \frac{\rho}{2}) \times \{ -\frac{\rho}{2}\}$ or on~$(-\frac{\rho}{2}, \frac{\rho}{2}) \times \{ \frac{\rho}{2}\}$. Since~\eqref{e:pointy} holds, the endpoints can not both lie on the same side. Thus, $\Gamma^{+}$ connects $(-\frac{\rho}{2}, \frac{\rho}{2}) \times \{ -\frac{\rho}{2}\}$ and $(-\frac{\rho}{2}, \frac{\rho}{2}) \times \{\frac{\rho}{2}\}$. This concludes the proof of the proposition.
\end{proof}

\section{Proof of Theorem \ref{thm: main}}\label{sec: proof}
This section is entirely devoted to the proofs of Theorems \ref{thm: main} and \ref{thm: CNCC2}. \EEE

\begin{proof}[Proof of Theorem \ref{thm: main}]
We start by noting that the  Ciarlet-Ne\v{c}as \EEE non-interpenetration condition \eqref{e:CN} along with \eqref{eq:area} implies
\begin{align}\label{eq: NC refined}
 \int_E \det \nabla y_\eps \, {\rm d}x  = \mathcal{L}^2([y_\eps(E)])
\end{align}
for all measurable sets $E \subset \Omega$.

The proof   of the theorem \EEE is performed by contradiction. We suppose that there exists a rectifiable set $J^{\rm int} \subset J_u$ with $\mathcal{H}^1(J^{\rm int})>0$ such that $[u](x) \cdot \nu_u(x) < 0$ for all $x \in J^{\rm int}$. By a careful analysis of the blow-up around a point in~$J^{\rm int}$, we will construct a sequence of subsets~$E_{\eps} \subseteq \Omega$ which violates~\eqref{eq: NC refined}, i.e., such that 
\begin{align}\label{eq: to show}
 \int_{E_\eps} \det (\nabla y_\eps) \, {\rm d}x  > \mathcal{L}^2([y_\eps(E_\eps)]).
\end{align}
The   argument \EEE is divided into several steps: in Step 1 we show by a blow-up argument that around a point in~$J^{\rm int}$ the sequence $u_{\eps} = \frac{1}{\eps} (y_{\eps} - \id)$ satisfies the assumptions~\eqref{eq: many properties2-lemma} of Proposition~\ref{prop: Korn}. In Steps 2 and 3 we estimate the two sides of~\eqref{eq: to show} separately, assuming that a sequence~$E_{\eps}  = G^+_\eps \cup G_\eps^-\EEE$ of subsets of~$\Omega$ exists such that~\eqref{eq: thickening-general} below is satisfied. The remaining part of the proof (Steps~4--6) \EEE is devoted to the construction of such  a \EEE sequence.

\noindent \emph{Step 1: Blow-up.}  Up \EEE to a translation and rotation, it is not restrictive to assume that  $0 \in J^{\rm int}$, that $\nu_u(0) = e_1$,  and that there exist $u^+,u^- \in \R^2$ with $(u^+-u^-) \cdot e_1 <0$  such that
\begin{subequations}
\label{eq: many properties1}
\begin{align}
 &  \lim_{\rho \to 0} \, \rho^{-1}\, \mathcal{H}^1\big(J_u \cap Q_\rho\big) = 1, \label{eq: many properties1-1}\\
 & \lim_{\rho \to 0} \, \rho^{-2}\Big( \mathcal{L}^2\big(\lbrace x\in Q^{+}_\rho\colon \, |u - u^+| > \epsilon  \rbrace  \big) +  \mathcal{L}^2\big(\lbrace x\in Q^{-}_\rho\colon \, |u - u^-| > \epsilon \rbrace  \big)  \Big) =0 \quad \forall \epsilon>0,\label{eq: many properties1-2} \\
 & \lim_{\rho \to 0} \, \rho^{-1} \int_{Q_\rho} |e(u)|^2 \, {\rm d}x = 0, \label{eq: many properties1-3}
\end{align}
\end{subequations}
where we recall that \EEE $Q^{\pm}_\rho = Q_\rho \cap \lbrace  \pm x \cdot  e_{1} \EEE >0 \rbrace$. Indeed,  \eqref{eq: many properties1} \EEE holds true for $\mathcal{H}^1$-a.e.\ $x \in J^{\rm int}$:   property \EEE \eqref{eq: many properties1-1} follows from the countably $\mathcal{H}^1$-rectifiability of $J^{\rm int}$,  \eqref{eq: many properties1-2} \EEE follows directly  from the definition of $J_u$, and   \eqref{eq: many properties1-3} \EEE holds due to $|e(u)|^2 \in L^1(\Omega)$. 

For convenience, we denote the direction of the jump by $\mu :=  (u^+ - u^-)/|u^+-u^-|$. Recall that  $\mu\cdot e_1 <0$. \EEE
We now pick a constant $\eta$ sufficiently small whose choice will become   clear \EEE along the proof. To this end, we first choose  $\delta \in (0,1)$ sufficiently small such that
\begin{align}\label{eq: delta}
16\delta \le |\mu \cdot e_1| \quad \quad  \text{and} \quad \quad 1 \ge (1-\delta) \sqrt{1+9\delta^2},
\end{align}
and then  choose \EEE $\lambda \in (0,1)$ sufficiently small such that 
\begin{subequations}
\label{eq: lambda}
\begin{align}
&   \text{for each $\nu \in \mathbb{S}^1$ with    $|\nu \cdot e_1| \ge 1-\lambda$, we have $\big|  |\mu \cdot e_1| -   |\mu \cdot \nu|      \big| \le \delta$,}\label{eq: lambda-1} \\
 & 2\sqrt{1 -  (1-\lambda)^{2}  } \le  \delta. \label{eq: lambda-2}
\end{align}
\end{subequations}
\EEE
For notational convenience, we indicate by $C_0$ a fixed constant with $C_0 \ge  10^3\EEE$.  Eventually, we define $\eta \in (0,1)$ such that
\begin{align}\label{eq: eta}
\eta < \min\Big\{ \theta_0,   \  \frac{|u^+-u^-||\mu\cdot e_1|}{ 16  (C_0N_2 + 1)},  \frac{1}{C_0^2}, \ \frac{\lambda\delta}{21} \Big\}, 
\end{align}
where $N_2$ denotes the dimensional constant appearing in Besicovitch's covering theorem (see, e.g.,~\cite[Theorem~2.18]{Ambrosio-Fusco-Pallara:2000}), and  $\theta_0$ denotes the constant from Proposition \ref{th: kornpoin-sharp}.

 By \cite[Theorem 11.3]{DalMaso:13} for every open subset $A$ of~$\Omega$ it holds that  \EEE
\begin{displaymath}
\liminf_{\eps \to 0}  \, \Vert e(u_\eps) \Vert^2_{L^2(A)}  \ge  \Vert  e(u)  \Vert^2_{L^2(A)}  \qquad \text{and} \qquad \liminf_{\eps \to 0} \, \mathcal{H}^{1}(J_{u_\eps} \cap A) \ge \mathcal{H}^{1}(J_{u} \cap A).
\end{displaymath}
 Therefore, hypothesis~\eqref{eq: assu-u-2}  implies
\begin{align}\label{eq: NC refined2.1}
\lim_{\eps \to 0} \,  \Vert e(u_\eps) \Vert_{L^2(A)} =  \Vert e(u) \Vert_{L^2(A)} \qquad \text{and} \qquad  \lim_{\eps \to 0} \, \mathcal{H}^{1}(J_{u_\eps}\cap A) =  \mathcal{H}^{1}(J_{u} \cap A)
\end{align}
for all $A \subset \Omega$ open with $\mathcal{H}^1(\partial A \cap J_u) = 0$. \EEE In view of  \eqref{eq: NC refined2.1} \EEE applied for $A = Q_\rho$,  of \EEE \eqref{eq: many properties1} with $\epsilon= \frac{\eta  }{  \bar{c}_{\eta^3/2} \EEE}$, and  of \EEE the fact that $u_\eps \to u$ in measure,  we can fix a particular $0<\rho \le 1$ with~$\mathcal{H}^1(\partial Q_\rho \cap J_u) = 0$  such that for all $\eps$ sufficiently small we have \EEE 
\begin{subequations}
\label{eq: many properties2}
\begin{align}
&  \vphantom{\int} \mathcal{H}^1\big(J_{u_\eps} \cap Q_\rho\big) \le    \rho (1 + \eta), \label{eq: many properties2-1}\\
&  \vphantom{\int} \mathcal{L}^2\Big(\Big\{ x\in Q^{+}_\rho\colon \, |u_\eps - u^+| >  \frac{\eta}{ \bar{c}_{\eta^3/2} } \Big\}  \Big) +  \mathcal{L}^2\Big(\Big\{ x\in Q^{-}_\rho\colon \, |u_\eps - u^-| > \frac{\eta }{ \bar{c}_{\eta^3/2}}  \Big\}  \Big)   \le \rho^2\eta^4,\label{eq: many properties2-2} \\
&   \int_{Q_\rho} |e(u_\eps)|^2 \, {\rm d}x \le \frac{\rho \eta^2}{ C^2_\eta \bar{c}_{\eta^3/2}^2}, \label{eq: many properties2-3} 
\end{align}
\end{subequations}\EEE
where $C_\eta \ge 1$ denotes the constant of Proposition \ref{th: kornpoin-sharp} applied for $\theta= \eta$,  and   $\bar{c}_{\eta^3/2} \ge 1$ \EEE denotes the constant of Lemma \ref{lemma: rigid motion} applied for  $\theta = \eta^3/2$. \EEE   In the following, without further notice, $\eps$ will always be chosen sufficiently small such that \eqref{eq: many properties2} holds. The strategy  of the proof is  to construct a measurable subset $E_\eps \subset Q_\rho$ such that~\eqref{eq: to show} holds, which \EEE is a contradiction to~\eqref{eq: NC refined}. To show~\eqref{eq: to show}, we now estimate separately  its  left- and right-hand side.\EEE

\noindent \emph{Step 2: Estimate on  the \EEE determinant.} In dimension two, a Taylor expansion implies that 
$$\big|\det(\Id + F) - (1 + {\rm tr}(F))\big| \le c_0|F|^2$$ 
for a universal constant $c_0>0$. For all measurable $E \subset Q_\rho$ this implies by  \eqref{eq: many properties2-3},   the fact that \EEE $ C_\eta \bar{c}_{\eta^3/2}  \ge 1$, \EEE and H\"older's inequality that \begin{align}\label{eq: det-eps}
 \int_{E} \det (\nabla y_\eps) \, {\rm d}x & \ge \mathcal{L}^2(E) - \int_E  \sqrt{2} \EEE \eps |e(u_\eps)| \, {\rm d}x  -  c_0\int_E \eps^2 |\nabla u_\eps|^2 \, {\rm d}x   \\ 
 & \ge \mathcal{L}^2(E) -  \sqrt{2} \EEE \eps\big(\mathcal{L}^2(E)\big)^{1/2} \Vert e(u_\eps) \Vert_{L^2(Q_\rho)} -  c_0\int_\Omega \eps^2 |\nabla u_\eps|^2 \, {\rm d}x \notag \\
 & \ge \mathcal{L}^2(E)  - \sqrt{2} \EEE\eps\rho \sqrt{\rho} \eta- c_0\int_\Omega \eps^2 |\nabla u_\eps|^2 \, {\rm d}x. \nonumber
\end{align}
Note that $\eps \int_\Omega|\nabla u_\eps|^2  \, {\rm d}x \to 0$ as $\eps \to 0$ by  \eqref{eq: assu-u-1} \EEE and the fact that $\gamma >\frac{1}{2}$. Therefore, for all $\eps>0$ sufficiently small (depending on  $\rho$ and  $\eta$),  we deduce from~\eqref{eq: det-eps} that \EEE
\begin{align}\label{eq:det}
 \int_{E} \det (\nabla y_\eps) \, {\rm d}x  \ge \mathcal{L}^2(E)  - 2\eps\rho^{3/2} \eta \ge  \mathcal{L}^2(E)  - 2\eps\rho \eta, 
\end{align}
where the last step follows from   observing that \EEE $\rho \le 1$.

 \noindent \emph{Step 3:  Estimate on   the measure of the \EEE image and conclusion.}   By $\rho \le 1$,   \EEE \eqref{eq: eta}, and   \eqref{eq: many properties2}, we can apply Proposition \ref{prop: Korn}  to \EEE $v = u_\eps$  and  $\omega^\pm=u^\pm$. \EEE We \EEE find two sets $D^\pm_\eps$ satisfying \eqref{eq: two case4-lemma} and the curves $\Gamma^\pm_\eps$. \EEE  Based on the definition of $ D^\pm_\eps$, we  construct in Steps 3--6 \EEE two disjoint sets $G^\pm_\eps \subset D^\pm_\eps$ satisfying
\begin{subequations}
\label{eq: thickening-general}
\begin{align}
 & \mathcal{L}^2\big(\lbrace x\in \R^2 \setminus G^\pm_\eps  \colon  \dist(x, G^\pm_\eps) \le 3\eta\eps \rbrace \big) \le  C_0N_2\eta\eps\rho,\label{eq: thickening-general-1}\\
 &   \mathcal{L}^2\big(  (\eps u^- +   G^-_\eps) \cup (\eps u^+ +   G^+_\eps)    \big) \le \mathcal{L}^2(G^+_\eps) + \mathcal{L}^2(G^-_\eps) - \frac{\rho}{ 8} \eps |u^+-u^-||\mu\cdot e_1|.  \label{eq: thickening-general-2} 
\end{align}
\end{subequations} \EEE
 Let us assume for the moment that such sets exist and let us \EEE explain how to conclude the contradiction.  By  \eqref{eq: two case4-lemma-1}, \EEE \eqref{eq: thickening-general}, and the fact that $G^\pm_\eps \subset D^\pm_\eps$ we find that, for $\eps$ sufficiently small, the functions $y_\eps = \id + \eps u_\eps $ satisfy
\begin{align}\label{eq: L2G}
\mathcal{L}^2& ([y_\eps(G^+_\eps \cup G^-_\eps)]) 
\\
& \le\mathcal{L}^2\Big( \lbrace x\in \R^2 \colon  \dist(x, \eps u^+ +  G^+_\eps) \le 3\eta\eps \rbrace \cup \lbrace x\in \R^2 \colon  \dist(x, \eps u^- +G^-_\eps) \le 3\eta\eps \rbrace \Big) \nonumber \\
& \le   \mathcal{L}^2\big(  (\eps u^- +   G^-_\eps) \cup (\eps u^+ +   G^+_\eps)    \big)  + 2C_0N_2\eta\eps\rho \nonumber\\
& \le \mathcal{L}^2(G^+_\eps) + \mathcal{L}^2(G^-_\eps) - \frac{\rho}{ 8 } \eps |u^+-u^-||\mu\cdot e_1| + 2C_0N_2\eta\eps\rho. \nonumber
\end{align}
On the other hand,  since~\eqref{eq:det} holds and $G^+_\eps \cap G^-_\eps = \emptyset$, we have
\begin{align}\label{eq: L2G2}
 \int_{G^+_\eps \cup G^-_\eps} \det (\nabla y_\eps) \, {\rm d}x  \ge \mathcal{L}^2(G^+_\eps) + \mathcal{L}^2(G^-_\eps)  - 2\eps\rho \eta.
\end{align}
Combining~\eqref{eq: eta} and~\eqref{eq: L2G}--\eqref{eq: L2G2} we infer that \EEE 
\begin{align*}
 \int_{G^+_\eps \cup G^-_\eps} \det (\nabla y_\eps) \, {\rm d}x &\ge \mathcal{L}^2([y_\eps(G^+_\eps \cup G^-_\eps)]) + \frac{\rho}{8 } \eps |u^+-u^-||\mu\cdot e_1| -  2C_0N_2\eta\eps\rho - 2\eps\rho \eta \\ & >  \mathcal{L}^2([y_\eps(G^+_\eps \cup G^-_\eps)]).
\end{align*}
This shows \eqref{eq: to show}  for $E_{\eps} = G^{+}_{\eps} \cup G^{-}_{\eps}$ \EEE and the argument by contradiction is concluded. To conclude the proof, it remains to give the construction of the sets $G^\pm_\eps$ and to prove  the \EEE properties \eqref{eq: thickening-general}.

 \noindent \emph{Step 4: Definition of $G_\eps^\pm$.}  By Proposition~\ref{prop: Korn} there exists a curve $\Gamma_{\eps} \subseteq \partial^{*}D_{\eps}^{+}$ which connects $(-\frac{\rho}{2}, \frac{\rho}{2}) \times \lbrace \frac{\rho}{2} \rbrace$ with $(-\frac{\rho}{2}, \frac{\rho}{2}) \times \lbrace -\frac{\rho}{2} \rbrace$. In particular, we assume that \EEE there exists a continuous curve $\gamma_\eps \colon  [a,b] \to \R^2$ with $\Gamma_\eps = \gamma_\eps([a,b])$, $\gamma_\eps((a,b)) \subset Q_\rho$, $\gamma_\eps(a) \in (-\frac{\rho}{2}, \frac{\rho}{2}) \times \lbrace \frac{\rho}{2} \rbrace$, and $\gamma_\eps(b) \in (-\frac{\rho}{2}, \frac{\rho}{2}) \times \lbrace -\frac{\rho}{2} \rbrace$. \EEE


 We define $F^+_\eps  \coloneqq \EEE {\rm Int}(\Psi_\eps)$ and $F^-_\eps  \coloneqq \EEE Q_\rho \setminus F^+_\eps$, where ${\rm Int}(\cdot)$  stands for \EEE the interior of a Jordan curve. We denote the connected components of ${\rm sat}(F^\pm_\eps \setminus D^\pm_\eps)$ by $(S^\pm_{i,\eps})_i$, where ${\rm sat}(\cdot)$  indicates \EEE the saturation of a set. Note that each of these sets is simple, i.e.,  $\partial^* S^\pm_{i,\eps}$ is equivalent to a rectifiable Jordan curve  up to an \EEE $\mathcal{H}^1$-negligible set. We define
\begin{align}\label{eq: G-def}
G^+_\eps := F^+_\eps \setminus  \bigcup_ {i}S^+_{i,\eps},  \quad \quad \quad  G_\eps^- := F^-_\eps \setminus  \bigcup_{i} S^-_{i,\eps}.  \EEE
\end{align}
By construction, we note that $\Gamma_\eps \subset \partial^* G^+_\eps \cap Q_\rho$ up to an $\mathcal{H}^1$-negligible set. Moreover, we have
\begin{subequations}
\label{eq: length bound}
\begin{align}
 &  \mathcal{H}^1(\Gamma_\eps) \le \mathcal{H}^1\big( (\partial^* G^+_\eps \cup \partial^* G^-_\eps) \cap Q_\rho \big) \le  \rho +7\rho\eta \le 2\rho, \label{eq: length bound-1} \\
 &  \mathcal{H}^1\big(\partial^* G^+_\eps \cup \partial^* G^-_\eps\big) \le 6\rho. \label{eq: length bound-2}
\end{align}
\end{subequations}\EEE
Indeed, by  \cite[Proposition 6(ii)]{Ambrosio-Morel} and by \eqref{eq: two case4-lemma-2} and \eqref{eq: many properties2-1} \EEE one can check that 
\begin{align*}
\mathcal{H}^1\big( (\partial^* G^+_\eps \cup \partial^* G^-_\eps) \cap Q_\rho \big) & \le \mathcal{H}^1\big( (\partial^* D^+_\eps \cup \partial^* D^-_\eps) \cap Q_\rho \big) \\
& \le  \mathcal{H}^1(J_{u_\eps} \cap Q_\rho) + \mathcal{H}^1\big( \big( (\partial^* D^+_\eps \cup \partial^* D^-_\eps) \setminus J_{u_\eps}\big) \cap Q_\rho \big) \le  \rho(1+7\eta).
\end{align*}
Then,  \eqref{eq: length bound-1} \EEE follows from the fact that $\Gamma_\eps \subset \partial^* G^+_\eps \cap Q_\rho$ up to an $\mathcal{H}^1$-negligible set, and the fact that $7\eta \le 1$, see \eqref{eq: eta}. To get  \eqref{eq: length bound-2}, \EEE we simply note that $\mathcal{H}^1(\partial Q_\rho) = 4\rho$.

\noindent \emph{\UUU Step 5: \EEE Proof of  \eqref{eq: thickening-general-1}.\EEE} Without  loss of generality, we show  \eqref{eq: thickening-general-1} only for $G^+_\eps$. The proof for~$G^-_\eps$ works in the same way,   up to a different notational realization. \EEE Let $\mathcal{S}^{\rm small}_\eps  \coloneqq \EEE \lbrace i \colon  \mathcal{H}^1(\partial^{\ZZZ *} S^+_{i,\eps}) \le {3\eps\eta} \rbrace$ and  $\mathcal{S}^{\rm big}_\eps  \coloneqq \EEE\lbrace i \colon  \mathcal{H}^1(\partial^* S^+_{i,\eps}) > {3\eps\eta} \rbrace$. By \eqref{eq: G-def} and the fact that $\partial^* F^+_\eps$ and $( \partial^* S^+_{i,\eps})_i$ are equivalent to rectifiable Jordan curves  we get
\begin{align}\label{eq: 6-1}
\mathcal{L}^2\big(\lbrace x\in \R^2 \setminus G^+_\eps \colon  \dist(x, G^+_\eps) \le 3\eta\eps \rbrace \big) & \le  \sum_{i \in \mathcal{S}^{\rm small}_\eps} \mathcal{L}^2(S^+_{i,\eps})  + \sum_{i \in \mathcal{S}^{\rm big}_\eps} \mathcal{L}^2\big(\lbrace x \colon  \dist(x, \partial^* S^+_{i,\eps}) \le 3\eps\eta \rbrace \big) \notag  \\ & \ \ \ + \mathcal{L}^2\big(\lbrace x \colon  \dist(x, \partial^* F^+_{\eps}) \le 3\eps\eta \rbrace \big).
\end{align}
We now estimate the terms  on the right-hand side of~\eqref{eq: 6-1} \EEE separately.  For the first term, \EEE we recall the definition of $\mathcal{S}^{\rm small}_\eps$ and use 
the isoperimetric inequality to get
\begin{align}\label{eq: iso}
 \sum_{i \in \mathcal{S}^{\rm small}_\eps} \EEE \mathcal{L}^2(S^+_{i,\eps}) \le \frac{1}{4\pi}  \sum_{i \in \mathcal{S}^{\rm small}_\eps} \EEE \big(\mathcal{H}^1(\partial^* S^+_{i,\eps})\big)^{2} \le \frac{{3\eps\eta}}{4\pi}  \sum_{i \in \mathcal{S}^{\rm small}_\eps} \EEE \mathcal{H}^1(\partial^* S^+_{i,\eps}).  
\end{align}
For the other two terms, we show that  
\begin{subequations}
\label{eq: iso2}
\begin{align}
 \mathcal{L}^2\big(\lbrace x \colon  \dist(x, \partial^* S^+_{i,\eps}) \le 3\eps\eta \rbrace \big)  \le 40 \EEE N_2\eps\eta \, \mathcal{H}^1 (\partial^* {S}^+_{i,\eps}),  \label{eq: iso2-1} \\  \mathcal{L}^2\big(\lbrace x \colon  \dist(x, \partial^* F^+_{\eps}) \le 3\eps\eta \rbrace \big) \le 40 \EEE N_2\eps\eta  \,  \mathcal{H}^1 (\partial^* F^+_{\eps}), \label{eq: iso2-2}
\end{align}
\end{subequations}\EEE
where $N_2$ denotes the constant in the Besicovitch covering theorem. 
We first perform the proof for the sets $S^+_{i,\eps}$. For notational simplicity, we set $\tilde{S}^+_{i,\eps}  \coloneqq \EEE \lbrace x \colon  \dist(x, \partial^* S^+_{i,\eps}) \le {3\eps\eta} \rbrace$. We cover $\tilde{S}^+_{i,\eps}$ by balls $B_{ 6\EEE \eps\eta}(x)$ with $x \in \partial^* {S}^+_{i,\eps}$.  In particular, since  $\partial^* {S}^+_{i,\eps}$  is a rectifiable Jordan curve and $i \in \mathcal{S}^{\rm big}_\eps$ we have that 
\begin{equation}\label{eq: smth}
\mathcal{H}^1( \partial^* {S}^+_{i,\eps} \cap  B_{3\eps\eta}(x)) \ge {3\eps\eta}.
\end{equation}\EEE 
Then, by the Besicovitch covering theorem  there exists a countable subcollection of balls $B_{3\eps\eta}(x)$, $x \in \mathcal{X}_{i,\eps}$, which cover $ \partial^* \EEE \tilde{S}^+_{i,\eps}$ up to an  $\mathcal{H}^1$-negligible \EEE set such that each $y \in \R^2$ is contained in at most~$N_2$ different balls. Clearly, $B_{6\eps\eta}(x)$, $x \in \mathcal{X}_{i,\eps}$, then covers $\tilde{S}^+_{i,\eps}$ up to a set of negligible $\mathcal{L}^2$-measure. \EEE Therefore,  by~\eqref{eq: smth} \EEE we compute 
\begin{align*}
\mathcal{L}^2\big(\tilde{S}^+_{i,\eps}\big) \le \sum_{x \in \mathcal{X}_{i,\eps}} \EEE \mathcal{L}^2(B_{ 6 \EEE \eps\eta}(x)) \le  \sum_{x \in \mathcal{X}_{i,\eps}} \EEE  12 \EEE \pi{\eps\eta} \mathcal{H}^1(B_{3\eps\eta}(x)  \cap \partial^* {S}^+_{i,\eps}) \le  12 \EEE N_2\pi{\eps\eta} \mathcal{H}^1 (\partial^* {S}^+_{i,\eps}). 
\end{align*} 
This shows \eqref{eq: iso2} for the sets $S^+_{i,\eps}$. The proof  of~\eqref{eq: iso2-2} for the set $F^+_\eps$ works in the same way once we notice that $\mathcal{H}^1(\partial^* F_\eps^+)  \geq\rho\EEE> {3\eps\eta}$ for $\eps$ sufficiently small.\EEE 

Eventually, we conclude the proof of  \eqref{eq: thickening-general-1} \EEE by combining   \eqref{eq: length bound}--\eqref{eq: iso2}. \EEE

\noindent \emph{\UUU Step 6: \EEE Proof of  \eqref{eq: thickening-general-2}.\EEE} We recall that $\mu =  (u^+ - u^-)/|u^+-u^-|$ satisfies $\mu \cdot e_1 < 0$ and let $\tau:= |u^+-u^-|$ for brevity. We also write $\Lambda_\eps := (\partial^* G^+_\eps \cup \partial^* G^-_\eps) \cap Q_\rho$ and recall  that \EEE the curve $\Gamma_\eps$ (without its endpoints) is contained in $\Lambda_\eps$.  We also recall the notation in \eqref{eq: slicing-not}--\eqref{eq: slicing-not2}.  The main point of this step is to prove the estimate
\begin{align}\label{eq: main7}
\mathcal{H}^1\big( V^\mu_\eps \big) \ge  \frac{\rho}{ 8 } |\mu \cdot e_1|, \ \text{where } V^\mu_\eps:=\big\{ w \in \Pi^\mu\colon  \mathcal{H}^0( (\Gamma_\eps \cap Q_{\rho-2\tau\eps})^\mu_w) = 1 \text{ and } \mathcal{H}^0(( \Lambda_\ep)^\mu_w) = 1  \big\}. 
\end{align}
The set $V^\mu_\eps$ (see Figure~\ref{f:6}) corresponds to the vectors $w \in \Pi^\mu$ such that there exists a unique $t_w \in \R$ with   $w + t\mu \in \Gamma_\eps$, $w + t\mu \in G^+_\eps$ for $t \in (t_w-\tau\eps ,t_w)$, and  $w +  t_w \EEE \mu \in G^-_\eps$ for $t \in (t_w,t_w+\tau\eps)$. This estimate implies  \eqref{eq: thickening-general-2}. \EEE In fact, for $w \in \Pi^\mu$, the two sets
$$ \lbrace    w + t\mu + \eps u^+   \colon   t\in (t_w-\tau\eps ,t_w) \rbrace \quad \text{ and } \quad  \lbrace    w + t\mu + \eps u^-  \colon   t\in (t_w ,t_w + \tau\eps) \rbrace   $$
coincide, where we used that $u^+-u^- = \tau \mu$. This implies 
\begin{align*}
\mathcal{L}^2\big(  (\eps u^- +   G^-_\eps) \cup (\eps u^+ +   G^+_\eps)    \big) \le \mathcal{L}^2(G^+_\eps) + \mathcal{L}^2(G^-_\eps) - \eps \tau \mathcal{H}^1(V^\mu_\eps)\,.  
\end{align*}

\begin{figure}[h!]
\begin{tikzpicture}
\draw[black, very thick] (-0.6,-0.6) rectangle (4,4);
\draw[black, thick] (0, 0) rectangle (3.4, 3.4);
\draw[black, very thick] (3, -3) -- (7, 1);
\draw[black, very thick, ->] ( 6.5, 0.5) -- (5.8, 1.2);
\node at (6, 1.4) {\large{$\mu$}};
\node at (2.7, -2.8) {\large{$\Pi^{\mu}$}};
\node at (4, -0.9) {\large{$Q_{\rho}$}};
\node at (3.3, -0.3) {\large{$Q_{\rho - 2\tau\eps}$}};
\draw[thick, smooth] plot coordinates
            {
                (1.7, -0.6)
                (1.9, -0.3)
                (1.6, 0.4)
                (1.8, 0.9)
                (2, 1.5)
                (1.75, 1.9)
                (2, 2.4)
                (1.9, 2.85)
                (1.5, 3.3)
                (1.7, 4) 
            };
            \node at (2.2, 2) {\large{$\Gamma_{\eps}$}};
\draw[thick, smooth] plot coordinates
{ (2.4, 3.2) (3.2, 3.3) (2.6, 2.6) (2.4, 3.2)};
\node at (3.1, 2.6) {\large{$\Theta_{\eps}$}};
 \draw[black, thick] (0.8,0.5) ellipse (6mm and 2mm);
 \node at (0.43, 0.93) {\large{$\Theta_{\eps}$}};
 \draw[black] (3.965,-2.035) -- (-1.035,  2.965);
 \draw[black] (5.455, -0.545) -- (0.455, 4.455);
 \draw [decorate,decoration={brace, amplitude=5pt}, xshift= 3pt,yshift= -3pt]
(5.455, -0.545) -- (3.965,-2.035)  node [black,midway,yshift=-0.4cm, xshift= 0.4cm] { \large{$V^{\mu}_{\eps}$}};
\end{tikzpicture}
\caption{Visualization of $V^{\mu}_{\eps}$. Here, $\Theta_\eps = \Lambda_\eps \setminus \Gamma_\eps$.} \label{f:6}
\end{figure}
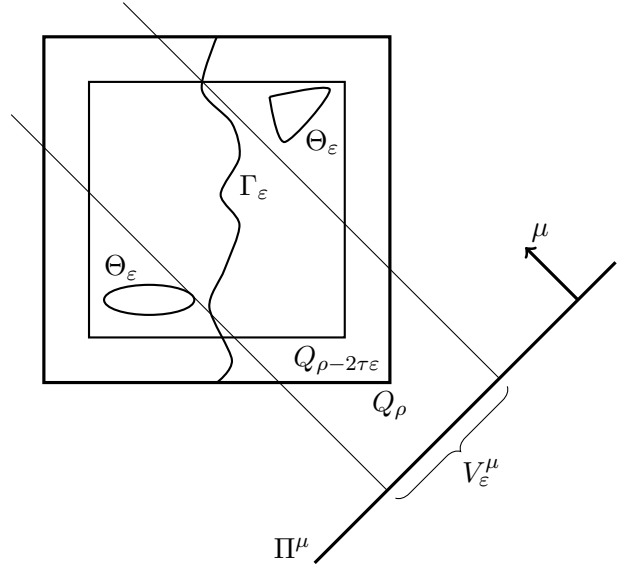


We are hence left to prove  \eqref{eq: main7}. To this end, we recall the definition of $\delta$ and $\lambda$ in \eqref{eq: delta}--\eqref{eq: lambda}. We define $\Gamma'_\eps:= \lbrace x\in \Gamma_\eps\colon \, |\nu_{\Gamma_\eps}(x) \cdot e_1| \ge 1 - \lambda \rbrace$, where $\nu_{\Gamma_\eps}$ denotes \ZZZ a \EEE unit normal vector of the curve $\Gamma_\eps$. We start by observing that 
\begin{align}\label{eq: gamma'}
\mathcal{H}^1\big(\Gamma_\eps \setminus \Gamma'_\eps\big) \le 7\rho\eta/\lambda \le \rho \delta.
\end{align}
Indeed, by the area formula,  by \eqref{eq: length bound-1}, and by \EEE the fact that $\Gamma_\eps$ connects $(-\frac{\rho}{2}, \frac{\rho}{2}) \times \lbrace \frac{\rho}{2} \rbrace$ with $(-\frac{\rho}{2}, \frac{\rho}{2}) \times \lbrace -\frac{\rho}{2} \rbrace$ we calculate
\begin{align*}
\rho &\le \int_{\Pi^{e_1}}  \mathcal{H}^0((\Gamma_\eps)^{e_1}_w)      \, {\rm d}\mathcal{H}^1(w) =     \int_{\Gamma_\eps} |\nu_{\Gamma_\eps} \cdot e_1| \, {\rm d}\mathcal{H}^1 \le \mathcal{H}^1(\Gamma_\eps')+ (1-\lambda)\mathcal{H}^1(\Gamma_\eps \setminus \Gamma'_\eps) \\ &  = \mathcal{H}^1(\Gamma_\eps) - \lambda\mathcal{H}^1(\Gamma_\eps \setminus \Gamma'_\eps)  \le \rho +7\rho\eta  -  \lambda\mathcal{H}^1(\Gamma_\eps \setminus \Gamma'_\eps).
\end{align*}
This along with $7\eta/\lambda \le \delta$ (see \eqref{eq: eta}) shows \eqref{eq: gamma'}. Now, \eqref{eq: gamma'} and the definition of $\Gamma'_\eps$ particularly imply that
\begin{align}\label{eq: supi}
\sup\nolimits_{x_1,x_2 \in \Gamma_\eps} |(x_1 - x_2) \cdot e_1  | \le \mathcal{H}^1(\Gamma'_\eps) \sqrt{1 - \UUU (1-\lambda)^{2} \EEE} + \mathcal{H}^1(\Gamma_\eps \setminus \Gamma'_\eps) \le 2\delta\rho, 
\end{align}
where in the last step we also used  \eqref{eq: lambda-2} and \eqref{eq: length bound-1}. \EEE Thus, we can choose two points $z^\eps_+,z^\eps_- \in \Gamma_\eps$ with $z_\pm^\eps \cdot e_2 = \pm (\frac{\rho}{2} - \tau \eps)$ such that the segment connecting $z^\eps_+$ and $z^\eps_-$, denoted by $\sigma^\eps$, satisfies $\pi_\mu(\sigma^\eps) \YYY \le \EEE \pi_\mu(\Gamma_\eps \cap Q_{\rho-2\tau\eps})$. (Recall notation \eqref{eq: slicing-not2}.) Clearly, $\rho - 2\tau\eps \le \mathcal{H}^1(\sigma^\eps) \le \mathcal{H}^1(\Gamma_\eps) \le  \rho + 7\eta\rho$ by  \eqref{eq: length bound-1} \EEE and $\frac{|\nu_{\sigma^\eps} \cdot e_2|}{|\nu_{\sigma^\eps} \cdot e_1|} \le \frac{2\delta \rho}{\rho-2\tau\eps } \le 3\delta$ by  \eqref{eq: supi}, provided that $\eps$ is small enough. By \eqref{eq: delta},  the latter also yields $|\nu_{\sigma^\eps} \cdot e_1| \ge  \frac{1}{\sqrt{1+9\delta^2}} \ge 1-\delta$.
For $\eps$ sufficiently small, this gives
\begin{align}\label{eq: lasti}
\mathcal{L}^1\big( \pi_\mu(\Gamma_\eps \cap Q_{\rho-2\tau\eps}) \big) & \YYY \ge \EEE \mathcal{L}^1\big( \pi_\mu( \sigma^\eps) \big) = \mathcal{L}^1(\sigma^\eps) |\mu \cdot \nu_{\sigma^\eps}| \ge (\rho - 2\tau\eps) \big( |\mu \cdot e_1| |\nu_{\sigma^\eps} \cdot e_1| - |\nu_{\sigma^\eps} \cdot e_2|         \big)\notag \\&\ge  (\rho - 2\tau\eps) |\nu_{\sigma^\eps} \cdot e_1| \big( |\mu \cdot e_1|  - 3\delta        \big) \ge \rho |\mu \cdot e_1| -4\rho\delta \ge \frac{3\rho}{4} |\mu \cdot e_1|,
\end{align}
 where the last step follows from \eqref{eq: delta}. By  \eqref{eq: length bound-1}, \EEE \eqref{eq: gamma'},   the inclusion \EEE $\Gamma_\eps \subset \Lambda_\eps$  (up to the endpoints), and the fact that $\mathcal{H}^1(\Gamma_\eps) \ge \rho$ we find $\mathcal{H}^1(\Lambda_\eps \setminus \Gamma'_\eps) \le (7\eta+\delta)\rho$, where  $\Lambda_\eps = (\partial^* G^+_\eps \cup \partial^* G^-_\eps) \cap Q_\rho$. Recalling again the definition of $\Gamma'_\eps$ and using $\mathcal{H}^1(\Gamma_\eps') \le \rho + 7\eta\rho$ (see  \eqref{eq: length bound-1}), \EEE  we get by  \eqref{eq: lambda-1} \EEE and  the area formula
\begin{align}\label{eq: lasti2}
\int_{\Pi^{\mu}}  \mathcal{H}^0\big( (\Lambda_\eps)^{\mu}_w \big)      \, {\rm d}\mathcal{H}^1(w) & =     \int_{\Lambda_\eps} |\nu_{\Lambda_\eps} \cdot \mu| \, {\rm d}\mathcal{H}^1 \le  (|\mu\cdot e_1| +  \delta) \mathcal{H}^1(\Gamma'_\eps)     +                  \mathcal{H}^1(\Lambda_\eps \setminus \Gamma'_\eps) \notag \\
&   \le \rho |\mu \cdot e_1| + 21\eta\rho+ 2\rho \delta \le \rho |\mu \cdot e_1| + 4\rho\delta \le  \frac{5\rho}{4} |\mu \cdot e_1|,
\end{align} 
where in the last steps we used \eqref{eq: delta} and \eqref{eq: eta}. Here, $\nu_{\Lambda_\eps}$ denotes \ZZZ  a \EEE unit normal of $\Lambda_\eps$. Consequently, by \eqref{eq: lasti}--\eqref{eq: lasti2} and the fact that $\Gamma_\eps \subset \Lambda_\eps$ (up to the endpoints) we conclude
\begin{align*}
\mathcal{L}^1\Big( \big\{ w\colon  \mathcal{H}^0( (\Gamma_\eps \cap Q_{\rho-2\tau\eps})^\mu_w) = 1 \text{ and } \mathcal{H}^0( (\Lambda_\eps)^\mu_w) = 1  \big\}  \Big) \ge  \frac{\rho}{ 8 } |\mu \cdot e_1|.
\end{align*}
This shows \eqref{eq: main7} and concludes the proof  of the theorem. \EEE 
   \end{proof}

\begin{proof}[Proof of Theorem \ref{thm: CNCC2}]
As \eqref{eq: CC} is a local condition, it is enough to prove the statement on any Lipschitz set $\Omega_u \subset \subset \Omega' \setminus E_u$ with $\mathcal{H}^1(J_u \cap \partial \Omega_u) = 0$. 
In view of Theorem \ref{thm: main} (applied on $\Omega_u$) and Definition \ref{def:conv}, in particular \eqref{eq: first conditions-3}, it suffices to prove that \eqref{eq: assu-u-2} holds for the rescaled displacements $u_\eps$ defined in \eqref{eq: modifica} (for $\Omega_u$ in place of $\Omega'$). To this end, we recall that in the proof of the $\Gamma$-liminf inequality in \cite{higherordergriffith}, see particularly \cite[(4.16)ff.]{higherordergriffith}, it has been shown that
\begin{align*}
\liminf_{\eps \to 0} \frac{1}{\eps^2} \int_{\Omega'} W(\nabla y_\eps)  \,{\rm d} x  \ge \liminf_{\eps \to 0} \int_{\Omega'} \chi_\eps \frac{1}{2}Q(e(u_\eps))  \,{\rm d} x  \ge   \int_{\Omega'}  \frac{1}{2}Q(e(u)) \,{\rm d} x \,,
\end{align*}
 where $(\chi_\eps)_\eps$ is a sequence of indicator functions satisfying $\chi_\eps \to 1$ in measure on  $\Omega'$. More specifically, 
$\chi_\eps(x) = \chi_{[0,\eta_\eps)}(|\nabla u_\eps|(x))$ for a sequence $\eta_\eps \to  + \infty $ with $\eps^{1-\gamma}\eta_\eps \to +\infty$.  By the same argument, for any open $A \subset \Omega'$ and any second sequence $(\bar{\chi}_\eps)_\eps$ of indicator functions with  $\bar{\chi}_\eps \to 1$ in measure on  $A$ it still holds
\begin{align}\label{eq: newproof1}
\liminf_{\eps \to 0} \frac{1}{\eps^2} \int_{A} \bar{\chi}_\eps W(\nabla y_\eps) \,{ d} x \EEE \ge \liminf_{\eps \to 0} \int_{A} \chi_\eps \bar{\chi}_\eps \frac{1}{2}Q(e(u_\eps))  \,{\rm d} x \EEE \ge   \int_{A}  \frac{1}{2}Q(e(u)) \,{\rm d} x \EEE\,.
\end{align}
By \eqref{eq: newproof1} for $\bar{\chi}_\eps  \equiv 1$ and $A \in \lbrace \Omega_u, \Omega' \setminus \overline{\Omega_u} \rbrace$, since  $\mathcal{E}_\eps(y_\eps) \to \mathcal{E}(u)$, and  $\mathcal{H}^1(J_u \cap \partial \Omega_u) = 0$, by~\eqref{eq: the main convergence-3} we have 
\begin{align}\label{eq: contraaa}
 \lim_{\eps \to 0}\frac{1}{\eps^2} \int_{\Omega_u} W(\nabla y_\eps)  \,{\rm d} x \EEE = \int_{\Omega_u}  \frac{1}{2}Q(e(u))\, {\rm d} x \EEE\,, \quad \quad \lim_{\eps \to 0} \mathcal{H}^1(J_{u_\eps} \cap \Omega_u) = \mathcal{H}^1(J_u \cap \Omega_u). 
 \end{align}
 It remains to show $\Vert e(u_\eps) \Vert_{L^2(\Omega_u)} \to \Vert e(u) \Vert_{L^2(\Omega_u)}$. To see this, we will use an argument based on equiintegrability, related to the one in \cite[Proof of Theorem 2.3]{Schmidt:08}.
 
As a preliminary step, we check that  the sequence $g_\eps\colon \Omega_u \to \R$ given by $g_\eps := \frac{1}{\eps^2} \dist^2(\nabla y_\eps,SO(2))$ is equiintegrable. In fact, if the statement were wrong, we would get
$$\lim_{M \to \infty} \limsup_{\eps \to 0} \int_{\lbrace g_\eps > M \rbrace } g_\eps  \,{\rm d} x   \ge \kappa $$
for some $\kappa >0$. Then, by a diagonal argument we can choose a sequence $(M_\eps)_\eps$ with $M_\eps \to +\infty$ such that
\begin{align}\label{eq: kappa}
\liminf_{\eps \to 0} \int_{\lbrace g_\eps > M_\eps \rbrace } g_\eps  \,{\rm d} x \EEE \ge \kappa. 
\end{align}
Define $\bar{\chi}_\eps := \chi_{\lbrace g_\eps \le M_\eps\rbrace }$, and note that  $\bar{\chi}_\eps \to 1$ in measure on  $\Omega_u$ by $\sup_{\eps > 0} \mathcal{E}_\eps(y_\eps) <+\infty$, \eqref{assumptions-W-3}, and $M_\eps \to + \infty$. Thus, by using  \eqref{assumptions-W-3}, \eqref{eq: newproof1} for $A= \Omega_u$, and \eqref{eq: kappa} we calculate
\begin{align*}
\liminf_{\eps \to 0} \frac{1}{\eps^2} \int_{\Omega_u} W(\nabla y_\eps) \,{\rm d} x \EEE & \ge \liminf_{\eps \to 0}  \Big(\frac{1}{\eps^2} \int_{\Omega_u} \bar{\chi}_\eps W(\nabla y_\eps) \,{\rm d} x \EEE + \int_{\Omega_u} (1-\bar{\chi}_\eps) cg_\eps  \,{\rm d} x \EEE \Big)
\\
&
 \ge  \int_{\Omega_u}  \frac{1}{2}Q(e(u)) \,{\rm d} x \EEE + c\kappa.
\end{align*}
This, however, contradicts \eqref{eq: contraaa}, and shows that $g_\eps$ is equiintegrable on $\Omega_u$.

Next, we show that $|e(u_\eps)|^2$ is equiintegrable on $\Omega_u$. To this end, by using the linearization formula $|{\rm sym}(F -\Id)| =  \dist(F,SO(d)) + {\rm O} (|F- \Id|^2)$ (see \cite[(4.12)]{higherordergriffith}) and \eqref{eq: first conditions-3.5},  we get for each $x \in \Omega_u$ satisfying $\dist(\nabla y_\eps(x),SO(2)) \le 1$ that 
\begin{align*}
|e(u_\eps)(x)|^2 &\le Cg_\eps(x) + C\eps^{-2}|\nabla y^{\rm rot}_\eps(x) - \Id|^4 \le Cg_\eps(x) + C\eps^{-2}\big( \eps^{4\gamma} + \dist^4(\nabla y_\eps(x),SO(2)) \big) \\&\le  Cg_\eps(x) + C\eps^{-2}\dist^2(\nabla y_\eps(x),SO(2)) + C \le C g_\eps(x) + C, 
\end{align*}
where we used $\gamma \ge \frac{1}{2}$. On the other hand, if $\dist(\nabla y_\eps(x),SO(2)) > 1$, we easily find
$$|e(u_\eps)(x)|^2 \le \eps^{-2}  |\nabla y^{\rm rot}_\eps(x) - \Id|^{2}  \le C\eps^{-2}\dist^2(\nabla y_\eps(x),SO(2)) = Cg_\eps$$
for a sufficiently large universal constant $C>0$. Combining both estimates, we get that $|e(u_\eps)|^2$ is equiintegrable since $g_\eps$ is equiintegrable. 

Moreover, \eqref{eq: first conditions-3} implies 
\begin{align}\label{eq: smalliseti}
\lim_{\eps\to 0 }\mathcal{L}^2(\lbrace |\nabla u_\eps| \ge \eta_\eps \rbrace \cap \Omega_u ) = \lim_{\eps\to 0 }\mathcal{L}^2(\lbrace |\nabla y^{\rm rot}_\eps- \Id| \ge \eps \eta_\eps \rbrace \cap \Omega_u )   = 0,
\end{align}
where we used that  $\eps^{1-\gamma}\eta_\eps \to +\infty$.  We now conclude as follows. By \eqref{eq: newproof1} for $A= \Omega_u$ and $\bar{\chi}_\eps \equiv 1$,  by the equiintegrability of $|e(u_\eps)|^2$, and \eqref{eq: smalliseti} we get 
$$\liminf_{\eps \to 0} \frac{1}{\eps^2} \int_{\Omega_u}   W(\nabla y_\eps)  \,{\rm d} x \EEE\ge \liminf_{\eps \to 0} \int_{\Omega_u \cap \lbrace |\nabla u_\eps| < \eta_\eps \rbrace }  \frac{1}{2}Q(e(u_\eps)) \,{\rm d} x \EEE = \liminf_{\eps \to 0} \int_{\Omega_u}  \frac{1}{2}Q(e(u_\eps)) \,{\rm d} x \EEE\,. $$
This along with \eqref{eq: the main convergence-2}, \eqref{eq: contraaa}, and the fact that $Q$ is positive definite on $\R^{2 \times 2}_{\rm sym}$ implies 
$$\int_{\Omega_u}  \frac{1}{2}Q(e(u))  \,{\rm d} x \EEE=  \lim_{\eps \to 0}\frac{1}{\eps^2} \int_{\Omega_u} W(\nabla y_\eps)  \,{\rm d} x \EEE\ge  \liminf_{\eps \to 0} \int_{\Omega_u}  \frac{1}{2}Q(e(u_\eps))  \,{\rm d} x \EEE \ge\int_{\Omega_u}  \frac{1}{2}Q(e(u)) \,{\rm d} x \EEE\,.$$
This yields convergence of the linearized energies which together with weak convergence shows the strong convergence $\Vert e(u_\eps) \Vert_{L^2(\Omega_u)} \to \Vert e(u) \Vert_{L^2(\Omega_u)}$. This concludes the proof. 
\end{proof}

\EEE

\section{Proof of Theorem \ref{th: recovery}}
\label{s:th: recovery}

This section is devoted to the proof of Theorem \ref{th: recovery}. We start by a preliminary   approximation \EEE result which allows us to strengthen the contact condition. 

\begin{lemma}[Stronger contact condition]\label{lemma: contact}
Let $\Omega \subset \Omega' \subset \R^2$ be bounded Lipschitz domains  satisfying~\eqref{eq: density-condition2}--\eqref{eq: density-condition}. \EEE Given $ h \EEE \in W^{r,\infty}(\Omega;\R^2)$ for $r \in \N$, let $u \in GSBD^2_{ h \EEE}(\Omega')$  satisfy \EEE \eqref{eq: CC}. Then, there exist sequences $(\tau_n)_n$  in \EEE $(0,+\infty)$ and   $(u_n)_n$  in \EEE $GSBD^2_{ h } (\Omega')$ such that
\begin{subequations}
\label{eq:lemma6.1}
\begin{align}
& \vphantom{\lim_{n \to \infty}} \ u_n \to  u  \text{ in measure on } \Omega',\label{eq:lemma6.1-1}\\
 &  \lim_{n \to \infty} \, \Vert e(u_n) - e(u) \Vert_{L^2(\Omega')}  = \EEE 0,\label{eq:lemma6.1-2}\\
 &   \lim_{n \to \infty} \, \mathcal{H}^{1}(J_{u_n})  = \EEE \mathcal{H}^{1}(J_u),\label{eq:lemma6.1-3}\\
 &  \lim_{n \to \infty} \, \mathcal{H}^1\big(\lbrace x \in J_{u_n}\colon \,   [u_n](x) \cdot \nu_{ u_{n} \EEE} (x) \le 2\tau_n \rbrace \big) = 0. \label{eq:lemma6.1-4}
\end{align}
\end{subequations}
\end{lemma}

\begin{proof}
Fix $0<\theta\le \frac{1}{2}$. It suffices to construct a function $\bar{u} \in GSBD^2_{ h \EEE} (\Omega')$  such that 
\begin{align}\label{eq: all-propi}
\Vert u - \bar{u} \Vert_{L^\infty(\Omega')} + \Vert e(u) - e(\bar{u}) \Vert_{L^2(\Omega')} + \mathcal{H}^1(J_u \triangle J_{\bar{u}}) \le c(1+\mathcal{H}^1(J_u))\theta
\end{align}
for some universal $c>0$, and such that  for some $\bar{\tau} >0$ we have
\begin{align}\label{eq: all-propi2}
\mathcal{H}^1\big(\lbrace x \in J_{\bar{u}}\colon \,    [\bar u] \EEE(x) \cdot \nu_{  \bar u \EEE }(x) \le  2\bar{\tau} \rbrace \big) \le c(1+\mathcal{H}^1(J_u))\theta.
\end{align}
Then the result follows by considering a sequence $(\theta_n)_n$ converging to $0$.

We start by using the fact that  $J_u$ is  countably $\mathcal{H}^1$-rectifiable: 
arguing as in, e.g., \cite[Proof of Theorem~2]{Chambolle:2004} or \cite[Proof of Theorem~1.1]{Crismale2}, we infer  that for $\mathcal{H}^1$-a.e.\ $x_0 \in J_u$ there exist  the approximate unit normal $\nu_{u}(x_{0}) \in \mathbb{S}^{1}$ to $J_{u}$ at $x_{0}$, a positive number $\overline{\rho}(x_{0}) \in (0, \theta^3)$, and a curve~$\Gamma_{x_0}$ such that  the following properties hold: $\Gamma_{x_{0}}$ is the graph of a $C^1$ and Lipschitz function with~$x_0 \in \Gamma_{x_0}$, for every $\rho< \overline{ \rho}(x_0)$ the curve $\Gamma_{x_0}\cap B_\rho(x_0)$ separates $B_\rho(x_0)$ in two open connected components $B_\rho^{\Gamma,\pm}(x_0)$, and  
\begin{subequations}
\label{1104201859}
\begin{align}
 &   \mathcal{H}^1\big(J_u  \cap {B_\rho(x_0)} \big) \ge (1-\theta)2\rho, \label{1104201859-i} \\
&  \mathcal{H}^1\big((J_u \triangle \Gamma_{x_0}) \cap {B_\rho(x_0)}\big) \le \theta\rho,\label{1104201859-ii}\\
&      \mathcal{H}^1\big(  J_u   \cap (B_\rho(x_0) \setminus B_{(1-\theta)\rho}(x_0)) \big)  \le 3\theta \rho, \quad  \mathcal{H}^1\big(  \Gamma_{x_0} \cap (B_\rho(x_0) \setminus B_{(1-\theta)\rho}(x_0)) \big)  \le 3\theta \rho \label{1104201859-iii}\\
&  \nu_{\Gamma_{x_0}} \cdot \nu_u(x_0) > 1 - \theta \quad \text{ on $\Gamma_{x_0} \cap B_\rho(x_0)$},\label{1104201859-iv}
\end{align}
\end{subequations}
where $\nu_{\Gamma_{x_0}}(x)$ denotes the outer normal to $\partial B_\rho^{\Gamma,-}(x_0) \cap \Gamma_{x_0}$ at~$x$. \EEE Moreover,  for  each  $\rho< \overline{ \rho}(x_0)$, we have
\begin{align}\label{1104201859-2}
B_\rho(x_0) \subset \Omega \quad \text{ if $  x_0 \EEE \in J_u \cap \Omega$}, \quad \quad \quad B_\rho^{\Gamma,+}(x_0) \subset \Omega \quad \text{ if $  x_0 \EEE \in J_u \cap \partial \Omega$,} 
\end{align}
where in the latter case $\nu_u(x_0) $ corresponds to the inner unit normal at $x_0 \in \partial \Omega$.

 For $x \in J_{u}$ and $\rho \in (0, \overline{\rho}(  x \EEE))$, the balls $B_{\rho}(x)$ \EEE are a fine cover of  $J_u$ up to a set of negligible $\mathcal{H}^1$-measure. By applying  Besicovitch's covering theorem to this fine cover, we find a finite number of pairwise disjoint balls ${B_{\rho_i}(x_i)}$, $i=1,\ldots,m$,  such that $x_i \in J_u$, \eqref{1104201859}--\eqref{1104201859-2} hold, and \EEE 
\begin{align}\label{eq: outer jump}
\mathcal{H}^1\Big(J_u \setminus \bigcup_{i=1}^m \EEE B_{\rho_i}(x_i) \Big) \le \theta.
\end{align}
We consider $\varphi \in C^\infty_c( B_1(0))$ with $\varphi \equiv 1$ on $B_{1-\theta}(0)$, $\Vert \varphi \Vert_\infty = 1$, and $\Vert \nabla \varphi \Vert_\infty \le c\theta^{-1}$ for some $c>0$. We define the function
\begin{align}\label{eq: ubardef}
\bar{u}(x)  \coloneqq \EEE u(x) +  \sum_{i=1}^m \rho_i  \,   \varphi\big(  (x-x_i)/\rho_i   \big) \, \nu_u(x_i) \,              \chi_{B_{\rho_i}^{\Gamma,+}(x_i)}(x) \quad   \text{ for $x \in \Omega'$}. \EEE 
\end{align}
We start by observing that \eqref{1104201859-2} implies $\bar{u} = u$ on $\Omega'\setminus \overline{\Omega}$, and thus $\bar{u} \in GSBD^2_{ h \EEE} (\Omega')$. Let us now check \eqref{eq: all-propi}--\eqref{eq: all-propi2}. First, since $\Vert \varphi \Vert_\infty \le 1$ and the balls $B_{\rho_i}(x_i)$ are pairwise disjoint,  we clearly have 
\begin{equation}
\label{eq: baru1}
\Vert u - \bar{u} \Vert_{L^\infty(\Omega')} \le \max_i \rho_i \le \theta^3 \le \theta, 
\end{equation}
 where we used $\max_i \rho_i \le \theta^3$.  By a change of variables,   \eqref{1104201859-i}, and  $\Vert \nabla \varphi\Vert_\infty \le c\theta^{-1}$ we further get 
\begin{align}
\label{eq: baru2}
\Vert e( \bar u ) - e(u) \Vert^2_{L^2(\Omega')} &  \le \sum_{i=1}^m  \EEE   \int_{B_{\rho_i}^{\Gamma,+}(x_i)} |\nabla \varphi \big(  (x-x_i)/\rho_i   \big)    |^2 \, {\rm d} x 
\\ 
&
\le \sum_{i=1}^m \rho^2_{i} \int_{B_1(0)} |\nabla \varphi(x)   |^2 \, {\rm d} x \nonumber
\\
&
 \le  \frac{\theta^3}{2(1-\theta)} \Vert \nabla \varphi \Vert^2_{L^2(B_1(0))}  \sum\nolimits_{i=1}^m \mathcal{H}^1(J_u \cap B_{\rho_i}(x_i)) \nonumber
 \\ 
 &\vphantom{\sum_{i=1}^{N}} \le  c\theta \mathcal{H}^1(J_u) \nonumber
\end{align}
\EEE
for a universal constant $c>0$. Up to slightly altering the values of~$\rho_i$, we can suppose that $\mathcal{H}^1(J_u \setminus J_{\bar{u}}) = 0$. As $J_{\bar{u}} \setminus J_u \subset \bigcup_{i=1}^m  (\Gamma_{x_i} \setminus J_u) \cap B_{\rho_i}(x_i)$,  \eqref{1104201859-i}--\eqref{1104201859-ii} imply 
\begin{equation}
\label{eq: baru3}
\mathcal{H}^1(J_u \triangle J_{\bar{u}}) \le \sum_{i=1}^m \theta\rho_i \le c\theta \mathcal{H}^1(J_u).
\end{equation}
Combining~\eqref{eq: baru1}--\eqref{eq: baru3} we conclude~\eqref{eq: all-propi}. \EEE

We now show \eqref{eq: all-propi2}.  First, for  $i=1,\ldots,m$ and for $\mathcal{H}^1$-a.e.\  
$x \in  \big( J_{\bar{u}} \cap J_u \cap \Gamma_{x_i}  \big)  \cap B_{(1-\theta)\rho_i}(x_i) $  we find $\nu_{\bar{u}}(x) = \nu_u(x) = \nu_{\Gamma_{x_i}}(x)$.  Thus, by  \eqref{1104201859-iv}, by \eqref{eq: ubardef}, and by~\eqref{eq: CC}, we get  for $\mathcal{H}^{1}$-a.e.~$x \in  \bigcup_{i=1}^{m}\big( J_{\bar{u}} \cap J_u \cap \Gamma_{x_i}  \big)  \cap B_{(1-\theta)\rho_i}(x_i) $ \EEE
\begin{align}\label{eq: longi1}
[\bar{u}](x) \cdot \nu_{\bar{u}}(x) = [{u}](x) \cdot \nu_{{u}}(x) + \big([\bar{u}](x)  - [{u}](x) \big) \cdot \nu_{\Gamma_{x_i}}(x) \ge   \rho_i \EEE \nu_u(x_i) \cdot  \nu_{\Gamma_{x_i}}(x) > (1-\theta)  \rho_i. \EEE
\end{align}
On the other hand, we have
\begin{align*}
\mathcal{H}^1\Big( J_{\bar{u}} \setminus \bigcup_{i=1}^m \big(J_u \cap   \Gamma_{x_i}  \cap B_{(1-\theta)\rho_i}(x_i)                \big) \Big) & \le \mathcal{H}^1\Big(J_u \setminus \bigcup_{i=1}^m B_{\rho_i}(x_i) \Big) + \sum_{i=1}^m \mathcal{H}^1\big(  ( J_{{u}} \triangle \Gamma_{x_i}  )  \cap B_{(1-\theta)\rho_i}(x_i)     \big) \\ & \ \ \  +  \sum_{i=1}^m \EEE \mathcal{H}^1\big(         ( J_{{u}} \cup \Gamma_{x_i}  ) \cap  \big(   B_{\rho_i}(x_i) \setminus B_{(1-\theta)\rho_i}(x_i)      \big)\big). 
\end{align*}
Then, by  \eqref{1104201859-i}--\eqref{1104201859-iii} \EEE and  \eqref{eq: outer jump} we conclude
$$\mathcal{H}^1\Big( J_{\bar{u}} \setminus  \bigcup_{i=1}^m \EEE \big(J_u \cap  \Gamma_{x_i}  \cap B_{(1-\theta)\rho_i}(x_i)                \big) \Big) \le \theta +  \sum_{i=1}^m \EEE 7\theta \rho_i \le \theta + c\theta \mathcal{H}^1(J_u).$$
This along with \eqref{eq: longi1} shows that  \eqref{eq: all-propi2} holds for $\bar{\tau} = \frac{1}{2}(1-\theta)\min_i  \rho_i$.
\end{proof}

We now provide an adaption of the $GSBD^2$-density result stated in Theorem \ref{th: crismale-density2} which guarantees the contact condition up to a part of the jump set with small $\mathcal{H}^1$-measure.

\begin{theorem}[Density with boundary data and contact condition]\label{th: density new}
Let $\Omega \subset \Omega' \subset  \R^2$  be bounded Lipschitz domains satisfying \eqref{eq: density-condition2}--\eqref{eq: density-condition}. Let $\theta>0$, $\tau >0$,  $h \in W^{r,\infty}(\Omega')$ for $r \in \N$, and let $u \in GSBD^2_h(\Omega')$  satisfy \EEE 
\begin{align}\label{eq: good-t}
\mathcal{H}^1\big(\lbrace x \in J_{u}\colon \,   [u](x) \cdot \nu_u(x)  \le  2\tau \rbrace \big) \le \theta.
\end{align}
 Then, there  exist  a sequence of functions $(u_n)_n$  in  $SBV^2(\Omega; \R^2)$, a sequence of neighborhoods $(U_n)_n$ of $\Omega' \setminus \Omega$, and a sequence of neighborhoods $(\Omega_n)_n$ of $\Omega \setminus U_n$  such that  $U_{n}  \subset \Omega'$, $\Omega_{n} \subset \Omega$,  $u_n =  h $ on $\Omega' \setminus \overline{\Omega}$, $u_n|_{U_n} \in W^{r,\infty}(U_n;\R^2)$, $u_n|_{\Omega_n} \in \mathcal{W}(\Omega_n;\R^2)$, and 
\begin{subequations}
\label{eq: dense-boundary-new}
\begin{align}
 & \ \vphantom{\lim_{n\to \infty} } u_n \to  u  \text{ in measure on } \Omega' \text{ as $n \to \infty$},\label{eq: dense-boundary-new-1}\\
 &  \lim_{n\to \infty} \,\Vert e(u_n) - e(u) \Vert_{L^2(\Omega')} = 0,\label{eq: dense-boundary-new-2}\\
 &  \lim_{n \to \infty} \,  \mathcal{H}^{1}(J_{u_n}) = \mathcal{H}^{1}(J_u),\label{eq: dense-boundary-new-3}\\
 & \limsup_{ n \to \infty } \, \mathcal{H}^1\big(\lbrace x \in J_{u_n}\colon \,   [u_n](x) \cdot \nu_{u_n}(x)  \le  \tau \rbrace \big) \le 3\theta. \label{eq: good-t2}
\end{align}
\end{subequations} \EEE
In particular, $u_n \in W^{r,\infty}(\Omega \setminus J_{u_n};\R^2)$ for all $n \in \N$. \EEE 
\end{theorem}

\begin{proof}
Given $\theta>0$ and $u \in GSBD^2_{ h \EEE} (\Omega')$ as in the statement, we apply Theorem \ref{th: crismale-density2}  to \EEE $u$ to obtain an approximating sequence $(u_n)_n \subset  SBV^2(\Omega; \R^2)$ satisfying the properties  \eqref{eq: dense-boundary-new-1}--\eqref{eq: dense-boundary-new-3}.  Note also that it is not restrictive to assume that $\mathcal{H}^1(J_u)>0$. Otherwise, \eqref{eq: good-t2} would follow directly from \eqref{eq: dense-boundary-new-3}. \EEE By defining  $J_{u_n}^{\rm bad}  \coloneqq \EEE \lbrace  x \in J_{u_n}  \colon \,   [u_n](x) \cdot  \nu_{u_n}(x) \le \tau   \rbrace$, we see that  to conclude~\eqref{eq: good-t2} we  need \EEE to show that  
\begin{equation}
\label{eq: bad-theta}
\limsup_{n \to \infty}\, \mathcal{H}^1(J_{u_n}^{\rm bad}) \le 3\theta.
\end{equation}\EEE

 Let us fix  $\zeta>0$ \EEE sufficiently large such that $\mathcal{H}^1(\lbrace x\in J_u\colon \, |[u](x)| <  \zeta^{-1} \EEE \text{ or } |[u](x)| >  \zeta \EEE) \le \theta$  and let us set \EEE  $J_u^{\rm good}  \coloneqq \EEE \lbrace x \in J_{u}\colon \,   [u](x) \cdot \nu_u(x) > 2\tau, \  \zeta^{-1} \EEE \le |[u](x)| \le  \zeta \EEE  \rbrace.$ By \eqref{eq: good-t} we get
\begin{align}\label{eq: good part}
\mathcal{H}^1(J_u \setminus J^{\rm good}_u) \le 2\theta.
\end{align}
Let us also \EEE fix $\lambda \in (0,1)$   such that 
\begin{align}\label{eq: lambda-neu}
\lambda \le \frac{\tau}{2  \zeta \EEE}
\end{align}
and $\eta \in (0,1)$ such that
\begin{align}\label{eq: eta-neu}
\eta < \min\Big\{  \frac{\lambda^2\theta}{  56 \EEE (\lambda^2+1)\mathcal{H}^1(J_u)}, \frac{\tau}{12}, \frac{1}{7 \zeta \EEE}, \theta_0, 10^{-4}  \Big\}, 
\end{align}
 with $\theta_0$ from Proposition \ref{th: kornpoin-sharp}. The choice of~$\lambda$ and $\eta$ will become clear along the proof. \EEE

\noindent \emph{Step 1: Blow-up.} We now introduce a covering of $J_u^{\rm good}$: for $\mathcal{H}^1$-a.e.\ $x \in J^{\rm good}_u$, we find $\nu_u(x) \in \mathbb{S}^1$,  $u^+_{x},u^-_{x} \in \R^2$, and  $0 <\bar{\rho}(x) \le 1$ such that for all  $0 < \rho < \bar{\rho}(x)$ it holds that  
\begin{subequations}
\label{eq: many properties1-neu}
\begin{align}
 &   \!\!\! \vphantom{\int} |\mathcal{H}^1\big(J_u \cap Q^{x}_\rho\big) - \rho| \le \frac{\rho\eta}{2}, \label{eq: many properties1-neu-1}\\
 & \!\!\!  \mathcal{L}^2\Big(\Big\{ y\in Q^{x,+}_\rho\colon  |u(y) - u^+_{x}| > \frac{\eta}{\bar{c}_{\eta^3/2}} \Big\}  \Big) +  \mathcal{L}^2\Big(\Big\{  y\in Q^{x,-}_\rho\colon  |u(y) - u^-_{x}| > \frac{\eta}{\bar{c}_{\eta^3/2}}  \Big\}  \Big)   \le \frac{\rho^2\eta^4}{2},\label{eq: many properties1-neu-2} \\
 & \!\!\!  \int_{Q^{x}_\rho} |e(u)|^2 \, {\rm d}y \le \frac{\rho\eta^2}{2C_\eta \bar{c}^2_{\eta^3/2}}, \label{eq: many properties1-neu-3}
\end{align}
\end{subequations}\EEE
where $Q^{x}_\rho$ denotes the square with sidelength $\rho$ centered at $x$ with two sides parallel to   $\nu_u(x)$, \EEE  and $Q^{x,\pm}_\rho  \coloneqq \EEE Q^{x}_\rho \cap \lbrace \YYY y\colon \, \EEE  \pm (y - x) \cdot \nu_{ u \EEE}(x) >0 \rbrace$. Moreover,   $C_\eta \ge 1$ and $ \bar{c}_{\eta^3/2} \EEE \ge 1$ denote the constants of Proposition \ref{th: kornpoin-sharp} applied for $\theta= \eta$ and  of Lemma \ref{lemma: rigid motion} applied for $\theta =  \eta^3/2 \EEE$, respectively.  We refer to \eqref{eq: many properties1}  and \eqref{eq: many properties2} \EEE above for an analogous  argument.

 For $x \in J^{\rm good}_u$ and for  $\rho < \overline{\rho}(x)$ such that $\mathcal{H}^1\big(J_u \cap \partial Q^{x}_\rho\big) =0$, the squares $Q^{x}_{\rho}$ form \EEE a fine cover of  $J^{\rm good}_u$ up to a set of negligible $\mathcal{H}^1$-measure.  By applying  Besicovitch's covering theorem to this fine cover, we find a finite number of pairwise disjoint  squares \EEE ${Q^{x_i}_{\rho_i}}$, $i=1,\ldots,m$,  such that the centers~$x_{i}$ belong to $J^{\rm good}_u$, \eqref{eq: many properties1-neu} holds, $\mathcal{H}^1\big(J_u \cap \partial Q^{x_i}_{\rho_i}\big) =0$, and \EEE 
\begin{align}\label{eq: outer jump-new}
\mathcal{H}^1\Big(J^{\rm good}_u \setminus  \bigcup_{i=1}^m \EEE Q^{x_i}_{\rho_i} \Big) \le  \frac{\theta}{2}.\EEE
\end{align}
 Arguing as in~\eqref{eq: NC refined2.1}, by  \eqref{eq: dense-boundary-new-2}--\eqref{eq: dense-boundary-new-3} \EEE and the fact that  $\mathcal{H}^1\big(J_u \cap \partial Q^{x_i}_{\rho_i}\big) =0$ we deduce that
 \begin{subequations}
\label{eq: NC refined2}
\begin{align}
 & \vphantom{\bigcup_{i=1}^{m}}   \lim_{n \to \infty} \, \Vert e(u_n) \Vert_{L^2(Q^{x_i}_{\rho_i})}  =  \Vert e(u) \Vert_{L^2(Q^{x_i}_{\rho_i})} \qquad    \text{for all $i=1,\ldots,m$},\label{eq: NC refined2-1}\\
 & \vphantom{\bigcup_{i=1}^{m}} \lim_{n \to \infty} \, \mathcal{H}^{1}(J_{u_n}\cap Q^{x_i}_{\rho_i})  =   \mathcal{H}^{1}(J_{u} \cap Q^{x_i}_{\rho_i}) \qquad \text{for all $i=1,\ldots,m$}, \label{eq: NC refined2-2}\\
&  \lim_{ n \to \infty} \, \mathcal{H}^{1}\Big(J_{u_n}\setminus \bigcup_{i=1}^m Q^{x_i}_{\rho_i}\Big)  =   \mathcal{H}^{1}\Big(J_{u} \setminus \bigcup_{i=1}^m Q^{x_i}_{\rho_i}\Big).\label{eq: NC refined2-3}
\end{align}
\end{subequations} \EEE
 The convergence in~\eqref{eq: NC refined2} along with \eqref{eq: dense-boundary-new-1} and \eqref{eq: many properties1-neu} shows  that for $n$ large enough we have
\begin{subequations}
\label{eq: many properties2-neu}
\begin{align}
&  \vphantom{\int} \mathcal{H}^1\big(J_{u_n} \cap Q^{x_i}_{\rho_i}\big) \le    \rho_i (1 + \eta), \label{eq: many properties2-neu-1}\\
 &  \vphantom{\int} \mathcal{L}^2\Big(\Big\{ x\in Q^{x_i,+}_{\rho_i}\colon \, |u_n - u^+_{x_i}| > \frac{\eta}{\bar{c}_{\eta^3/2}}  \Big\}  \Big) +  \mathcal{L}^2\Big(\Big\{ x\in Q^{x_i,-}_{\rho_i}\colon \, |u_n - u^-_{x_i}| > \frac{\eta}{ \bar{c}_{\eta^3/2}}  \Big\}  \Big)   \le \rho_i^2\eta^4,\label{eq: many properties2-neu-2} \\
 &   \int_{Q^{x_i}_{\rho_i}} |e(u_n)|^2 \, {\rm d}x \le \frac{\rho_i \eta^2}{C^2_\eta \bar{c}_{\eta^3/2}^2}.\label{eq: many properties2-neu-3}
\end{align}
\end{subequations} \EEE
In the following, without further notice, $n$ will always be chosen sufficiently large such that \eqref{eq: many properties2-neu} holds for all $i=1,\ldots,m$.

\noindent \emph{Step 2: Conclusion.} The main step of the proof consists in showing that
\begin{align}\label{eq: main bad step}
\mathcal{H}^1\Big(J^{\rm bad}_{u_n} \cap Q_{\rho_i}^{x_i}    \Big) \le \frac{\theta}{4 \mathcal{H}^1(J_u)} \rho_i \quad \text{for $i=1,\ldots,m$}.
\end{align}
Once we have proved \eqref{eq: main bad step}, the  claim~\eqref{eq: bad-theta} \EEE is achieved as follows:  by applying \eqref{eq: good part}, \eqref{eq: outer jump-new}, and    \eqref{eq: NC refined2-3} \EEE we find
\begin{align*}
\limsup_{n \to \infty}\mathcal{H}^1(J^{\rm bad}_{u_n}) \le \limsup_{n \to \infty}\mathcal{H}^1\Big( J_{u_n} \setminus \bigcup_{i=1}^m Q_{\rho_i}^{x_i}\Big) + \limsup_{n \to \infty}\sum_{i=1}^m\mathcal{H}^1(J^{\rm bad}_{u_n} \cap Q_{\rho_i}^{x_i}) \le \frac{5}{2}\theta + \sum_{i=1}^m  \frac{\theta}{4 \mathcal{H}^1(J_u)} \rho_i. 
\end{align*}
Then, in view of  \eqref{eq: many properties1-neu-1},   the assumption \EEE $\eta \le 1$, and \EEE the fact that the squares $(Q_{\rho_i}^{x_i})_{i=1}^m$ are pairwise disjoint,   we get
$${\limsup_{n \to \infty}\mathcal{H}^1(J^{\rm bad}_{u_n}) \le \frac{5}{2}\theta + \frac{\theta}{4 \mathcal{H}^1(J_u)} \frac{1}{1-\eta/2}  \sum_{i=1}^m \EEE \mathcal{H}^1(J_u \cap Q_{\rho_i}^{x_i}) \le 3 \theta,}  $$
 and the proof of~\eqref{eq: bad-theta} is thus concluded.\EEE

\noindent \emph{Step 3: Proof of \eqref{eq: main bad step}.} It remains to prove \eqref{eq: main bad step}.  Let us fix~$i \in \{1, \ldots, m\}$. After possible rotation and translation, we may suppose \EEE that $x_i =0$ and $\nu_u(x_i) = e_1$, and  we \EEE write $Q_{\rho_i}$ in place of $Q_{\rho_i}^{x_i}$.  

By the choice of $\eta$ in \eqref{eq: eta-neu} and the fact that $ x_i= 0  \in J_u^{\rm good}$ we particularly have $\eta \le 1/(7 \zeta ) \le  |u^+_{x_i} - u^-_{x_i}|/7$. Moreover, \eqref{eq: eta-neu} implies  $\eta < \min\lbrace \theta_0,10^{-4}\rbrace$. Thus, in view of   \eqref{eq: many properties2-neu}, we can apply Proposition \ref{prop: Korn}  to \EEE $v = u_n$  and $ \omega^\pm=u^\pm_{x_i}$  to find two  subsets $D^+_n, D^-_n \subset Q_{\rho_i}$ such that    \EEE
\begin{subequations}
\label{eq: two case4-anwendung}
 \begin{align}
 &  \Vert u_n - u^+_{x_i} \Vert_{L^\infty(D^+_n)} \le 3\eta \quad \quad \text{and} \quad \quad \Vert u_n - u^-_{x_i} \Vert_{L^\infty(D^-_n)} \le 3\eta, \label{eq: two case4-anwendung-1} \\
&   \mathcal{H}^1\Big( \big( (\partial^* D^+_n \cup \partial^* D^-_n) \setminus J_{u_n}\big) \cap Q_{\rho_i} \Big) \le  6\eta{\rho_i},\label{eq: two case4-anwendung-2}
\end{align}
\end{subequations}\EEE
and two curves $\Gamma^{\pm}_n \subseteq \partial^{*} D^{\pm}_n \cap Q_{\rho_i}$ connecting $(-\frac{\rho}{2}, \frac{\rho}{2}) \times \{ - \frac{\rho}{2}\}$ to $(-\frac{\rho}{2}, \frac{\rho}{2}) \times \{ \frac{\rho}{2}\}$.  For simplicity of notation, let us set $\Psi_n  \coloneqq  \Gamma_n^-$. \EEE As observed above, we have $\eta \le |u^+_{x_i} - u^-_{x_i}|/7$. Then   \eqref{eq: two case4-anwendung-1} implies   $J_{u_n} \supset \Psi_n$ up  to an $\mathcal{H}^1$-negligible set. \EEE  We now show that
\begin{subequations}
\label{eq: nocmal alles2}
\begin{align}
& \rho_i \le  \mathcal{H}^1( \Psi_n)  \le {\rho_i}(1+7\eta), \label{eq: nocmal alles2-1} \\
&
 \mathcal{H}^1\big( \Psi_n' \big) \le 7{\rho_i}\eta, \ \ \text{where $ \Psi_n' := \lbrace x \in \Psi_n \colon \,  \nu_{u_n}(x) \cdot e_1  < 0   \rbrace$},\label{eq: nocmal alles2-2}
\end{align}
\end{subequations} \EEE
and where we choose the orientation of $\nu_{u_n}$ such that $\nu_{u_n}$  coincides with the outer normal to $D_n^-$. 
Inequality \eqref{eq: nocmal alles2-1} \EEE follows from \eqref{eq: many properties2-neu-1} and \eqref{eq: two case4-anwendung-2}. As for~\eqref{eq: nocmal alles2-2}, by the fact $\Psi_n$ is a curve connecting $(-\frac{\rho}{2}, \frac{\rho}{2}) \times \{ - \frac{\rho}{2}\}$ to $(-\frac{\rho}{2}, \frac{\rho}{2}) \times \{ \frac{\rho}{2}\}$ we find that 
$$ \Psi_n'   \subset \lbrace x \in \Psi_n \colon \,  t_{\Psi_n} \cdot e_2  < 0   \rbrace,$$
where $t_{\Psi_n}$ denotes the tangent vector of the curve $\Psi_n$. This along with \eqref{eq: nocmal alles2-1}  shows \EEE \eqref{eq: nocmal alles2-2}.  

 We recall the definition of $\lambda$ in \eqref{eq: lambda-neu} and define
$${ \Psi''_n \EEE  \coloneqq \lbrace x \in \Psi_n \colon \,  |\nu_{u_n}(x) - e_1 |\le  \lambda   \rbrace.}$$
Since $  x_{i} = 0  \in J_u^{\rm good}$, we have  $| u^+_{x_{i}} - u^-_{x_{i}} | \le  \zeta$  as well as $( u^+_{x_{i}} - u^-_{x_{i}}  ) \cdot e_1 > 2\tau$, and thus  for each $x \in \Psi''_n$ we deduce from  \eqref{eq: two case4-anwendung-1}, \eqref{eq: lambda-neu}, and the fact that $\eta < \tau/12$ (see \eqref{eq: eta-neu}) that
\begin{align*}
[u_n](x) \cdot \nu_{u_n}(x) & \ge  ( u^+_{x_{i}} - u^-_{x_{i}}  ) \cdot \nu_{u_n}(x)         -6\eta  \\ 
& \ge  ( u^+_{x_{i}} - u^-_{x_{i}} )  \cdot e_1  -  | u^+_{x_{i}} - u^-_{x_{i}}  | \lambda       -6\eta  \ge 2\tau  - \zeta \lambda -6\eta > \tau.
\end{align*} 
This implies that $\Psi''_n \cap J_{u_n}^{\rm bad} = \emptyset$. Based on this, we now derive \eqref{eq: main bad step}. First, we observe that  
\begin{align}\label{eq: gamma---XXX}
\mathcal{H}^1\big(\Psi_n \setminus \YYY  (\Psi'_n  \cup \Psi''_n) \EEE \big) \le  \frac{\YYY 14 \EEE {\rho_i}\eta }{\lambda^2}. \EEE 
\end{align}
Indeed, for $x \in \Psi_n \setminus \YYY  (\Psi'_n  \cup \Psi''_n) \EEE $ we find $\YYY 0 \le \EEE e_1 \cdot \nu_{u_n}(x) =1 - |e_1 - \nu_{u_n}(x)|^2/2 \le 1-\lambda^2/2$ by a simple expansion. Then,  by the area formula and by  \YYY \eqref{eq: nocmal alles2-1} \EEE  we estimate 
\begin{align*}
\YYY {\rho_i} \EEE &\le \int_{\Pi^{e_1}}  \mathcal{H}^0((\Psi_n)^{e_1}_w)      \, {\rm d}\mathcal{H}^1(w) =     \int_{\Psi_n} |\nu_{u_n} \cdot e_1| \, {\rm d}\mathcal{H}^1 \le \mathcal{H}^1(\YYY \Psi'_n \cup \Psi_n'' \EEE )+ (1-\tfrac{\lambda^2}{2})\mathcal{H}^1(\Psi_n \setminus \YYY  (\Psi'_n  \cup \Psi''_n) \EEE ) \\ &  = \mathcal{H}^1(\Psi_n) - \tfrac{\lambda^2}{2}\mathcal{H}^1(\Psi_n \setminus \YYY  (\Psi'_n  \cup \Psi''_n) \EEE )  \le {\rho_i} +7{\rho_i}\eta  -  \tfrac{\lambda^2}{2}\mathcal{H}^1(\Psi_n \setminus \YYY  (\Psi'_n  \cup \Psi''_n) \EEE ),
\end{align*}
 which \EEE yields   \eqref{eq: gamma---XXX}. 

 Since $\Psi''_n \cap J_{u_n}^{\rm bad} = \emptyset$ and $J_{u_n} \cap Q_{{\rho_i}} \supset \Psi_n$, we conclude by  \eqref{eq: many properties2-neu-1}, \YYY \eqref{eq: nocmal alles2}, \EEE and  \eqref{eq: gamma---XXX} that 
\begin{align*}
\mathcal{H}^1\big( J^{\rm bad}_{u_n} \cap Q_{\rho_i} \big) &\le \mathcal{H}^1\Big( (J_{u_n}\setminus \Psi''_n) \cap Q_{\rho_i}   \Big) \le \mathcal{H}^1\Big( (J_{u_n}\setminus \Psi_n) \cap Q_{\rho_i}   \Big)  +  \mathcal{H}^1\big(\Psi_n \setminus \Psi''_n\big) \\
&  =  \mathcal{H}^1\big( J_{u_n} \cap Q_{\rho_i}   \big) -  \mathcal{H}^1(\Psi_n \cap Q_{\rho_i}\EEE)  +  \mathcal{H}^1\big(\Psi_n \setminus \Psi''_n\big) \\&\le {\rho_i}(1+\eta) - \YYY {\rho_i} \EEE +   \frac{\YYY 14 \EEE {\rho_i}\eta }{\lambda^2} + \YYY 7{\rho_i}\eta . \EEE
\end{align*}
Therefore, $\mathcal{H}^1( J^{\rm bad}_{u_n} \cap Q_{\rho_i}) \le \YYY 14 \EEE \eta{\rho_i}(1+1/\lambda^2)$ which by~\eqref{eq: eta-neu} implies  \eqref{eq: main bad step}. This concludes the proof. 
\end{proof}



We close this section with the proof of Theorem \ref{th: recovery}. \YYY Recall the definition in \eqref{eq: boundary-spaces}. \EEE

\begin{proof}[Proof of Theorem \ref{th: recovery}]
Consider $u \in GSBD^2_h  (\Omega') \EEE$  with $h \in W^{2,\infty}(\Omega';\R^d)$ satisfying  \eqref{eq: CC}. \EEE Let $\gamma \in (\frac{2}{3},\beta)$. By Lemma \ref{lemma: contact} and  the definition of the energy in \eqref{rig-eq: Griffith en-lim} we obtain sequences $(\tau_n)_n \subset (0,+\infty)$ and   $(u_n)_n \subset GSBD^2_h(\Omega')$ such that $u_n \to u$ in measure on $\Omega'$, $\mathcal{E}(u_n) \to \mathcal{E}(u)$, and 
$$\theta_n:=\mathcal{H}^1\big(\lbrace x \in J_{u_n}\colon \,   [u_n](x) \cdot \nu_{u_n}(x) \le 2\tau_n \rbrace \big) \to 0 \quad \text{ as $n \to  \infty \EEE$}.$$  
Since the convergence in  Definition \ref{def:conv} allows for diagonal arguments (measure convergence and weak convergence on bounded sets are metrizable), it suffices to 
construct  for every~$u_{n}$ a sequence~$(y^n_\eps)_\eps$ such that $y^n_\eps \rightsquigarrow u_n$ and $\limsup_{\eps \to 0} \mathcal{E}_\eps (y^n_\eps) \le \mathcal{E}(u_n) + 12\theta_n$. Then, since $u_n \to u$ in measure on $\Omega'$ and $\mathcal{E}(u_n) \to \mathcal{E}(u)$, by a diagonal argument and \cite[Theorem 2.7(ii)]{higherordergriffith} we obtain an energy-convergent sequence for~$u$. \EEE  

To simplify notation,  in what follows we drop the index $n$, so that we consider a function $u \in GSBD_h^2(\Omega')$ and two positive parameters $\theta$ and $\tau$ such that \EEE
\begin{align}\label{eq: needed}
  \theta \coloneqq \EEE \mathcal{H}^1\big(\lbrace x \in J_{u}\colon \,   [u](x) \cdot \nu_u(x) \le 2\tau \rbrace \big)  ,
\end{align}
 and we construct a sequence $(y_\eps)_\eps$, $y_\eps \in \mathcal{S}_{\eps,h}$, satisfying   \eqref{e:CN} and  such that  $ y_\eps \rightsquigarrow u$ and
\begin{align}\label{eq: for recov}
 \limsup_{\eps \to 0 } \, \mathcal{E}_{\eps}(y_\eps) \le \mathcal{E}(u)    +12\theta.  
\end{align} \EEE

\noindent \emph{Step 1: Construction of $(y_\eps)_\eps$.} In view of \eqref{eq: needed}, we can apply Theorem \ref{th: density new} to find a sequence $(v_\eps)_\eps \subset GSBV^2_2(\Omega';\R^2)$  such that $v_\eps = h$ on $\Omega' \setminus \overline{\Omega}$, the jump set $J_{v_{\eps}}$ of~$v_{\eps}$ is a finite union of disjoint segments $(S^i_\eps)_{i=1}^{m_\eps}$, $v_\eps \in W^{2,\infty}(\Omega' \setminus J_{v_\eps};\R^2)$, and the following conditions hold:
\begin{subequations}
\label{eq: dense-in-appli}
\begin{align}
 & \ v_\eps \to  u  \text{ in measure on } \Omega' \text{ as $\eps \to 0$},\label{eq: dense-in-appli-1}\\
 &  \lim_{\eps \to 0} \, \Vert e(v_\eps) - e(u) \Vert_{L^2(\Omega')} = 0,\label{eq: dense-in-appli-2}\\
 &   \lim_{\eps \to 0} \,  \mathcal{H}^{1}  (J_{v_\eps})  =   \mathcal{H}^{1}  (J_{u}),\label{eq: dense-in-appli-3}\\
 &   \  \Vert v_\eps \Vert_{L^\infty(\Omega')} +    \Vert \nabla v_\eps \Vert_{L^\infty(\Omega')} +  \Vert \nabla^2 v_\eps \Vert_{L^\infty(\Omega')}  \le \eps^{(\beta-1)/2} \le \eps^{\gamma-1},\label{eq: dense-in-appli-4}\\  
 &  \dist(S_\eps^i,S_\eps^j) \ge 4 \sqrt{\eps} \text{ for all $1\le i < j \le m_\eps$,}\label{eq: dense-in-appli-5}\\
 &  \dist(S_\eps^i, \Omega' \setminus \Omega) \ge 4\sqrt{\eps} \text{ for all $1 \le i \le m_\eps$}, \label{eq: dense-in-appli-6}\\
&  \limsup_{\eps \to 0}  \, \mathcal{H}^1\big(\lbrace x \in J_{v_\eps}\colon \,   [v_\eps](x) \cdot \nu_{v_\eps}(x) \le \tau \rbrace \big) \le 3\theta. \label{eq: good-t2X}
\end{align}
\end{subequations}
Indeed, properties  \eqref{eq: dense-in-appli-1}--\eqref{eq: dense-in-appli-3} and \eqref{eq: good-t2X} \EEE follow directly from  Theorem \ref{th: density new}. Property  \eqref{eq: dense-in-appli-4} \EEE can be achieved by a diagonal argument since the approximations satisfy $v_\eps \in W^{2,\infty}(\Omega' \setminus J_{v_\eps};\R^{ 2 \EEE})$. (Recall $\gamma < \beta < 1$.) Eventually, properties  \eqref{eq: dense-in-appli-5} and \eqref{eq: dense-in-appli-6} \EEE can again be guaranteed by a diagonal argument since the segments  $(S^i_\eps)_{i=1}^{m_\eps}$ are closed,  pairwise disjoint, and do not intersect a neighborhood of $\Omega' \setminus \Omega$.    Moreover,  $v_\eps \in W^{2,\infty}(\Omega' \setminus J_{v_\eps};\R^2)$ also implies $J_{\nabla v_\eps} \subset J_{v_\eps}$.  

Since $v_\eps \in W^{2,\infty}(\Omega' \setminus J_{v_\eps};\R^2)$ and $J_{v_\eps}$ consists of a finite number of segments, by the coarea formula applied on $x \mapsto [v_\eps](x) \cdot \nu_{v_\eps}(x)$ we find $\tau_\eps \in (\tau/2,\tau)$ such that 
\begin{align}\label{eq: vbad}
J_{v_\eps}^{\rm bad} := \lbrace x \in J_{v_\eps}\colon \,   [v_\eps](x) \cdot \nu_{v_\eps}(x) \le \tau_\eps \rbrace
\end{align}
 consists of a finite number of segments $(T^i_\eps)_{i=1}^{n_\eps}$. We cover these segments by pairwise disjoint rectangles $R_\eps^i$, $i=1,\ldots,n_\eps$, of length $\mathcal{H}^1(T^i_\eps)$ and height $\min \lbrace \mathcal{H}^1(T^i_\eps), \sqrt{\eps}\rbrace$ such that~$T^i_\eps$ separates~$R_\eps^i$ into two rectangles of length $\mathcal{H}^1(T^i_\eps)$ and height $\min \lbrace \mathcal{H}^1(T^i_\eps), \sqrt{\eps}\rbrace/2$, as in Figure~\ref{f:3}.
 
 \medskip
 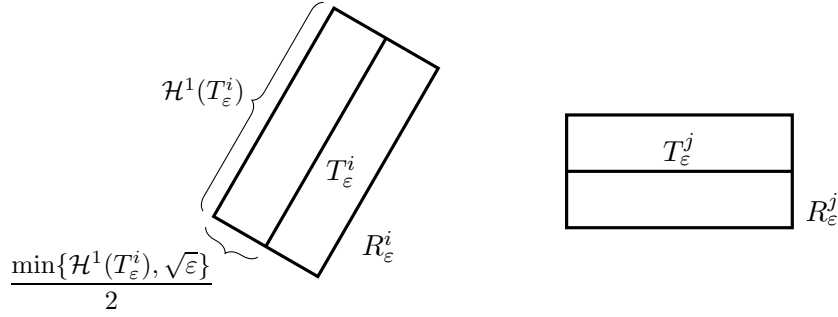
\begin{figure}[h!]
 \begin{tikzpicture}
 \draw[black, very thick, rotate around={60: (-2, -2)}] (-2, -2) -- (1.2, -2);
 \draw[black, very thick, rotate around={60:(-2,-2)}] (-2,-2.8) rectangle (1.2,-1.2);
 \draw [decorate,decoration={brace, amplitude=5pt}, xshift=-4pt,yshift=2pt, rotate around={60:(-2, -2)}]
(-2, -1.2) -- (1.21,-1.2) node [black,midway,xshift=-0.8cm, yshift=0.2cm] { $\mathcal{H}^{1}(T^{i}_{\eps})$};
 \draw [decorate,decoration={brace, amplitude=5pt}, xshift=-3pt,yshift=-4pt, rotate around={60:(-2, -2)}]
(-2, -2) -- (-2,-1.2) node [black,midway,xshift=-1.6cm, yshift=-0.5cm] { $\displaystyle \frac{\min\{\mathcal{H}^{1}(T^{i}_{\eps}), \sqrt{\eps}\}}{2}$};
\node at (-1, -1) {\large{$T^{i}_{\eps}$}};
\node at (-0.5,-2) {\large{$R^{i}_{\eps}$}};
 \draw[black, very thick] (2,-1) -- (5, -1);
 \node at (3.5, -0.7) {\large{$T^{j}_{\eps}$}};
 \draw[black, very thick] (2, -1.75) rectangle (5, -0.25);
 \node at (5.4, -1.5) {\large{$R^{j}_{\eps}$}};
 \end{tikzpicture}
 \caption{The rectangles $R^{i}_{\eps}$ and $R^{j}_{\eps}$}\label{f:3}
 \end{figure}
 
  Clearly, by~\eqref{eq: good-t2X} and~\eqref{eq: vbad} we obtain
\begin{align}\label{eq balls}
 \sum_{i=1}^{n_\eps} \EEE \mathcal{H}^1(\partial R_\eps^i) \le  \sum_{i=1}^{n_\eps} \EEE 4\mathcal{H}^1(T^i_\eps)\le 4\mathcal{H}^1\big(J_{v_\eps}^{\rm bad}\big)\le 12\theta.
\end{align}   
 We define 
\begin{align}\label{eq: ww}
w_\eps  \coloneqq \EEE  v_\eps\chi_{\Omega' \setminus  \bigcup_{i=1}^{n_\eps}  R_\eps^i } +  \sum_{i=1}^{n_\eps} \EEE s^i_\eps \chi_{R_\eps^i}
\end{align}
for suitable constants $(s_\eps^i)_i \subset \R^2$ for which the functions  $y_\eps  \coloneqq \EEE \id + \eps w_\eps$  are such that \EEE  the sets
\begin{align}\label{eq: injec}
\big[y_\eps\big(\Omega' \setminus  \bigcup_{i=1}^{n_\eps} \EEE R_\eps^i\big) \big],  \quad [y_\eps(R_\eps^i) ], \ i=1,\ldots,n_\eps, \quad \text{are pairwise disjoint}.   
\end{align}
 Note that this is  possible since $v_\eps \in L^\infty(\Omega';\R^2)$. By construction and  by \eqref{eq: dense-in-appli-6} we see that the rectangles $(R_\eps^i)_i$ do not intersect $\Omega' \setminus \overline{\Omega}$.  As $v_\eps \in GSBV_2^2(\Omega';\R^2)$ and  $v_\eps = h$ on $\Omega' \setminus \overline{\Omega}$,   we get $y_\eps \in \mathcal{S}_{\eps,h}$, see~\eqref{eq: boundary-spaces}.

\noindent \emph{Step 2:  Ciarlet-Ne\v{c}as \EEE condition.} We now check that $y_\eps$ is injective. Clearly, $y_\eps$ is injective on each~$R_\eps^i$, $i=1,\ldots,n_\eps$. In view of \eqref{eq: injec}, it suffices to  check that $y_\eps$ is also injective on $\Omega' \setminus \bigcup\nolimits_{i=1}^{n_\eps} R_\eps^i$. To this end, fix arbitrary $x_1,x_2 \in \Omega' \setminus \bigcup\nolimits_{i=1}^{n_\eps} R_\eps^i$, $x_1 \neq x_2$, and recall that $ y_\eps  = \id +  \eps \EEE v_\eps$ on $\Omega' \setminus \bigcup\nolimits_{i=1}^{n_\eps} R_\eps^i$.  We distinguish between two cases according to the distance between $x_{1}$ and~$x_{2}$. \EEE

\noindent \emph{Case 1. $ | x_{1} - x_{2}| \ge \sqrt{\eps}$.\EEE}  By \eqref{eq: dense-in-appli-4} \EEE and  $\gamma >\frac{2}{3}$  we get \EEE
\begin{align*}
|y_\eps(x_1) - y_\eps(x_2)| \ge |x_1 - x_2| - 2\eps \Vert v_\eps \Vert_{L^\infty(\Omega')} \ge \sqrt{\eps} - 2\eps \eps^{\gamma-1} >0. 
\end{align*}
\noindent \emph{Case 2. $ | x_{1} - x_{2}| < \sqrt{\eps}$.\EEE}  Inequality~\eqref{eq: dense-in-appli-5} \EEE implies that the segment between $x_1$ and $x_2$, denoted by $[x_1;x_2]$, intersects at most one segment $S_\eps^i$.  We subdivide this case in two subcases, distinguishing between $[x_{1}; x_{2}] \cap \YYY S_{\eps}^{i}  \EEE = \emptyset$ and $[x_{1}; x_{2}] \cap \YYY S_{\eps}^{i} \EEE \neq \emptyset$. \EEE 

\noindent \emph{Case 2\,{\rm (i)}.}  If $[x_1;x_2]$ does not intersect one of the segments $S_\eps^i$, $v_\eps$ is Lipschitz in a neighborhood of $[x_1;x_2]$, and we get by  \eqref{eq: dense-in-appli-4} \EEE and  $\gamma >\frac{2}{3}$ that 
\begin{align*}
|y_\eps(x_1) - y_\eps(x_2)| \ge |x_1 - x_2| - \eps |x_1 -x_2| \Vert \nabla v_\eps \Vert_{L^\infty(\Omega')} \ge  |x_1 - x_2|  (1 - \eps\eps^{\gamma-1}) >0. 
\end{align*}
\noindent \emph{Case 2\,{\rm (ii)}.} Let us now suppose that $[x_1;x_2]$ intersects $S_\eps^i$.   By construction of $R_\eps^i$, we can find a piecewise affine curve $\Gamma \colon [0,l_\Gamma] \to \Omega' \setminus \bigcup\nolimits_{i=1}^{n_\eps} R_\eps^i$ with $\Gamma(0) = x_1$, $\Gamma(l_\Gamma) = x_2$, parametrized by arc-length,  such that
\begin{align}\label{eq: lengthi}
{\rm (a)} \ \  l_\Gamma = |x_1-x_2| \quad \quad  \text{or} \quad \quad {\rm (b)} \ \ l_\Gamma \le    |x_1 - x_2| + \mathcal{H}^1(\partial R_\eps^j),  
\end{align}
where case (b) holds if $[x_1;x_2]$ intersects some $T_\eps^j \subset J_{v_\eps}^{\rm bad}$. Moreover, we have that  $\Gamma(t) \in  S_\eps^i$ for at most one $t\in (0,l_\Gamma)$, where in this case we have $ \nu_{S_\eps^i}  \cdot \Gamma'(t) \ge 0$, where $\nu_{S_\eps^i}$ denotes the normal vector to  $S_\eps^i$  oriented such that $ \nu_{S_\eps^i}  \cdot (x_2-x_1) \ge 0$  (see Figure~\ref{f:4}). 

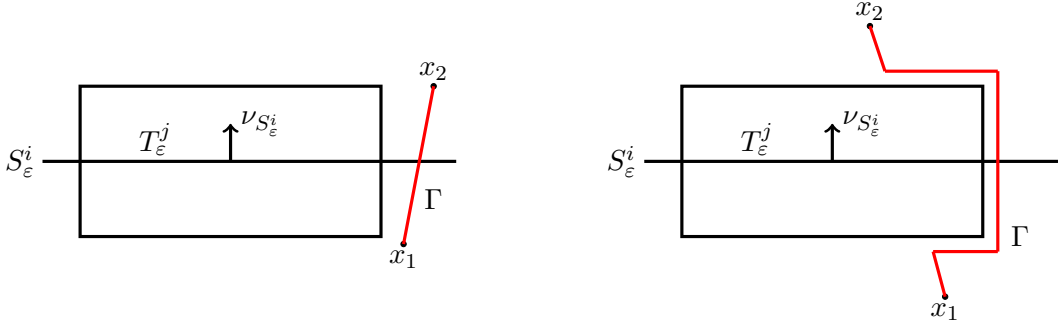
\begin{figure}[h!]
\begin{tikzpicture}
\draw[black, very thick] (0,0) rectangle (4,2);
\draw[black, very thick] (-0.5,1) -- (5, 1);
\node at (1, 1.3) {\large{$T^{j}_{\eps}$}};
\draw[black, very thick, ->] (2, 1) -- (2, 1.5);
\node at (2.4, 1.5) {\large{$\nu_{S^{i}_{\eps}}$}};
\node at (-0.8, 1) {\large{$S^{i}_{\eps}$}};
\node at (4.7, 2.2) {\large{$x_{2}$}};
\node at (4.3, -0.3) {\large{$x_{1}$}};
\draw [fill] (4.7,2) circle (1pt);
\draw [fill] (4.3, -0.1) circle (1pt);
\draw[red, very thick] (4.7, 2) -- (4.3, -0.1);
\node at (4.7, 0.5) {\large{$\Gamma$}};

\draw[black, very thick] (8, 0) rectangle (12,2);
\draw[black, very thick] (7.5,1) -- (13, 1);
\node at (9, 1.3) {\large{$T^{j}_{\eps}$}};
\draw[black, very thick, ->] (10, 1) -- (10, 1.5);
\node at (10.4, 1.5) {\large{$\nu_{S^{i}_{\eps}}$}};
\node at (7.2, 1) {\large{$S^{i}_{\eps}$}};
\node at (10.5,3) {\large{$x_{2}$}};
\node at (11.5, -1) {\large{$x_{1}$}};
\draw [fill] (10.5,2.8) circle (1pt);
\draw [fill] (11.5, -0.8) circle (1pt);
\draw[red, very thick] (10.5,2.8) -- (10.7,2.2);
\draw[red, very thick] (10.7, 2.2) -- (12.2, 2.2);
\draw[red, very thick] (12.2, 2.2) -- (12.2, -0.2);
\draw[red, very thick] (12.2, -0.2) -- (11.34, -0.2);
\draw[red, very thick] (11.34, -0.2) -- (11.5, -0.8); 
\node at (12.5, 0) {\large{$\Gamma$}};
\end{tikzpicture}
\caption{Visualization of the curve $\Gamma$ in the case $\mathrm{(a)}$ (left) and $\mathrm{(b)}$ (right) of~\eqref{eq: lengthi}. }
\label{f:4}
\end{figure} 

 If~{\rm (a)} of~\eqref{eq: lengthi} holds, then
\begin{equation}
\label{e:length-l}
{\rm (a)} \ \ l_{\Gamma} \leq \sqrt{\eps}\,.
\end{equation}
If~{\rm (b)} of~\eqref{eq: lengthi} holds, we further distinguish two cases, namely
\begin{equation}
\label{e:length-l-3}
\mathrm{(b_{1})} \ \ \mathcal{H}^{1}(T^{j}_{\eps}) \geq \sqrt{\eps} \qquad \text{or} \qquad \mathrm{(b_{2})} \ \ \mathcal{H}^{1}(T^{j}_{\eps}) < \sqrt{\eps}\,.
\end{equation}
If~$\mathrm{(b_{2})}$ holds, we immediately infer that
\begin{equation}
\label{e:length-l-2}
\mathrm{(b_{2})} \ \ l_{\Gamma} \leq 5\sqrt{\eps}\,. 
\end{equation}
In the case~$\mathrm{(b_{1})}$, instead, we get that $x_k + \R\nu_{S_\eps^i}$ intersects $T_\eps^j$ for some $k=1,2$, and since  $R^j_\eps$  has height $\min \lbrace \mathcal{H}^1(T^j_\eps), \sqrt{\eps}\rbrace  = \sqrt{\eps}$, we get that
\begin{align}\label{eq: projecti}
 \mathrm{(b_{1})}\EEE  \ \  (x_2-x_1)\cdot \nu_{S_\eps^i}   \ge   \sqrt{\eps}/2 \qquad \text{and} \qquad l_\Gamma \le    \sqrt{\eps}  + 4\mathcal{H}^1(T_\eps^j)  \le 5\mathcal{H}^1(T_\eps^j)\,. 
\end{align}

In the following, we will only treat  the cases $\mathrm{(b_{1})}$ and~$\mathrm{(b_{2})}$, \EEE as the argument for {\rm (a)} is easier. By the fundamental theorem of calculus we compute 
\begin{align*}
  y_\eps(x_2) -  y_\eps(x_1) = x_2 - x_1 + \int_0^t  \eps\nabla v_\eps(s) \cdot \Gamma'(s) \, {\rm ds} +   \eps [v_\eps](\Gamma(t))     + \int_t^{l_\Gamma}  \eps \nabla v_\eps(s) \cdot \Gamma'(s) \, {\rm ds}\, , 
\end{align*}
where $t$ is chosen uniquely such that $\Gamma(t) \in S_\eps^i$. Since $\Gamma(t) \in J_{v_\eps} \setminus J_{v_\eps}^{\rm bad}$, we get $[v_\eps] \cdot\nu_{\YYY S_\eps^i\EEE} \ge \tau_\eps \ge \tau/2$.   If $\mathrm{(b_{1})}$ of~\eqref{e:length-l-3} holds, by using  \eqref{eq: dense-in-appli-4}, property $\mathrm{(b_{1})}$ in~\eqref{eq: projecti}, and the arc-length parametrization of~$\Gamma$, giving $|\Gamma'| \equiv 1$, we find
\begin{align*}
  \big(y_\eps(x_2) -  y_\eps(x_1)\big) \cdot  \nu_{S_\eps^i} &\ge (x_2-x_1) \cdot \nu_{S_\eps^i} - l_\Gamma \eps^\gamma  + \eps [v_\eps](\Gamma(t)) \cdot \nu_{S_\eps^i}  \\&\ge  \frac{1}{2}    \sqrt{\eps}    - l_\Gamma \eps^\gamma +  \frac{\tau}{2} \eps
  \geq \frac{1}{2}  \sqrt{\eps}   +  \frac{\tau}{2} \eps  -  5  \eps^\gamma \mathcal{H}^{1}(T^{j}_{\eps})\,. 
\end{align*}
If~$\mathrm{(b_{2})}$ of~\eqref{e:length-l-3} holds, arguing in a similar way   and using~$\mathrm{(b_{2})}$ in~\eqref{e:length-l-2}  we obtain  
\begin{align*}
  \big(y_\eps(x_2) -  y_\eps(x_1)\big) \cdot  \nu_{S_\eps^i} &\ge (x_2-x_1) \cdot \nu_{S_\eps^i} - l_\Gamma \eps^\gamma  + \eps [v_\eps](\Gamma(t)) \cdot \nu_{S_\eps^i}   \ge    0  -        5 \eps^{(\gamma + \frac{1}{2})}  +  \frac{\tau}{2} \eps. 
\end{align*}
In both cases, since~$\gamma >\frac{2}{3}$ we find that $(y_{\eps}(x_{2}) - y_{\eps}(x_{1})) \cdot \nu_{S_{\eps}^{i}} >0$ for $\eps$ sufficiently small, depending only on~$\mathcal{H}^{1}(J_{u})$ and $\tau$. This shows $y_\eps(x_1) \neq y_\eps(x_2)$ and yields that $y_\eps$ is injective.  

By  \eqref{eq: dense-in-appli-4} \EEE and \eqref{eq: ww}  we further get ${\rm det}(\nabla y_\eps)>0$ for a.e.\ $x \in \Omega'$, provided that $\eps$ is sufficiently small. Therefore, $y_\eps$ satisfies the   Ciarlet-Ne\v{c}as \EEE non-interpenetration  condition.

\noindent \emph{Step 3: Convergence of functions and energies.}   We now check that  $y_\eps \rightsquigarrow u$ in the sense of Definition~\ref{def:conv}. We define $y_\eps^{\rm  rot} = y_\eps$, i.e., the Caccioppoli partition in \eqref{eq: modifica} consists of the set~$\Omega'$ only with corresponding rotation $\Id$. As  $\nabla y_\eps^{\rm rot} - \Id = \eps \nabla v_\eps\chi_{\Omega' \setminus \bigcup\nolimits_{i=1}^{n_\eps} R_\eps^i }$,  \eqref{eq: first conditions-2}--\eqref{eq: first conditions-3.5} follow from  \eqref{eq: dense-in-appli-2} and~\eqref{eq: dense-in-appli-4}. The rescaled displacement fields $u_\eps$ defined in~\eqref{eq: modifica} satisfy $u_\eps = v_\eps\chi_{\Omega' \setminus \bigcup\nolimits_{i=1}^{n_\eps} R_\eps^i }$. Then, \eqref{eq: the main convergence-1}--\eqref{eq: the main convergence-4}  for $E_u = \emptyset$ follows from  \eqref{eq: dense-in-appli-1}--\eqref{eq: dense-in-appli-2}, the lower semicontinuity result in~\cite[Theorem 11.3]{DalMaso:13}, and the fact that
$$ \sum_{i=1}^{n_\eps} \mathcal{L}^2(R_\eps^i) \le \sqrt{\eps}  \sum_{i=1}^{n_\eps} \EEE \mathcal{H}^1(T_\eps^i) \le 3\sqrt{\eps}\theta, $$
where in the last step we used \eqref{eq balls}.

Finally, we confirm  \eqref{eq: for recov}. Since $J_{\nabla y_\eps} \subset J_{y_\eps}$ and  \eqref{eq: dense-in-appli-3}, \eqref{eq balls}, and \eqref{eq: ww} hold, we get \EEE 
\begin{align*}
\limsup_{\eps \to 0} \mathcal{H}^1(J_{y_\eps}) \le  \limsup_{\eps \to 0} \mathcal{H}^1(J_{v_\eps}) + \limsup_{\eps \to 0}  \sum_{i=1}^{n_\eps} \EEE \mathcal{H}^1(\partial R_\eps^i)  \le \mathcal{H}^1(J_u) + 12\theta.
\end{align*}
Consequently, by the definition of the energies in \eqref{rig-eq: Griffith en} and   \eqref{rig-eq: Griffith en-lim}, it suffices to show 
\begin{equation}
\label{eq: final}
\lim_{\eps \to 0} \Big( \frac{1}{\eps^2}\int_{\Omega'} W(\nabla y_\eps)  \,{\rm d}x   + \frac{1}{\eps^{2\beta}} \int_{\Omega'} |\nabla^2 y_\eps|^2  \,{\rm d}x \Big) = \int_{\Omega'} \frac{1}{2} Q(e(u))  \,{\rm d}x  \,.
\end{equation}
The second term  in~\eqref{eq: final} vanishes by \eqref{eq: dense-in-appli-4},  $\beta <1$, and the fact that $\nabla^2 y_\eps = \eps \nabla^2 v_\eps$. For the first term  in~\eqref{eq: final},  we use that $W(\Id + F) =  \frac{1}{2}Q( {\rm sym}(F)) + \omega(F)$ with $|\omega(F)|\le C|F|^3$ for $|F| \le 1$, and  compute by  \eqref{eq: dense-in-appli-2} and \eqref{eq: dense-in-appli-4} \EEE
\begin{align*} 
\lim_{\eps \to 0}\frac{1}{\eps^2} \int_{\Omega'} W(\nabla y_\eps)  \,{\rm d}x  &  \leq  \lim_{\eps \to 0} \frac{1}{\eps^2} \int_{\Omega'} W(\Id +  \eps  \nabla v_\eps) \,{\rm d}x   = \lim_{\eps \to 0}  \int_{\Omega'}  \Big( \frac{1}{2} Q(e(v_\eps))  \,{\rm d}x  + \frac{1}{\eps^2} \omega(\eps \nabla v_\eps) \,{\rm d}x   \Big)  \\ & = \int_{\Omega'}   \frac{1}{2} Q(e(u)) \,{\rm d}x  +   \lim_{\eps \to 0}\int_{\Omega'}    {\rm O}\big( \eps|\nabla v_\eps|^3 \big) \,{\rm d}x  = \int_{\Omega'}   \frac{1}{2} Q(e(u)) \,{\rm d}x  \,,
\end{align*}
 where in the last step we have used that $\Vert \nabla v_\eps \Vert_{L^\infty(\Omega')} \le C\eps^{\gamma - 1}$ for some $\gamma >2/3$. This concludes the proof  of~\eqref{eq: final} and of the theorem. \EEE  
\end{proof}


\section*{Acknowledgements} 
This work was supported by the DFG project FR 4083/1-1 and by the Deutsche Forschungsgemeinschaft (DFG, German Research Foundation) under Germany's Excellence Strategy EXC 2044 -390685587, Mathematics M\"unster: Dynamics--Geometry--Structure.   This research was additionally supported by the Austrian Science Fund (FWF) projects F65, V 662, Y1292, and I 4052, as well as from the OeAD-WTZ projects CZ04/2019 and CZ 01/2021. \EEE


 \typeout{References}

\end{document}